%% file: ArticleGIGO.tex
\pdfoutput=1
\documentclass[10pt,a4paper]{article}
\usepackage[utf8]{inputenc}
\usepackage[english]{babel}
\usepackage{amsmath}
\usepackage{amsfonts}
\usepackage{amssymb}
\usepackage{amsthm}
\usepackage{graphicx}
\usepackage[nottoc, notlof, notlot]{tocbibind}
\usepackage{algorithm}
\usepackage{algorithmic}
\usepackage{color}
\usepackage{sidecap}
\usepackage{url}
\author{Jérémy Bensadon}
\title{Black-box optimization using geodesics in statistical manifolds}
\date{}
\newcommand{\R}{\mathbb{R}}
\newcommand{\Rd}{\mathbb{R}^{d}}
\newcommand{\C}{\mathbb{C}}
\newcommand{\N}{\mathbb{N}}
\newcommand{\G}{\mathbb{G}}
\newcommand{\dep[3]}{ \frac{\partial #2}{\partial #3}  }

\newtheorem{theo}{Theorem}
\newtheorem{prop}[theo]{Proposition}

\newtheorem{defi}[theo]{Definition}

\newtheorem{nota}[theo]{Notation}

\newtheorem{cor}[theo]{Corollary}

\DeclareMathOperator{\tr}{tr}
\DeclareMathOperator{\rk}{rk}
\newcommand{\trt}{{}^t \! }
\newcommand{\Gd}{\mathbb{G}_{d}}
\newcommand{\Rn}{\mathbb{R}^{n}}
\newcommand{\tGd}{\tilde{\mathbb{G}}_{d}}
\newcommand{\dt}{\delta \hspace{-0.07mm} t}

\def\d{{\operatorname{d}}}
\renewcommand{\geq}{\geqslant}
\renewcommand{\leq}{\leqslant}
\renewcommand{\Re}{\mathrm{Re}}
\renewcommand{\Im}{\mathrm{Im}}
\newcommand{\trsp[2]}{#2^{T}}
\newcommand{\deff[6]}{\begin{array}{ccccc}
#2 &: & #3 & \to & #4 \\
 & & #5 & \mapsto & #6 \\
\end{array}}
\begin{document}


\maketitle

\begin{abstract}

\emph{Information geometric optimization} (IGO, \cite{IGO}) is a general framework for stochastic optimization problems aiming at limiting the influence of arbitrary parametrization choices. The initial problem is transformed into the optimization of a smooth function on a Riemannian manifold (following \cite{Ama}), 
defining a parametrization-invariant first order differential equation. However, in practice, it is necessary to discretize time, and then, parametrization invariance holds only at first order in the step size.

We define the \emph{Geodesic IGO} update (GIGO), which uses the Riemannian manifold structure to obtain an update entirely independent from the parametrization of the manifold. We test it with classical objective functions.

Thanks to Noether's theorem from classical mechanics, we find an efficient way to write a first order differential equation satisfied by the geodesics of the statistical manifold of Gaussian distributions, and thus to compute the corresponding GIGO update. We then compare GIGO, pure rank-$\mu$ CMA-ES \cite{CMATuto} and xNES \cite{xNES} (two previous algorithms that can be recovered by the IGO framework), and show that while the GIGO and xNES updates coincide when the mean is fixed, they are different in general, contrary to previous intuition. We then define a new algorithm (Blockwise GIGO) that recovers the xNES update from abstract principles. 

\end{abstract}
\clearpage
\tableofcontents

\section*{Introduction}
\addcontentsline{toc}{section}{Introduction}

Consider an objective function $ f \colon X \to \R$ to be minimized. We suppose we have absolutely no knowledge about $f$: the only thing we can do is ask for its value at any point $x \in X$ (\textit{black-box} optimization), and that the evaluation of $f$ is a costly operation. We are going to study algorithms that can be described in the IGO framework (see \cite{IGO}), 

A possible way to optimize such a function is the following:

We choose $(P_{\theta} )_{\theta \in \Theta }$ a family of probability distributions\footnote{This is not the definition we want formally for $\Theta$: the theoretical context is the context of differentiable manifolds. The importance of working in a manifold appears when changing the parametrization or when using a local parametrization. See \cite{Ama}.} on $X$, and an initial probability distribution $P_{\theta^{0}}$. Now, we replace $f$ by $F \colon \Theta \to \R$ (for example $F(\theta ) = E_{x \sim P_{\theta }} [f(x)]$), and we optimize $F$ with a gradient descent:

\begin{equation} \frac{\d{\theta^{t}}}{\d{t}} = - \nabla_{\theta}  E_{x \sim P_{\theta }} [f(x)]. \label{Flot0} \end{equation}

However, because of the gradient, this equation depends entirely on the parametrization we chose for $\Theta$, which is disturbing: we do not want to have two different updates because we chose different numbers to represent the objects we are working with. That is why invariance is a \emph{design principle} behind IGO. More precisely, we want invariance with respect to monotone transformations of $f$, and invariance under reparametrization of $\theta$.

IGO provides a differential equation on $\theta$ with the desired properties, but because of the discretization of time needed to obtain an explicit algorithm, we lose invariance under reparametrization of $\theta$: two IGO algorithms for the same problem, but with different parametrizations, coincide only at first order in the step size. A possible solution to this problem is Geodesic IGO (GIGO), introduced here\footnote{See also IGO-ML, in \cite{IGO}, for example.}: the initial direction of the update at each step of the algorithm remains the same as in IGO, but instead of moving straight for the chosen parametrization, we use the structure of Riemannian manifold of our family of probability distributions (see \cite{Ama}) by following its geodesics.

Finding the geodesics of a Riemannian manifold is not always easy, but Noether's theorem will allow us to obtain quantities that are preserved along the geodesics. In the case of Gaussian distributions, it is possible to find enough invariants to obtain a first order differential equation satisfied by the geodesics, which makes their computation easier. \\

Although the geodesic IGO algorithm is not strictly speaking parametrization invariant when no closed form for the geodesics is known, it is possible to compute them at arbitrary precision without increasing the numbers of objective function calls.\\

The first two sections are preliminaries:
in Section \ref{SecIGO}, we recall the IGO algorithm, introduced in \cite{IGO}, and in Section \ref{SecRie}, after a reminder about Riemannian geometry, we state Noether's theorem, which will be our main tool to compute the GIGO update for Gaussian distributions.

In Section \ref{SecGIGOt}, we consider Gaussian distributions with covariance matrix proportional to the identity matrix: this space is isometric to the hyperbolic space, and the geodesics of the latter are known.

In Section \ref{SecGIGO}, we consider the general Gaussian case, and we use Noether's theorem to obtain two different sets of equations to compute the GIGO update. The equations were already known, see \cite{Eriksen}, \cite{SolGeo} and \cite{GeoRemarks}, but the connection with Noether's theorem has not been mentioned. We then give the explicit solution for these equations, from \cite{SolGeo}.

In Section \ref{SecComp}, we recall quickly the xNES and CMA updates and we introduce a slight modification of the IGO algorithm to incorporate the direction-dependent learning rates used in CMA-ES and xNES. We then compare these different algorithms, prove that xNES is not GIGO in general and we finally introduce a new family of algorithms extending GIGO and recovering xNES from abstract principles.

Finally, Section \ref{SecNum} presents numerical experiments, which suggest that when using GIGO with Gaussian distributions the step size must be chosen carefully.\\

To simplify notation, we will use $\Rn$ instead of a general manifold as much as possible, but the latter is the ``right" formal context.\\

\textbf{Acknowledgement.} I would like to thank Yann Ollivier for his numerous remarks about this article, and Frédéric Barbaresco for finding the reference \cite{SolGeo}.

%
%
%

\section{Definitions: IGO, GIGO}
\label{SecIGO}

In this section, we recall what the IGO framework is, and we define the geodesic IGO update. Consider again Equation \ref{Flot0}: 
$$ \frac{\d{\theta^{t}}}{\d{t}} = - \nabla_{\theta}  E_{x \sim P_{\theta }} [f(x)].$$
Its main problems are that
\begin{itemize}

\item The gradient depends on the parametrization of our space of probability distributions (see \ref{SubSecGIGO} for an example).

\item The equation is not invariant under monotone transformations of $f$. For example, the optimization for $10f$ moves ten times faster than the optimization for $f$. 
\end{itemize}

In this section, we recall how IGO deals with this (see \cite{IGO} for a better presentation).

\subsection{Invariance under reparametrization of $\theta $: Fisher metric}

In order to achieve invariance under reparametrization of $\theta $, it is possible to turn our family of probability distributions into a Riemannian manifold (it is the main topic of \textit{information geometry}, see \cite{Ama}), and therefore, there is a canonical gradient (called the \textit{natural gradient}). 

\begin{defi}
Let $P,Q$ be two probability distributions on $X$. The Kullback--Leibler divergence of $Q$ from $P$ is defined by:
\begin{equation} \mathrm{KL} (Q \Vert P  ) = \int_{X} \ln ( \frac{Q (x)}{P (x) } ) \d{Q }(x)  .\end{equation}
\end{defi}

By definition, it does not depend on the parametrization. It is not symmetrical, but if for all $x$, the application $ \theta \mapsto P_{\theta} (x)$ is $C^{2}$, then a second-order expansion yields:

\begin{equation} \mathrm{KL} ( P_{\theta + \d{\theta} } \Vert P_{\theta } ) = \frac{1}{2} \sum_{i,j} I_{ij}(\theta) \d{\theta_{i}} \d{\theta_{j}} + o(\d{\theta ^{2}}), \end{equation}
where
\begin{equation} I_{ij}(\theta) = \int_{X} \dep{\ln P_{\theta}(x)}{\theta_{i}}\dep{\ln P_{\theta}(x)}{\theta_{j}} \d{P_{\theta} ( x)}  =-\int_{X} \frac{\partial^{2} \ln P_{\theta}(x)}{\partial \theta_{i} \partial \theta_{j}} \d{P_{\theta} (x) }. 
\end{equation}
This is enough to endow the family $(P_{\theta})_{\theta \in \Theta}$ with a Riemannian manifold structure\footnote{A Riemannian manifold $M$ is a differentiable manifold (which can be seen as pieces of $\Rn$ glued together), with a \emph{metric}. The metric at $x$ is a symmetric positive definite quadratic form on the tagent space of $M$ at $x$: it indicates how expensive it is to move in a given direction on the manifold. We will think of the updates of the algorithms we will be studying as paths on $M$.}.
The matrix $I(\theta)$ is called ``Fisher information matrix", the metric it defines is called the ``Fisher metric".

Given a metric, it is possible to define a gradient attached to this metric: the key property of the gradient is that for any smooth function $f$
\begin{equation}\label{GradProp}
f(x+h)=f(x) + \sum_{i} h_{i} \dep{f}{x_{i}}+o(\Vert h \Vert ) =f(x)+\langle h, \nabla f (x) \rangle +o(\Vert h \Vert ),
\end{equation}
where $\langle x ,y \rangle = \trsp{x} I y$ is the dot product in metric $I$. Therefore, in order to keep the property of Equation \ref{GradProp}, we must have $\nabla f=I^{-1} \dep{ f}{x}$.

We have therefore the following gradient (called ``natural gradient", see \cite{Ama}): \begin{equation} \tilde{\nabla}_{\theta} = I^{-1} (\theta ) \dep{}{\theta}, \end{equation}
and since the Kullback--Leibler divergence does not depend on the parametrization, neither does the natural gradient.

Later in this paper, we will study families of Gaussian distributions. The following proposition gives the Fisher metric for these families.

\begin{prop}
\label{FishG}
Let $(P_{\theta} )_{\theta \in \Theta}$ be a family of normal probability distributions: $P_{\theta } = \mathcal{N} (\mu (\theta ), \Sigma (\theta ))$. If $\mu$ and $\Sigma$ are $C^{1}$, the Fisher metric is given by:
\begin{equation} I_{i,j} (\theta )= \dep{\trsp{ \mu } }{\theta_{i}} \Sigma^{-1} \dep{\mu }{\theta_{j}} + \frac{1}{2} \tr \left( \Sigma^{-1} \dep{ \Sigma}{\theta_{i}} \Sigma^{-1}  \dep{\Sigma}{\theta_{j}} \right). \label{FisherG} \end{equation}
\end{prop}

\textit{Proof.}
This is a non-trivial calculation. See \cite{RefFisher} for more details.

As we will often be working with Gaussian distributions, we introduce the following notation:

\begin{nota}
$\Gd  $ is the manifold of Gaussian distributions in dimension $d$, equipped with the Fisher metric.

$\tGd  $ is the manifold of Gaussian distributions in dimension $d$, with covariance matrix proportional to identity in the canonical basis of $\Rd$, equipped with the Fisher metric.
\end{nota}

\subsection{ IGO flow, IGO algorithm}

In IGO \cite{IGO}, invariance with respect to monotone transformations is achieved by replacing $f$ by the following transform: we set 
\begin{equation} q (x) = P_{x' \sim P_{\theta} } (f(x') \leq f(x) ),\end{equation}
a non-increasing function $w \colon [0;1] \to \R$ is chosen (the \textit{selection scheme}), and finally $W^{f}_{\theta} (x) =w(q (x ) ). \footnote{This definition has to be slightly changed if the probability of a tie is not zero. See \cite{IGO} for more details. } $ By performing a gradient descent on $E_{x \sim P_{\theta }} [W_{f}^{\theta^{t}} (x)]$, we obtain the \emph{``IGO flow"}:

\begin{equation} \frac{\d{\theta^{t}}}{\d{t}} = \tilde{\nabla}_{\theta} \int_{X} W^{f}_{\theta^{t}} (x)P_{\theta} ( \d{x} ) = \int_{X} W^{f}_{\theta^{t}} (x)\tilde{\nabla}_{\theta} \ln P_{\theta} (x) P_{\theta^{t} }( \d{x} ). \label{FlotIGO} \end{equation}

For practical implementation, the integral in \eqref{FlotIGO} has to be approximated. For the integral itself, the Monte-Carlo method is used: $N$ values  $(x_{1},...,x_{N} )$ are sampled from the distribution $P_{\theta^{t}}$, and the integral becomes
\begin{equation} \frac{1}{N}\sum_{i=1}^{N} W^{f}_{\theta^{t}} (x_{i}) \tilde{\nabla}_{\theta} \ln P_{\theta} (x_{i}) \end{equation} and we approximate $\frac{1}{N}W_{\theta}^{f} (x_{i} ) = \frac{1}{N}w(q (x_{i} ) )$ by $\hat{w}_{i} = \frac{1}{N}w (\frac{\rk (x_{i} ) + 1/2}{N} ) $, where  $\rk (x_{i} )= \vert \{ j, f(x_{j}) < f(x_{i}) \} \vert$.\footnote{ It can be proved (see \cite{IGO}) that $ \lim_{N \rightarrow \infty} N\hat{w}_{i} = W_{f}^{\theta^{t}} (x_{i}) $. Here again, we are assuming there are no ties.}

We now have an algorithm that can be used in practice if the Fisher information matrix is known.

\begin{defi}
\label{IGO}
The \emph{IGO update} associated with parametrisation $\theta$, sample size $N$, step size $\dt$ and selection scheme $w$ 
\footnote{One could start directly with the $\hat{w}_{i}$ rather than $w$, as we will do later.}
 is given by the following update rule:

\begin{equation} \theta^{t+\dt} = \theta^{t} + \dt I^{-1} (\theta^{t} ) \sum_{i=1}^{N} \hat{w}_{i} \dep{\ln P_{\theta } (x_{i} )}{\theta} .\end{equation}

We call \emph{IGO speed} the vector $ I^{-1} (\theta^{t} ) \sum_{i=1}^{N} \hat{w}_{i} \dep{\ln P_{\theta } (x_{i} )}{\theta}$.
\end{defi}

\subsection{Geodesic IGO}
\label{SubSecGIGO}
Although the IGO flow associated with a family of probability distributions is intrinsic (it only depends on the family itself, not the parametrization we choose for it), the IGO update is not. However, two different IGO updates are ``close" when $\dt \rightarrow 0$: the difference between two steps of IGO that differ only by the parametrization is a $O(\dt^{2})$.  
%
%
%

Intuitively, the reason for this difference is that two IGO algorithms start at the same point, and follow ``straight lines" with the same initial speed, but the definition of ``straight lines" changes with the parametrization.

For instance, in the case of Gaussian distributions, let us consider two different IGO updates with Gaussian distributions in dimension $1$, the first one with parametrization $(\mu , \sigma )$, and the second one with parametrization $(\mu , c:=\sigma^{2} )$. We suppose the IGO speed for the first algorithm is $(\dot{\mu} , \dot{ \sigma} )$. The corresponding IGO speed in the second parametrization is given by the identity $\dot{c}=2\sigma \dot{\sigma}$. Therefore, the first algorithm gives the standard deviation $\sigma_{\mathrm{new, 1}} = \sigma_{\mathrm{old}}+ \dt \dot{\sigma}$, and the variance $c_{\mathrm{new, 1}} =(\sigma_{\mathrm{new, 1}})^{2} = c_{\mathrm{old}} + 2 \dt  \sigma_{\mathrm{old}}  \dot{\sigma} + \dt^{2} \dot{\sigma}^{2} =c_{\mathrm{new, 2}} + \dt^{2} \dot{\sigma}^{2} $.

The geodesics of a Riemannian manifold are the generalization of the notion of straight line: they are curves that locally minimize length\footnote{In particular, given two points $a$ and $b$ on the Riemannian manifold $M$, the shortest path from $a$ to $b$ is always a geodesic. The converse is not true, though.}. The notion will be explained precisely in Section \ref{SecRie}, but let us define the geodesic IGO algorithm, which follows the geodesics of the manifold instead of following the straight lines for an arbitrary parametrization.

\begin{defi}[GIGO]
\label{DefGIGO}
The \emph{geodesic IGO} update (GIGO) associated with sample size $N$, step size $\dt$ and selection scheme $w$  is given by the following update rule:

\begin{equation} \theta^{t+\dt} = \exp_{\theta^{t}} (Y\dt) \end{equation}
where
\begin{equation} 
\label{IGOSpeed}
Y= I^{-1} (\theta^{t} ) \sum_{i=1}^{N} \hat{w}_{i} \dep{\ln P_{\theta } (x_{i} )}{\theta}, \end{equation}
is the IGO speed and $\exp_{\theta^{t}}$ is the exponential of the Riemannian manifold $\Theta$. Namely, $\exp_{\theta^{t}} (Y\dt)$ is the endpoint of the geodesic of $\Theta$ starting at $\theta^{t}$, with initial speed $Y$, after a time $\dt$.
By definition, this update does not depend on the parametrization $\theta$.
\end{defi}

Notice that while the GIGO update is compatible with the IGO flow (in the sense that when $\dt \rightarrow 0$ and $N \rightarrow \infty$, a parameter $\theta^{t}$ updated according to the GIGO algorithm is a solution of Equation \ref{FlotIGO}, the equation defining the IGO flow), it not necessarily an IGO update. More precisely, the GIGO update is an IGO update if and only the geodesics of $\Theta$ are straight lines for some parametrization\footnote{By Beltrami's theorem, this is equivalent to $\Theta$ having constant curvature.}.

The main problem with this update is that in general, obtaining equations for the geodesics is a difficult problem. In the next section, we will state Noether's Theorem, which will be our main tool to compute the GIGO update for Gaussian distributions.

\section{Riemannian geometry, Noether's Theorem}
\label{SecRie}

\subsection{Riemannian geometry}
The goal of this section is to state Noether's theorem. We will not prove anything here, see \cite{Arn} for the proofs, \cite{Bou} or \cite{Jost} for a more detailed presentation. 
Noether's theorem states that if a system has symmetries, then, there are invariants attached to this symmetries. Firstly, we need some definitions.

\begin{defi}[Motion in a Lagrangian system]
Let $M$ be a differentiable manifold, $TM$ the set of tangent vectors on $M$\footnote{A tangent vector is identified by the point at which it is tangent, and a vector in the tangent space.}
, and $\deff{\mathcal{L}}{TM}{\R}{(q,v)}{\mathcal{L} (q,v)} $ a differentiable function (called the Lagrangian function\footnote{In general, it could depend on $t$.}). 
A ``motion in the lagrangian system $(M, \mathcal{L})$ from $x$ to $y$" is map $\gamma \colon [t_{0} , t_{1} ] \to M$ such that:
\begin{itemize} 

\item $\gamma ( t_{0} ) =x$
\item $\gamma (t_{1} )=y$ 
\item $\gamma $ is a local extremum of the functional:
\begin{equation} \Phi (\gamma ) = \int_{t_{0}}^{t_{1}} \mathcal{L}(\gamma (t), \dot{\gamma} (t) ) \d{t} ,\end{equation}
among all curves $c \colon [t_{0} , t_{1} ] \to M$ such that $c(t_{0} ) = x$, and $c(t_{1} ) = y$.
\end{itemize}

\end{defi}

For example, when $(M,g)$ is a Riemannian manifold, the length of a curve $\gamma $ between $\gamma ( t_{0} )$ and $\gamma ( t_{1} )$ is:
\begin{equation} \label{eqlength} \int_{t_{0}}^{t_{1}} \sqrt{(g(\dot{\gamma} (t),\dot{\gamma} (t)  ) )}\d{t}. \end{equation}

The curves that follow the shortest path between two points $x, y \in M$ are therefore the minima $\gamma $ of the functional ($\ref{eqlength}$) such that $\gamma (t_{0} ) = x$ and $\gamma (t_{1} ) =y$, and the corresponding Lagrangian function is $(q,v)\mapsto \sqrt{g(v,v)}$. The solution to the problem of finding a parametrized curve with a shortest length is not unique: any curve following the shortest trajectory will have minimum length. For example, if $\gamma_{1} : [a,b] \to M$ is a curve of shortest path, so is $\gamma_{2} : t\mapsto \gamma_{1} (t^{2} )$: these two curves define the same trajectory in $M$, but they do not travel along this trajectory at the same speed. This leads us to the following definition:

\begin{defi}[Geodesics]
Let $I$ be an interval of $ \R$, $(M,g)$ be a Riemannian manifold. A curve $\gamma: I \to M$ is called a \textit{geodesic} if for all $t_{0}, t_{1} \in I$, $\gamma \vert_{[t_{0}, t_{1}]}$ is a motion in the Lagrangian system $(M,\mathcal{L})$ from $\gamma (t_{0} )$ to $\gamma (t_{1} )$, where

\begin{equation} \mathcal{L}(\gamma )=\int_{t_{0}}^{t_{1}} (g(\dot{\gamma} (t),\dot{\gamma} (t)  ) )\d{t}. \end{equation}
\end{defi}

It can be shown (see \cite{Bou}) that geodesics are curves that \emph{locally} minimize length, with constant velocity\footnote{In the sense that $\frac{dg(\dot{\gamma} (t) ,(\dot{\gamma} (t)  )}{dt}=0$.}, which solves the previous uniqueness problem. More precisely, given a starting point and a starting speed, the geodesic is unique. This motivates the definition of the exponential of a Riemannian manifold.

\begin{defi}
Let $(M,g)$ be a Riemannian manifold. We call \textit{exponential} of $M$ the application:
$$\begin{array}{ccccc}
\exp &: & TM & \to & M \\
& & (x,v) & \mapsto & \exp_{x} (v), \\
\end{array}$$
such that for any $x\in M$, if $\gamma$ is the geodesic of $M$ satisfying $\gamma (0)=x$ and $\gamma ' (0)=v$, then $\exp_{x} (v)=\gamma (1)$.
\end{defi}

%
%
%
%
%

In order to find an extremal of a functional, the most commonly used result is called the ``Euler--Lagrange equations" (see \cite{Arn} for example). Using them, it is possible to show that the geodesics of a Riemannian manifold follow the ``geodesic equations":

\begin{equation} \label{geoEq}\ddot{x}^{k} + \Gamma_{ij}^{k} \dot{x}^{i} \dot{x}^{j} = 0, \end{equation}
where the \begin{equation}\Gamma_{ij}^{k} = \frac{1}{2} g^{lk} \left(\dep{g_{jl}}{q_{i}} + \dep{g_{li}}{q_{j}}  -\dep{g_{ij}}{q_{l}} \right)  \end{equation}
are called ``Christoffel symbols" of the metric $g$. However, these coefficients are tedious (and sometimes difficult) to compute, and (\ref{geoEq}) is a second order differential equation. Noether's theorem will give us a first order equation to compute the geodesics.


\subsection{Noether's Theorem}

\begin{defi}
\label{AdmMap}
Let $h \colon M \to M$ a diffeomorphism. We say that the Lagrangian system $(M,\mathcal{L})$ \emph{admits the symmetry $h$} if for any $(q,v)\in TM$, 
\begin{equation} \mathcal{L} \left( h(q),\d{h}(v) \right) = \mathcal{L}(q,v),\end{equation}
where $\d{h}$ is the differential of $h$.  

If $M$ is obvious, we will sometimes say that $\mathcal{L}$ is \emph{invariant under $h$}.
\end{defi}

An example will be given in the proof of Theorem \ref{INoe}.

We can now state Noether's theorem (see for example \cite{Arn}).
\begin{theo}[Noether's Theorem]
\label{NoeTh} 
If the Lagrangian system $(M,\mathcal{L} )$ admits the one-parameter group of symmetries $h^{s} \colon M \to M$, $s \in \R$, then the following quantity remains constant during the motions in the system $(M,\mathcal{L} )$. Namely, 
\begin{equation} I(\gamma(t), \dot{\gamma} (t) ) =  \dep{\mathcal{L}}{v} \left( \frac{\d{h}^{s} (\gamma (t))}{\d{s}} \vert_{s=0} \right) \end{equation}
does not depend on $t$ if $\gamma$ is a motion in $(M,\mathcal{L})$.
\end{theo}

Now, we are going to apply this theorem to our problem: computing the geodesics of Riemannian manifolds of Gaussian distributions.

\section{GIGO in $\tGd $}
\label{SecGIGOt}

We do not know of a closed form for the geodesics of $\Gd$. However, the situation is better if we force the covariance matrix to be either diagonal or proportional to the identity matrix. In the former case, the manifold we are considering is $(\mathbb{G}_{1})^{d}$, and in the latter case, it is $\tGd$. The geodesics of $(\mathbb{G}_{1})^{d}$ are given by

\begin{prop}
Let $M$ be a Riemannian manifold, let $d\in \N$, let $\Phi$ be the Riemannian exponential of $M^{d}$, and let $\phi$ be the Riemannian exponential of $M$. We have:

\begin{equation} \Phi_{(x_{1},...,x_{n})}((v_{1},...,v_{n}))=(\phi_{x_{1}} (v_{1}),...,\phi_{x_{n}} (v_{n})) \end{equation}

In particular, knowing the geodesics of $\mathbb{G}_{1}$ is enough to compute the geodesics of $(\mathbb{G}_{1})^{d}$.
\end{prop}

This is true because a block of the product metric does not depend on variables of the other blocks.

Consequently, a GIGO update with a diagonal covariance matrix with the sample $(x_{i})$ is equivalent to $d$ separate $1$-dimensional GIGO updates using the same samples. Moreover, $\G_{1} \cong \tilde{\G}_{1}$, the geodesics of which are given below.

We will show that  $\tGd$ and the ``hyperbolic space", of which the geodesics are known, are isometric. 

\subsection{Preliminaries: Poincaré half-plane, hyperbolic space}
\label{PrelimHspace}
In dimension $2$, the hyperbolic space is called ``hyperbolic plane", or Poincaré half-plane. We recall its definition:

\begin{defi}[Poincaré half-plane]
We call ``Poincaré half-plane" the Riemannian manifold $$\mathcal{H} = \{ (x,y) \in \R^{2} ,\, y >0 \},$$
with the metric $ds^{2} = \frac{dx^{2}+dy^{2}}{y^{2}}$.
\end{defi}

We also recall the expression of its geodesics (see for example \cite{GHL}):

\begin{prop}[Geodesics of the Poincaré half-plane]

The geodesics of the Poincaré half-plane are exactly the:
\label{poinca}
$$ t \mapsto \left(\Re(z(t) ), \Im (z(t) )\right),$$
where 
\begin{equation} z(t)=\frac{aie^{vt}+b}{cie^{vt}+d},\end{equation}
with $ad-bc=1$, and $v >0$.
\end{prop}

The geodesics are half-circles perpendicular to the line $y=0$, and vertical lines.

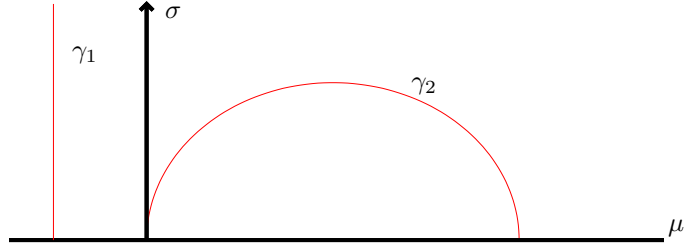
\begin{figure}[h]

\begin{center}

\input{GrGtex.tex}
\end{center}
\caption{Geodesics of the Poincaré half-plane}
\label{FigurePKR}
\end{figure}
The generalization to higher dimension is the following:

\begin{defi}[Hyperbolic space]
We call ``hyperbolic space of dimension $n$" the Riemannian manifold $$\mathcal{H}_{n} = \{ (x_{1}, ...,x_{n-1},y) \in \R^{n}, y >0 \},$$
with the metric $ds^{2} = \frac{dx_{1}^{2}+...+dx_{n-1}^{2}+dy^{2}}{y^{2}}$.
\end{defi}

Its geodesics stay in a plane containing the direction $y$ and the initial speed.\footnote{A possible way is to prove this is to use Noether's theorem: the Lagrangian for the geodesics is invariant under all translations along the $x_{i}$.} The induced metric on this plane is the metric of the Poincaré half-plane. The geodesics are therefore given by the following proposition:

\begin{prop}[Geodesics of the hyperbolic space]
If $\gamma \colon t \mapsto (x_{1}(t),...,x_{n-1}(t),y(t))=(\textbf{x} (t), y(t))$ is a geodesic of $\mathcal{H}_{n}$, then, there exists $a,b,c,d \in \R$ such that $ad-bc=1$ and $v > 0$ such that 

$\textbf{x} (t) = \textbf{x} (0) +   \frac{ \dot{\textbf{x}}_{0}}{\Vert \dot{\textbf{x}}_{0} \Vert } \tilde{x} (t) $, $y(t) = \Im (\gamma_{\C} (t) )$, with $\tilde{x} (t)=\Re (\gamma_{\C} (t) )$ and

\begin{equation} \gamma_{\C} (t):= \frac{aie^{vt}+b}{cie^{vt}+d}.\end{equation}
\end{prop}

\subsection{Computing the GIGO update in $\tGd$}

If we want to implement the GIGO algorithm in $\tGd$, we need to compute the natural gradient in $\tGd$, and to be able to compute the Riemannian exponential of $\tGd$. 

Using Proposition \ref{FishG}, we can compute the metric of $\tGd$ in the parametrization $(\mu, \sigma ) \mapsto \mathcal{N} (\mu , \sigma^{2} I)$. We find:

\begin{equation}
\label{tgdMetric}
\left(
\begin{matrix}
\frac{1}{\sigma^{2}} & 0 &  \ldots & 0\\
   0 & \ddots  & \ddots & \vdots\\
   \vdots &  \ddots&  \frac{1}{\sigma^{2}} &0\\
    0& \ldots & 0& \frac{2d}{\sigma^{2}}
\end{matrix}
\right).
\end{equation}

Since this matrix is diagonal, it is easy to invert, and we immediately have the natural gradient, and, consequently, the IGO speed.

\begin{prop}
In $\tGd$, the IGO speed $Y$ is given by:

\begin{equation}
Y_{\mu}=\sum_{i}\hat{w}_{i} (x_{i}-\mu ),
\end{equation}

\begin{equation}
Y_{\sigma}= \sum_{i} \hat{w}_{i} \left( \frac{\trsp{(x_{i}-\mu )}(x_{i}-\mu )}{2d \sigma} - \frac{\sigma}{2}   \right).
\end{equation}
\end{prop}
\begin{proof}
We recall the IGO speed is defined by $Y=I^{-1} (\theta^{t} ) \sum_{i=1}^{N} \hat{w}_{i} \dep{\ln P_{\theta } (x_{i} )}{\theta}$. Since $P_{\mu , \sigma } (x)= (2\pi\sigma^{2} )^{-d/2} \exp (-\frac{\trsp{(x-\mu)}(x-\mu)}{2\sigma^{2}})$, we have
$$\dep{\ln P_{\mu , \sigma } (x )}{\mu}=x-\mu,$$
$$\dep{\ln P_{\mu , \sigma} (x )}{\sigma}=-\frac{d}{\sigma}+\frac{\trsp{(x-\mu)}(x-\mu)}{\sigma^{3}}.$$
The result follows.
\end{proof}

The metric defined by Equation \ref{tgdMetric} is not exactly the metric of the hyperbolic space, but with the substitution $\mu \leftarrow \frac{\mu}{\sqrt{2d}} $, the metric becomes $\frac{2d}{\sigma^{2}} I$, which is proportional to the metric of the hyperbolic space, and therefore defines the same geodesics.

\begin{theo}[Geodesics of $\tGd$]
\label{geodTgd}

If $\gamma \colon t \mapsto \mathcal{N} (\mu(t) ,\sigma (t)^{2} I)$ is a geodesic of $\tGd$, then, there exists $a,b,c,d \in \R$ such that $ad-bc=1$ and $v > 0$ such that 

$\mu (t) = \mu (0) + \sqrt{ 2d} \frac{ \dot{\mu}_{0}}{\Vert \dot{\mu}_{0} \Vert } \tilde{r} (t) $, $\sigma(t) = \Im (\gamma_{\C} (t) )$, with $\tilde{r} (t)=\Re (\gamma_{\C} (t) )$ and

\begin{equation} \gamma_{\C} (t):= \frac{aie^{vt}+b}{cie^{vt}+d}.\end{equation}

\end{theo}

Now, in order to implement the corresponding GIGO algorithm, we only need to be able to find the coefficients $a,b,c,d,v$ corresponding to an initial position $(\mu_{0}  ,\sigma_{0})$, and an initial speed $(\dot{\mu}_{0} , \dot{\sigma}_{0} )$. It is a tedious but easy computation, the result of which is given in Proposition \ref{calcultgdtwist}.

%
%
%

The pseudocode of GIGO in $\tGd$ is also given in the Appendix: it is obtained by concatenating Algorithms \ref{algobase} and \ref{AlgoTGD}. \footnote{Proposition \ref{calcultgdtwist} and the pseudocode in the Appendix allow the metric to be slightly modified, see Section \ref{TwistSec}.}

\section{GIGO in $\Gd$}

\subsection{Obtaining a first order differential equation for the geodesics of $\Gd$}
\label{SecGIGO}

In the case where both the covariance matrix and the mean can vary freely, the equations of the geodesics have been computed in \cite{SolGeo} and \cite{Eriksen}. However, these articles start with the equations of the geodesics obtained with the Christoffel symbols, then partially integrate them, obtaining equations \eqref{mu1} and \eqref{sigma1} of Theorem \ref{Noe1}. These equations are in fact a consequence of Noether's theorem, and can be found directly.

\begin{theo} Let $\gamma : t \mapsto \mathcal{N}(\mu_{t}, \Sigma_{t} )$ be a geodesic of $\G_{d}$. Then, the following quantities do not depend on $t$: 
\begin{equation}J_{\mu} = \Sigma_{t}^{-1} \dot{\mu }_{t},\end{equation}
\begin{equation}\label{NoeSigma}J_{\Sigma} = \Sigma_{t}^{-1} ( \dot{\mu}_{t} \trsp{ \mu_{t} }+  \dot{\Sigma }_{t}). \end{equation}
\label{INoe}

\end{theo}
\begin{proof}

It is a direct application of Noether's theorem, with suitable groups of diffeomorphisms. By Proposition \ref{FishG}, the Lagrangian associated with the geodesics of $\Gd$ is 
 \begin{equation}\mathcal{L}(\mu, \Sigma, \dot{\mu } , \dot{\Sigma } ) =  \trsp{\dot{\mu}}\Sigma^{-1} \dot{\mu} +
\frac{1}{2} \tr (\dot{\Sigma} \Sigma^{-1} \dot{\Sigma} \Sigma^{-1} ).\end{equation}
Its derivative is 
\begin{equation} \dep{\mathcal{L}}{\dot{\theta}} = \left[ (h, H) \mapsto \, 2\trsp{\dot{\mu} } \Sigma^{-1} h  +  \tr (H\Sigma^{-1}\dot{\Sigma} \Sigma^{-1})\right].\end{equation}

Let us show that this Lagrangian is invariant under affine changes of basis (thus illustrating Definition \ref{AdmMap}). 

The general form of an affine change of basis is $\phi_{\mu_{0},A}:(\mu,\Sigma) \mapsto (A\mu + \mu_{0},A\Sigma \trsp{A})$, with $\mu_{0} \in \Rd$ and $A\in \mathrm{GL}_{d} (\R )$.

We have 

\begin{equation}
\mathcal{L}( \phi_{\mu_{0},A}(\mu,\Sigma),d\phi_{\mu_{0},A}(\dot{\mu},\dot{\Sigma}) )=\trsp{\dot{\overline{A\mu}}}(A\Sigma \trsp{A})^{-1} \dot{\overline{A\mu}} +
\frac{1}{2} \tr \left( \dot{\overline{A\Sigma \trsp{A}}} (A\Sigma \trsp{A})^{-1} \dot{\overline{A\Sigma \trsp{A}}} (A\Sigma \trsp{A})^{-1} \right),
\end{equation}
and since $\dot{\overline{A\mu}}=A\dot{\mu}$ and $\dot{\overline{A\Sigma \trsp{A}}}=A\dot{\Sigma}\trsp{A}$, we find easily that 
\begin{equation}\mathcal{L}( \phi_{\mu_{0},A}(\mu,\Sigma),d\phi_{\mu_{0},A}(\dot{\mu},\dot{\Sigma}) )=\mathcal{L}(\mu, \Sigma, \dot{\mu } , \dot{\Sigma } ),\end{equation}
or in other words: $\mathcal{L}$ is invariant under $ \phi_{\mu_{0},A}$ for any $\mu_{0}\in \Rd,A \in GL_{d} (\R)$.

In order to use Noether's theorem, we also need \emph{one-parameter groups} of transformations. We choose the following:
\begin{enumerate}

\item Translations of the mean vector. For any $i \in [1,d]$, let $h^{s}_{i} \colon  (\mu ,\Sigma ) \mapsto  (\mu + s e_{i} , \Sigma )$, where $e_{i}$ is the $i$-th basis vector. We have $\frac{\d{h_{i}^{s}}}{ds}\vert_{s=0} =(e_{i},0) $, so by Noether's theorem, $$ \dep{\mathcal{L}}{\dot{\theta}} (e_{i},0 ) =2 \trsp{\dot{\mu}} \Sigma^{-1} e_{i}=2 \trsp{ e_{i}} \Sigma^{-1} \dot{\mu}$$ remains constant for all $i$. The fact that $J_{\mu}$ is an invariant immediately follows.

\item Linear base changes. For any $i,j \in [1,d] $, let $h^{s}_{i,j} :(\mu, \Sigma) \mapsto ( \exp ( s E_{ij} ) \mu, \exp (s  E_{ij} ) \Sigma  \exp (s E_{ji} ) )$, where $E_{ij}$ is the matrix with a $1$ at position $(i,j)$, and zeros elsewhere. We have
$$ \frac{\d{h^{s}_{E_{ij}}}}{\d{s}} \vert_{s=0} = ( E_{ij}\mu ,E_{ij}\Sigma  + \Sigma E_{ji} ). $$
So by Noether's theorem, we then obtain the following invariants:
\begin{align}J_{ij}:=& \dep{\mathcal{L}}{\dot{\theta}} (E_{ij}\mu ,E_{ij}\Sigma  + \Sigma E_{ji} )\\
 =&  2\, \trsp{ \dot{\mu} } \Sigma^{-1} E_{ij} \mu + \tr  ((E_{ij}\Sigma + 
\Sigma E_{ji} )\Sigma^{-1} \dot{\Sigma} \Sigma^{-1}    ) \\
=& 2 \, \trsp{ (\Sigma^{-1} \dot{\mu} )} E_{ij} \mu +  \tr(E_{ij} \dot{\Sigma}  \Sigma^{-1} ) +  \tr (E_{ji} \Sigma^{-1} \dot{\Sigma}   )\\
=& 2 ( J_{\mu} \trsp{ \mu })_{ij} + 2 ( \Sigma^{-1} \dot{\Sigma} )_{ij}, \end{align}
and the coefficients of $J_{\Sigma}$ in (\ref{NoeSigma}) are the $(J_{ij} / 2).$

\end{enumerate}
\end{proof}

This leads us to first order equations satisfied by the geodesics mentionned in \cite{Eriksen}, \cite{SolGeo} and \cite{GeoRemarks}.

\begin{theo}[GIGO-$\Sigma$]
\label{Noe1}
$t\mapsto \mathcal{N} (\mu_{t}, \Sigma_{t})$ is a geodesic of $\Gd$ if and only if 
$\mu: t \mapsto \mu_{t}$ and $\Sigma: t \mapsto \Sigma_{t}$ satisfy the equations

\begin{equation} \label{mu1} \dot{\mu}_{t} = \Sigma_{t} J_{\mu}
 \end{equation}
\begin{equation} \label{sigma1}\dot{\Sigma}_{t} = \Sigma_{t} (J_{\Sigma} - J_{\mu} \trsp{ \mu }_{t} )= \Sigma_{t} J_{\Sigma} - \dot{\mu}_{t} \trsp{ \mu   }_{t},
\end{equation}
where
$$J_{\mu} = \Sigma_{0}^{-1} \dot{\mu}_{0},$$
and
$$J_{\Sigma} =  \Sigma_{0}^{-1} \left( \dot{\mu}_{0} \trsp{ \mu_{0} }+\dot{\Sigma}_{0}   \right).$$
\end{theo}

\textit{Proof.} It is an immediate consequence of Proposition \ref{INoe}. \\

These equations can be solved analytically (see \cite{SolGeo}), but usually, that is not the case, and they have to be solved numerically, for example with the Euler method (the corresponding algorithm, which we call GIGO-$\Sigma$, is described in the Appendix). The goal of the remainder of the subsection is to show that having to use the Euler method is fine.


To avoid confusion, we will call the step size of the GIGO algorithm ($\dt$ in Proposition \ref{DefGIGO}) ``GIGO step size", and the step size of the Euler method (inside a step of the GIGO algorithm) ``Euler step size". 

Having to solve our equations numerically brings two problems:

The first one is a theoretical problem: the main reason to study GIGO is its invariance under reparametrization of $\theta$, and we lose this invariance property when we use the Euler method. However, GIGO can get arbitrarily close to invariance by decreasing the Euler step size. In other words, the difference between two different IGO alorithms is $O(\dt^{2})$, and the difference between two different implementations of the GIGO algorithm is $O(h^{2})$, where $h$ is the Euler step size, and it is easier to reduce the latter. Still, without a closed form for the geodesics of $\Gd$, the GIGO update is rather expensive to compute, but it can be argued that most of the computation time will still be the computation of the objective function $f$.

The second problem is purely numerical: we cannot guarantee that the covariance matrix remains positive definite along the Euler method. Here, apart from finding a closed form for the geodesics, we have two solutions.

 We can enforce this \textit{a posteriori}: if the covariance matrix we find is not positive definite after a GIGO step, we repeat the failed GIGO step with a reduced Euler step size (in our implementation, we divided it by $4$, see Algorithm \ref{algo1} in the Appendix.). 

The other solution is to obtain differential equations on a square root of the covariance matrix\footnote{\emph{any} matrix $A$ such that $\Sigma = A \trsp{A}$.}.

\begin{theo}[GIGO-A]
\label{Noe2}
If $\mu: t \mapsto \mu_{t}$ and $A: t \mapsto A_{t}$ satisfy the equations
\begin{equation} \label{mu2}\dot{\mu}_{t} =  A_{t}\trsp{ A}_{t} J_{\mu},\end{equation}
\begin{equation} \label{A2}\dot{A}_{t} = \frac{1}{2} \trsp{(J_{\Sigma } - J_{\mu}\trsp{ \mu }_{t})} A_{t},\end{equation}
where 
$$J_{\mu} = \trsp{ (A_{0}^{-1})} A_{0}^{-1} \mu_{0} $$ 
and
$$ J_{\Sigma} = \trsp{ (A_{0}^{-1})} A_{0}^{-1} (\dot{\mu}_{0} \trsp{ \mu_{0}} + 
\dot{A}_{0} \trsp{ A_{0}} +
 A_{0} \trsp{ \dot{A}}_{0} ),$$
then, $t\mapsto \mathcal{N} (\mu_{t}, A_{t}\trsp{ A_{t}})$ is a geodesic of $\Gd$.
\end{theo}

\begin{proof}
This is a simple rewriting of Theorem \ref{Noe1}: if we write $\Sigma:= A\trsp{A}$, we find that $J_{\mu}$ and $J_{\Sigma}$ are the same as in Theorem \ref{Noe1}, and we have
$$\dot{\mu} = \Sigma J_{\mu},$$
and
$$\dot{\Sigma}= (\dot{A} \trsp{ A} + A \trsp{ \dot{A}} )=\frac{1}{2} \trsp{(J_{\Sigma } - J_{\mu}\trsp{ \mu })} A \trsp{A} + \frac{1}{2}A \trsp{A} (J_{\Sigma} -J_{\mu} \trsp{\mu} )$$
$$=\frac{1}{2} \trsp{(J_{\Sigma } - J_{\mu}\trsp{ \mu })}\Sigma + \frac{1}{2}\Sigma (J_{\Sigma} -J_{\mu} \trsp{\mu} )
= \frac{1}{2}\Sigma (J_{\Sigma} -J_{\mu} \trsp{\mu} ) + \frac{1}{2}  \trsp{[\Sigma (J_{\Sigma} -J_{\mu} \trsp{\mu} )]}.  $$

By Theorem \ref{Noe1}, $\Sigma (J_{\Sigma} -J_{\mu} \trsp{\mu} )$ is symmetric (since $\dot{\Sigma}$ has to be symmetric). Therefore, we have $\dot{\Sigma}=\Sigma (J_{\Sigma} -J_{\mu} \trsp{\mu} )$, and the result follows.

\end{proof}

Notice that Theorem \ref{Noe2} gives an equivalence, whereas Theorem \ref{Noe1} does not. The reason is that the square root of a symmetric positive definite matrix is not unique. Still, it is canonical, see discussion in Section \ref{Discussion}.

As for Theorem \ref{Noe1}, we can solve Equations \ref{mu2} and \ref{A2} numerically, and we obtain another algorithm (Algorithm \ref{algo2} in the Appendix, we will call it GIGO-$A$), with a behavior similar to the previous one (with equations \ref{mu1} and \ref{sigma1}). For both of them, numerical problems can arise when the covariance matrix is almost singular.

%
%
%
We have not managed to find any example where one of these two algorithms converged to the minimum of the objective function whereas the other did not, and their behavior is almost the same.
 
More interestingly, the performances of these two algorithms are also the same as the performances of the exact GIGO algorithm, using the equations of Section \ref{ExplGeo}.
 
Notice that even though GIGO-$A$ directly maintains a square root of the covariance matrix, which makes sampling new points easier\footnote{To sample a point from $\mathcal{N}(\mu,\Sigma)$, a square root of $\Sigma$ is needed.}, both GIGO-$\Sigma$ and GIGO-$A$ still have to invert the covariance matrix (or its square root) at each step, which is as costly as the decomposition, so one of these algorithms is roughly as expensive to compute as the other.


\subsection{Explicit form of the geodesics of $\Gd$}
\label{ExplGeo}
We now give the exact geodesics of $\Gd$: the following results are a rewriting of Theorem $3.1$ and its first corollary in \cite{SolGeo}.

\begin{theo}
Let $(\dot{\mu}_{0}, \dot{\Sigma}_{0})\in T_{\mathcal{N}(0,I)}\Gd$. The geodesic of $\Gd$ starting from $\mathcal{N} (0,1)$ with initial speed $(\dot{\mu}_{0}, \dot{\Sigma}_{0})$ is given by:
\begin{equation}
\label{eqGeo}
\exp_{\mathcal{N}(0,I)} (s\dot{\mu}_{0}, s\dot{\Sigma}_{0})=\mathcal{N}\left( 2R(s)\mathrm{sh}(\frac{sG}{2})G^{-}\dot{\mu}_{0}, R(s)\trsp{R(s)} \right),
\end{equation}
where $\exp$ is the Riemannian exponential of $\Gd$, $G$ is any matrix satisfying
\begin{equation}
G^{2}=\dot{\Sigma}_{0}^2+2\dot{\mu}_{0}\trsp{\dot{\mu}_{0}},
\end{equation}
\begin{equation}\label{eqR}
R(s)=\trsp{\left(\left(\mathrm{ch}(\frac{sG}{2})-\dot{\Sigma}_{0} G^{-} \mathrm{sh} (\frac{sG}{2}) \right)^{-1}\right)}
\end{equation}
and $G^{-}$ is a pseudo-inverse of $G$.\footnote{In \cite{SolGeo}, the existence of $G$ (as a square root of $\dot{\Sigma}_{0}^2+2\dot{\mu}_{0}\trsp{\dot{\mu}_{0}}$) is proved. Notice that anyway, in the expansions of \eqref{eqGeo} and \eqref{eqR}, only even powers of $G$ appear.}
\end{theo}

And since for all $A \in GL_{d} (\R)$, for all $\mu_{0}\in \Rd$, the application \begin{equation}\deff{\phi}{\Gd}{\Gd}{\mathcal{N} (\mu, \Sigma )}{\mathcal{N} (A \mu + \mu_{0}, A\Sigma \trsp{A}) }
\end{equation}
preserves the geodesics, we find the general expression for the geodesics of $\Gd$.
\begin{cor}
\label{VraiesGeos}
Let $\mu_{0} \in \Rd$, $A \in GL_{d} (\R)$, and $(\dot{\mu}_{0}, \dot{\Sigma}_{0})\in T_{\mathcal{N}(\mu_{0},A_{0}\trsp{A_{0}})}\Gd$. The geodesic of $\Gd$ starting from $\mathcal{N} (\mu,\Sigma)$ with initial speed $(\dot{\mu}_{0}, \dot{\Sigma}_{0})$ is given by:
\begin{equation}
\exp_{\mathcal{N}(\mu_{0},A_{0}\trsp{A_0})} (s\dot{\mu}_{0}, s\dot{\Sigma}_{0})=\mathcal{N}(\mu_1,A_1\trsp{A_1}),
\end{equation}
with
\begin{equation}
\mu_1=2A_{0}R(s)\mathrm{sh}({\frac{sG}{2}})G^{-}A_{0}^{-1}\dot{\mu}_{0}+\mu_0,
\end{equation}
\begin{equation}
A_1=A_{0}R(s),
\end{equation}
where $\exp$ is the Riemannian exponential of $\Gd$, $G$ is any matrix satisfying
\begin{equation}
G^{2}=A_{0}^{-1}(\dot{\Sigma}_{0} \Sigma_0^{-1} \dot{\Sigma}_{0}+2\dot{\mu}_{0}\trsp{\dot{\mu}_{0}})\trsp{(A_{0}^{-1})},
\end{equation}
\begin{equation}
R(s)=\trsp{\left(\left(\mathrm{ch}(\frac{sG}{2})-A_{0}^{-1}\dot{\Sigma}_{0}\trsp{(A_{0}^{-1})} G^{-} \mathrm{sh} (\frac{sG}{2})\right)^{-1}\right)},
\end{equation}
and $G^{-}$ is a pseudo-inverse of $G$.
\end{cor}

It should be noted that the final values for mean and covariance do not depend on the choice of $G$.\footnote{As a square root of 
$$A_{0}^{-1}(\dot{\Sigma}_{0} \Sigma_0^{-1} \dot{\Sigma}_{0}+2\dot{\mu}_{0}\trsp{\dot{\mu}_{0}})\trsp{(A_{0}^{-1})} .$$} The reason for this is that $\mathrm{ch}(G)$ is a Taylor series in $G^2$, and so are $\mathrm{sh}(G)G^-$ and $G^-\mathrm{sh}(G)$.

For our practical implementation, we actually used these Taylor series instead of the expression of the corollary.

\section{Comparing GIGO, xNES and pure rank-$\mu$ CMA-ES}
\label{SecComp}
\subsection{Definitions}
In this section, we recall the xNES and pure rank-$\mu$ CMA-ES, and we describe them in the IGO framework, thus allowing a reasonable comparison with the GIGO algorithms.

\subsubsection{xNES}
We recall a restriction\footnote{This restriction is sufficient to describe the numerical experiments in \cite{xNES}, the article which introduced xNES.} of the xNES algorithm, introduced in \cite{xNES}.

\begin{defi}[xNES algorithm]
\label{xN}
The xNES algorithm with sample size $N$, weights $w_{i}$, and learning rates $\eta_{\mu}$ and $\eta_{\Sigma}$ updates the parameters $\mu\in \Rd$, $A \in M_{d} (\R )$ with  the following rule:
At each step, $N$ points $x_{1},...,x_{N}$ are sampled from the distribution $\mathcal{N} (\mu, A\trsp{ A})$. Without loss of generality, we assume $f(x_{1} ) <... < f(x_{N})$.
The parameter is updated according to:
$$\mu \leftarrow \mu + \eta_{\mu} A G_{\mu},$$

$$A \leftarrow A \exp (\eta_{\Sigma} G_{M} / 2),$$

where, setting $z_{i} =A^{-1} (x_{i} - \mu )$:
$$G_{\mu} = \sum_{i=1}^{N} w_{i} z_{i},$$
$$G_{M} = \sum_{i=1}^{N} w_{i} (z_{i}\trsp{ z_{i}} - I ).$$

\end{defi}

The more general version decomposes the matrix $A$ as $\sigma B$, where $\det B=1$, and uses two different learning rates for $\sigma$ and for $B$. We gave the version where these two learning rates are equal\footnote{In particular, for the default parameters in \cite{xNES}, these two learning rates are equal. }. This restriction of the xNES algorithm can be described in the IGO framework, provided all the learning rates are equal (most of the elements of the proof can be found in \cite{xNES}\footnote{The proposition below essentially states that xNES is a natural gradient update.}, or in \cite{IGO}):

\begin{prop}[xNES as IGO]
\label{xNESIGO}
The xNES algorithm with sample size $N$, weights $w_{i}$, and learning rates $\eta_{\mu}=\eta_{\Sigma}=\dt$ coincides with the IGO algorithm with sample size $N$, weights $w_{i}$, step size $\dt$, and in which, given the current position $(\mu_{t}, A_{t})$, the set of Gaussians is parametrized by
$$ \phi_{\mu_t, A_t} :(\delta, M) \mapsto \mathcal{N} \left(   \mu_{t} + A_{t} \delta , \left( A_{t} \exp (\frac{1}{2} M ) \right) \trsp{ \left(  A_{t} \exp (\frac{1}{2} M ) \right)} \right),$$
with $\delta \in \R^{m}$ and $M \in \mathrm{Sym} (\R^{m})$.

The parameters maintained by the algorithm are $(\mu , A )$, and the $x_{i}$ are sampled from $\mathcal{N} (\mu , A\trsp{A}).$
\end{prop}

\begin{proof}

Let us compute the IGO update in the parametrization $\phi_{\mu_t, A_t}$: we have $\delta^{t}=0$, $M^{t}=0$, and by using Proposition \ref{FishG}, we can see that for this parametrization, the Fisher information matrix at $(0,0)$ is the identity matrix. The IGO update is therefore,
$$(\delta , M)^{t+\dt} = (\delta , M)^{t} + \dt Y_{\delta}(\delta, M) +\dt Y_{M} (\delta, M)=\dt Y_{\delta}(\delta, M) +\dt Y_{M} (\delta, M), $$
where
$$Y_{\delta} (\delta , M) = \sum_{i=1}^{N} w_{i} \nabla_{\delta  } \ln (p(x_{i} \vert (\delta , M) )$$
and 
$$ Y_{M} (\delta, M)= \sum_{i=1}^{N} w_{i} \nabla_{ M } \ln (p(x_{i} \vert (\delta , M) ).$$
Since\footnote{Using $\tr (M)= \log (\det (\exp (M) ) )$}
$$\ln P_{\delta , M} (x) = - \frac{d}{2} \ln (2\pi ) - \ln (\det A) - \frac{1}{2} \tr M - \frac{1}{2} \Vert \exp (-\frac{1}{2} M ) A^{-1} (x-\mu - A\delta ) \Vert^{2},$$
a straightforward computation yields
\begin{align*}  Y_{\delta} (\delta , M)  & =  \sum_{i=1}^{N} w_{i} z_{i}= G_{\mu},
\end{align*}
and
\begin{align*}
 Y_{M} (\delta , M)  &= \frac{1}{2}  \sum_{i=1}^{N} w_{i} (z_{i} \trsp{ z_{i}} - I ) = G_{M}.
\end{align*}
Therefore, the IGO update is:

$$\delta (t+\dt )=\delta (t) + \dt G_{\mu }, $$
$$M (t+\dt ) = M (t) + \dt G_{M}, $$
or, in terms of mean and covariance matrix:
$$\mu(t+\dt)=\mu (t) +\dt A(t) G_{\mu}$$
$$A (t+\dt) = A(t) \exp (\dt G_{M} / 2), $$
or
$$\Sigma (t+\dt) = A (t)\exp (\dt G_{M} ) \trsp{ A(t)}. $$

This is the xNES update.

\end{proof}

\subsubsection*{Using a square root of the covariance matrix}
\label{Discussion}
Firstly, we recall that the IGO framework (on $\Gd$, for example) emphasizes the Riemannian manifold structure on $\Gd$. All the algorithms studied here (including GIGO, which is not strictly speaking an IGO algorithm) define a trajectory in $\Gd$ (a new point for each step), and to go from a point $\theta$ to the next one ($\theta '$), we follow some curve $\gamma : [0,\dt] \to \Gd$, with $\gamma (0)=\theta$, $\gamma (\dt)=\theta '$, and $\dot{\gamma } (0)$ given by the natural gradient ($\dot{\gamma} (0)=\sum_{i=1}^{N} \hat{w}_{i} \tilde{\nabla}_{\theta} P_{\theta} (x_{i}) \in T_{\theta} \Gd$).\\

To be compatible with this point of view, an algorithm giving an update rule for a \emph{square root} of the covariance matrix $A$ has to satisfy the following condition: For a given initial speed, the covariance matrix $\Sigma^{t+\dt}$ after one step must depend only on $\Sigma^{t}$, and not on the square root $A^{t}$ chosen for $\Sigma^{t}$.

The xNES algorithm does satisfy this condition: consider two xNES algorithms, with the same learning rates, respectively at $(\mu, A_{1}^{t})$ and $(\mu , A_{2}^{t})$, with $A_{1}^{t}\trsp{(A_{1}^{t})}=A_{2}^{t}\trsp{(A_{2}^{t})}$ (i.e.: they define the same $\Sigma^{t}$), using the same samples $x_{i}$ to compute the natural gradient update\footnote{\label{ftModif} We are talking about the $x_{i}$ of Definition \ref{xN}, which are sampled from $\mathcal{N} (\mu, A\trsp{A})$, not the $z_{i}$, which are sampled from $\mathcal{N} (0, I)$, and are actually used to compute the $x_{i}$, by $x_{i}=\mu+ Az_{i}$. Notice that consequently, the ``sampled points" will \emph{not} be the same for the two algorithms if they use the \emph{same} random number generator with the \emph{same} seed. However, $\theta '$ will have the same probability distribution for both algorithms.}, then we will have $\Sigma_{1}^{t+\dt}=\Sigma_{2}^{t+\dt}$.
Using the definitions of Section \ref{SecTj}, we have just shown that what we will call the ``xNES trajectory" is well-defined.\\ 
 
It is also important to notice that, in order to be well defined, a natural gradient algorithm updating a square root of the covariance matrix has to specify more conditions than simply following the natural gradient.

The reason for this is that the natural gradient is a vector tangent to $\Gd$: it lives in a space of dimension $d(d+3)/2$ (the dimension of $\Gd$) whereas the vector $(\mu, A)$ lives in a space of dimension $d(d+1)$, which is too large: we will have no uniqueness for the solutions of the equations expressed with $A$ (this is why Theorem \ref{Noe2} is simply an implication, whereas Theorem \ref{Noe1} is an equivalence). 

More precisely, let us consider $A$ in $\mathrm{GL}_{d} (\R )$, and $v_{A}$, $v_{A}'$ two infinitesimal updates of $A$.
Since $\Sigma=A\trsp{A}$, the infinitesimal update of $\Sigma$ corresponding to $v_{A}$ (resp. $v_{A}'$) is $v_{\Sigma}=A\trsp{v_{A}} + v_{A} \trsp{A}$ (resp. $v_{\Sigma}'=A\trsp{v_{A}'} + v_{A}' \trsp{A}$).

It is now easy to see that $v_{A}$ and $v_{A}'$ define the same direction for $\Sigma$ (i.e. $v_{\Sigma}=v_{\Sigma}'$) if and only if $A\trsp{M}+M\trsp{A}=0$, where $M=v_{A}-v_{A}'$. This is equivalent to $A^{-1}M$ antisymmetric.

For any $A\in \mathrm{M}_{d} (\R )$, let us denote by $T_{A}$ the space of the matrices $M$ such that $A^{-1}M$ is antisymmetric, or in other words $T_{A}:=\lbrace u\in \mathrm{M}_{d} (\R),\, A\trsp{u}+u\trsp{A}=0  \rbrace$. Having a subspace $S_{A}$ in direct sum with $T_{A}$ for all $A$ is sufficient (but not necessary) to have a well-defined update rule. Namely, consider the (linear) application
$$
\deff{\phi_{A}}{\mathrm{M}_{d}(\R )}{\mathrm{S}_{d} (\R )}{v_{A}}{A\trsp{v_{A}}+v_{A}\trsp{A}},
$$
sending an infinitesimal update of $A$ to the corresponding update of $\Sigma$. It is not bijective, but as we have seen before, $\mathrm{Ker} \, \phi_{A}=T_{A}$, and therefore, if we have, for some $U_{A}$,
 \begin{equation} \mathrm{M}_{d} (\R )=U_{A}\oplus T_{A}, \end{equation}
then $\phi_{A}\vert_{U_{A}}$ is an isomorphism. 
Let $v_{\Sigma}$ be an infinitesimal update of $\Sigma$. We choose the following update of $A$ corresponding to $v_{\Sigma}$:
\begin{equation}
v_{A}:=(\phi_{A}\vert_{U_{A}})^{-1} (v_{\Sigma} ).\end{equation}
\\

Any $U_{A}$ such that $U_{A} \oplus T_{A}=\mathrm{M}_{d} (\R )$ is a reasonable choice to pick $v_{A}$ for a given $v_{\Sigma}$. The choice
$S_{A}=\lbrace u\in \mathrm{M}_{d} (\R),\, A\trsp{u} - u\trsp{A}=0  \rbrace$ has an interesting additional property: it is the orthogonal of $T_{A}$ for the norm 
\begin{equation}
\Vert v_{A} \Vert_{\Sigma}^{2}:=\mathrm{Tr} ( \trsp{v_{A}} \Sigma^{-1} v_{A} )=\mathrm{Tr} ( \trsp{(A^{-1}v_{A})} A^{-1}v_{A} ).\footnote{To prove this, remark that $T_{A}=\lbrace M\in \mathrm{M}_{d} (\R ), A^{-1}M\text{ antisymmetric} \rbrace$ and $S_{A}=\lbrace M\in \mathrm{M}_{d} (\R ), A^{-1}M\text{ symmetric} \rbrace$, and that if $M$ is symmetric and $N$ is antisymmetric, then \begin{equation} \mathrm{Tr} ( \trsp{M} N )= \sum_{i,j=1}^{d} m_{ij} n_{ij}=\sum_{i=1}^{d} m_{ii} n_{ii}+ \sum_{1\leq i < j\leq d} m_{ij}(n_{ij}+n_{ji}) =0,\end{equation}}\\
\end{equation}
and consequently, it can be defined without referring to the parametrization, which makes it a canonical choice.

Let us now show that this is the choice made by xNES and GIGO-$A$ (which are well-defined algorithms updating a square root of the covariance matrix).

\begin{prop}
Let $A\in M_{n} (\R )$. The $v_{A}$ given by the xNES and GIGO-A algorithms lies in $S_{A}= \lbrace u\in \mathrm{M}_{d} (\R),\, A\trsp{u} - u\trsp{A}=0  \rbrace=S_{A}$.

\end{prop}
\begin{proof}
For xNES, let us write $\dot{\gamma} (0)= (v_{\mu}, v_{\Sigma})$, and $v_{A}:=\frac{1}{2}AG_{M}$. We have $A^{-1} v_{A}=\frac{1}{2}G_{M}$, and therefore, forcing $M$ (and $G_{M}$) to be symmetric in xNES is equivalent to $A^{-1}v_{A}= \trsp{(A^{-1}v_{A}) }$, which can be rewritten as $A\trsp{v_{A}}=v_{A}\trsp{A}$.
For GIGO-$A$, Equation \ref{sigma1} shows that $\Sigma_{t} (J_{\Sigma} - J_{\mu} \trsp{ \mu }_{t} ) $ is symmetric, and with this fact in mind, Equation \ref{A2} shows that we have $A\trsp{v_{A}}=v_{A}\trsp{A}$ ($v_{A}$ is $\dot{A}_{t}$).\end{proof}

%


\subsubsection{Pure rank-$\mu$ CMA-ES}

We now recall the pure rank-$\mu$ CMA-ES algorithm. The general CMA-ES algorithm is described in \cite{CMATuto}. 
\begin{defi}[pure rank-$\mu$ CMA-ES algorithm]
The pure rank-$\mu$ CMA-ES algorithm with sample size $N$, weights $w_{i}$, and learning rates $\eta_{\mu}$ and $\eta_{\Sigma}$ is defined by the following update rule:
At each step, $N$ points $x_{1},...,x_{N}$ are sampled from the distribution $\mathcal{N} (\mu, \Sigma)$. Without loss of generality, we assume $f(x_{1} ) <... < f(x_{N})$.
The parameter is updated according to:
$$\mu \leftarrow \mu + \eta_{\mu} \sum_{i=1}^{N} w_{i} (x_{i} - \mu),$$
$$\Sigma \leftarrow \Sigma + \eta_{\Sigma} \sum_{i=1}^{N} w_{i} ((x_{i} - \mu) \trsp{ (x_{i} - \mu)} - \Sigma).$$

\end{defi}

The pure rank-$\mu$ CMA-ES can also be described in the IGO framework, see for example \cite{CMAIGO}.

\begin{prop}[pure rank-$\mu$ CMA-ES as IGO]
The pure rank-$\mu$ CMA-ES algorithm with sample size $N$, weights $w_{i}$, and learning rates $\eta_{\mu}=\eta_{\Sigma}=\dt$ coincides with the IGO algorithm with sample size $N$, weights $w_{i}$, step size $\dt$, and the parametrization $(\mu , \Sigma).$
\end{prop}

%
%
%
%
%

%
%

\subsection{Twisting the metric }
\label{TwistSec}
As we can see, the IGO framework does not allow to recover the learning rates for xNES and pure rank-$\mu$ CMA-ES, which is a problem, since usually, the covariance learning rate is set much smaller than the mean learning rate (see either \cite{xNES} or \cite{CMATuto}).

A way to recover these learning rates is to incorporate them directly into the metric (see also Blockwise GIGO, in Section \ref{ssBGIGO}). More precisely:
\begin{defi}[Twisted Fisher metric]
\label{TwistMet}
Let $\eta_{\mu}, \eta_{\Sigma} \in \R$, and
let $(P_{\theta} )_{\theta \in \Theta}$ be a family of normal probability distributions: $P_{\theta } = \mathcal{N} (\mu (\theta ), \Sigma (\theta ))$, with $\mu$ and $\Sigma $ $C^{1}$.
We call ``$(\eta_{\mu},\eta_{\Sigma} )$-twisted Fisher metric" the metric defined by:

\begin{equation} I_{i,j}(\eta_{\mu }, \eta_{\Sigma} ) (\theta )= \frac{1}{\eta_{\mu }} \dep{\trsp{ \mu} }{\theta_{i}} \Sigma^{-1} \dep{\mu }{\theta_{j}} + \frac{1}{\eta_{\Sigma}}  \frac{1}{2} \tr \left( \Sigma^{-1} \dep{ \Sigma}{\theta_{i}} \Sigma^{-1}  \dep{\Sigma}{\theta_{j}} \right). \end{equation}
\end{defi}

All the remainder of this section is simply a rewriting of the work in Section \ref{SecIGO} with the twisted Fisher metric instead of the regular Fisher metric.\\

This approach seems to be somewhat arbitrary. Arguably, the mean and the covariance play a ``different role" in the definition of a Gaussian (only the covariance can affect diversity, for example), but we lack a resasonable intrinsic characterization that would make this choice of twisting more natural. Moreover, this construction can be slightly generalized (see Appendix).

Still, we can define the corresponding Riemannian manifolds:
\begin{nota}We denote by $\Gd (\eta_{\mu},\eta_{\Sigma})$ (resp. $\tGd (\eta_{\mu},\eta_{\Sigma})$) the manifold of Gaussian distributions (resp. Gaussian distributions with covariance matrix proportional to the identity) in dimension $d$, equipped with the $(\eta_{\mu},\eta_{\Sigma} )$-twisted Fisher metric.
\end{nota}

And the IGO flow and the IGO algorithms can be modified to take into account the twisting of the metric: the $(\eta_{\mu},\eta_{\Sigma})$-twisted IGO flow reads

\begin{equation}
\label{TwIGOFlow} \frac{\d{\theta^{t}}}{\d{t}} =I(\eta_{\mu }, \eta_{\Sigma} )^{-1} (\theta )\int_{X} W^{f}_{\theta^{t}} (x)\nabla_{\theta} \ln P_{\theta} (x) P_{\theta^{t} }( \d{x} ).\footnote{The only difference with \eqref{FlotIGO} is that $I^{-1}(\theta)$ has been replaced by $I(\eta_{\mu }, \eta_{\Sigma} )^{-1} (\theta )$.} \end{equation}


This leads us to the twisted IGO algorithms.

\begin{defi}
The $(\eta_{\mu}, \eta_{\Sigma} )$-twisted IGO algorithm associated with parametrisation $\theta$, sample size $N$, step size $\dt$ and selection scheme $w$ is given by the following update rule:

$$\theta^{t+\dt} = \theta^{t} + \dt I(\eta_{\mu }, \eta_{\Sigma} )^{-1} (\theta^{t} ) \sum_{i=1}^{N} \hat{w}_{i} \dep{\ln P_{\theta } (x_{i} )}{\theta} . $$
\end{defi}

\begin{defi}
The $(\eta_{\mu}, \eta_{\Sigma} )$-twisted geodesic IGO algorithm associated with sample size $N$, step size $\dt$ and selection scheme $w$  is given by the following update rule:

\begin{equation} \theta^{t+\dt} = \exp_{\theta^{t}} (Y\dt) \end{equation}
where 
\begin{equation} Y= I(\eta_{\mu }, \eta_{\Sigma} )^{-1} (\theta^{t} ) \sum_{i=1}^{N} \hat{w}_{i} \dep{\ln P_{\theta } (x_{i} )}{\theta}. \end{equation}

By definition, the twisted geodesic IGO algorithm does not depend on the parametrization\footnote{Given that are working in $\Gd (\eta_{\mu},\eta_{\Sigma} )$, which is a ``less canonical" manifold than $\Gd$.}.
\end{defi}

%
%

We can also notice that there is some redundancy between $\dt$, $\eta_{\mu}$, and $\eta_{\Sigma}$: the only values actually appearing in the equations are $\dt \eta_{\mu}$ and $\dt \eta_{\Sigma}$. More formally:

\begin{prop}[Obvious remark] 
\label{obvrem}
Let $k, d, N\in \N$, $\eta_{\mu}, \eta_{\Sigma}, \dt, \lambda_{1}, \lambda_{2} \in \R$, and $w\colon [0;1] \to \R$.

The $(\eta_{\mu}, \eta_{\Sigma})$-twisted IGO algorithm with sample size $N$, step size $\dt$,   and selection scheme $w$ coincides with the $(\lambda_{1} \eta_{\mu},\lambda_{1} \eta_{\Sigma})$-twisted IGO algorithm with sample size $N$, step size $\lambda_{2} \dt$,   and selection scheme $\frac{1}{\lambda_{1}\lambda_{2}}w$. The same is true for geodesic IGO.
\end{prop}

\subsection{Algorithms after twisting}

We now give all the algorithms mentioned earlier, but with the twisted Fisher metric. Except for Proposition \ref{calcultgdtwist}, which is a simple calculation, the proofs of the following statements are an easy rewriting of their non-twisted counterparts: one can return to the non-twisted metric (up to a $\eta_{\Sigma}$ factor) by changing $\mu$ to $\frac{\sqrt{\eta_{\sigma}}}{\sqrt{\eta_{\mu}}}\mu$.

\begin{theo}
\label{tgdtwist}
If $\gamma \colon t \mapsto \mathcal{N} (\mu(t) ,\sigma (t)^{2} I)$ is a geodesic of $\tGd (\eta_{\mu}, \eta_{\sigma})$, then, there exists $a,b,c,d \in \R$ such that $ad-bc=1$ and $v > 0$ such that 

$\mu (t) = \mu (0) + \sqrt{ \frac{2d\eta_{\mu}}{\eta_{\sigma}} }\frac{ \dot{\mu}_{0}}{\Vert \dot{\mu}_{0} \Vert } \tilde{r} (t) $, $\sigma(t) = \Im (\gamma_{\C} (t) )$, with $\tilde{r} (t)=\Re (\gamma_{\C} (t) )$ and

\begin{equation} \gamma_{\C} (t):= \frac{aie^{vt}+b}{cie^{vt}+d}.\end{equation}
\end{theo}

\begin{prop}
\label{calcultgdtwist}
Let $n\in \N$, $v_{\mu} \in \Rn$, $v_{\sigma}$, $\eta_{\mu}$, $\eta_{\sigma}$, $\sigma_{0} \in \R$, with $\sigma_{0}>0$.

Let
$v_{r}:=\Vert v_{\mu} \Vert $, $\lambda=\sqrt{\frac{2n\eta_{\mu}}{\eta_{\sigma}}}$ $v:=\sqrt{\frac{\frac{1}{\lambda^{2}}v_{r}^{2}+v_{\sigma}^{2}}{\sigma_{0}^{2}}}$, $M_{0}:=\frac{1}{\lambda}\frac{v_{r}}{v\sigma_{0}^{2}}$, and $S_{0}:=\frac{v_{\sigma}}{v\sigma_{0}^{2}}$.

Let $c:=\left( \frac{\sqrt{M_{0}^{2}+S_{0}^{2}}-S_{0}}{2}\right)^{\frac{1}{2}}$ and $d:=\left( \frac{\sqrt{M_{0}^{2}+S_{0}^{2}}+S_{0}}{2} \right)^{\frac{1}{2}}$.

Let $\gamma_{\mathbb{C}} (t):=\sigma_{0} \frac{die^{vt}-c}{cie^{vt}+d}$.

Then \begin{equation}\gamma : t  \mapsto \mathcal{N} \left( \mu_{0} + \lambda \frac{ v_{\mu}}{\Vert v_{\mu} \Vert } \Re (\gamma_{\C} (t) ), \Im (\gamma_{\C} (t) )\right)\end{equation}
is a geodesic of $\tilde{\mathbb{G}}_{n} (\eta_{\mu}, \eta_{\sigma} )$ satisfying $\gamma (0)=(\mu_{0},\sigma_{0})$, and $\dot{\gamma} (0)=(v_{\mu},v_{\sigma})$.
\end{prop}

\begin{theo} Let $\gamma : t \mapsto \mathcal{N} (\mu_{t} , \Sigma_{t})$ be a geodesic of $\G_{d} (\eta_{\mu}, \eta_{\Sigma} )$. Then, the following quantities are invariant:
\begin{equation}J_{\mu} = \frac{1}{\eta_{\mu}} \Sigma_{t}^{-1} \dot{\mu }_{t},\end{equation}
\begin{equation}J_{\Sigma} = \Sigma_{t}^{-1} (\frac{1}{\eta_{\mu}}  \dot{\mu}_{t} \trsp{ \mu }_{t}+ \frac{1}{\eta_{\Sigma}} \dot{\Sigma }_{t}). \end{equation}
\end{theo}

\begin{theo}
If $\mu: t\mapsto \mu_{t}$ and $\Sigma : t\mapsto \Sigma_{t}$ satisfy the equations

\begin{equation} \dot{\mu}_{t} = \eta_{\mu}\Sigma_{t} J_{\mu} \end{equation}
\begin{equation} \dot{\Sigma}_{t} = \eta_{\Sigma}\Sigma_{t} (J_{\Sigma} - J_{\mu} \trsp{ \mu }_{t} )= \eta_{\Sigma}\Sigma_{t} J_{\Sigma} - \frac{\eta_{\Sigma}}{\eta_{\mu}}\dot{\mu}_{t} \trsp{ \mu   }_{t},\end{equation}
where
$$J_{\mu} = \frac{1}{\eta_{\mu}}\Sigma_{0}^{-1} \dot{\mu}_{0},$$
and
$$J_{\Sigma} =  \Sigma_{0}^{-1} \left( \frac{1}{\eta_{\mu}}\dot{\mu}_{0} \trsp{ \mu_{0} }+ \frac{1}{\eta_{\Sigma}}\dot{\Sigma}_{0}   \right).$$
then, $t\mapsto \mathcal{N} (\mu_{t}, \Sigma_{t})$ is a geodesic of $\Gd (\eta_{\mu}, \eta_{\sigma})$.
\end{theo}
\begin{theo}
If $\mu: t\mapsto \mu_{t}$ and $A: t\mapsto A_{t}$ satisfy the equations 
\begin{equation} \dot{\mu} = \eta_{\mu} A_{t}\trsp{ A}_{t} J_{\mu},\end{equation}
\begin{equation}\dot{A}_{t} = \frac{\eta_{\Sigma}}{2} \trsp{(J_{\Sigma } - J_{\mu}\trsp{ \mu }_{t})} A_{t},\end{equation}
where 
$$J_{\mu} = \frac{1}{\eta_{\mu}}\trsp{ (A_{0}^{-1})} A_{0}^{-1} \dot{\mu}_0 $$ 
and 
$$ J_{\Sigma} = \trsp{ (A_{0}^{-1})} A_{0}^{-1} (\frac{1}{\eta_{\mu}}\dot{\mu}_{0} \trsp{ \mu}_{0} + 
\frac{1}{\eta_{\Sigma}}\dot{A}_{0} \trsp{ A_{0}} +
\frac{1}{\eta_{\Sigma}} A_{0} \trsp{ \dot{A}}_{0} ),$$
then, $t\mapsto \mathcal{N} (\mu_{t}, A_{t}\trsp{ A_{t}})$ is a geodesic of $\Gd (\eta_{\mu}, \eta_{\sigma})$.
\end{theo}

Notice that the equations found by twisting the metric are exactly the equations without twisting, except that we have ``forced" the learning rates $\eta_{\mu}$, $\eta_{\Sigma}$ to appear by multiplying the increments of $\mu$ and $\Sigma$ by $\eta_{\mu}$ and $\eta_{\Sigma}$. 

\begin{prop}[xNES as IGO]
\label{xNESIGOt}
The xNES algorithm with sample size $N$, weights $w_{i}$, and learning rates $\eta_{\mu}, \eta_{\sigma}=\eta_{B}=\eta_{\Sigma}$ coincides with the $\frac{\eta_{\mu}}{\dt}, \frac{\eta_{\Sigma}}{\dt}$-twisted IGO algorithm with sample size $N$, weights $w_{i}$, step size $\dt$, and in which, given the current position $(\mu_{t}, A_{t})$, the set of Gaussians is parametrized by
$$(\delta, M) \mapsto  \mathcal{N} \left(   \mu_{t} + A_{t} \delta , \left( A_{t} \exp (\frac{1}{2} M ) \right) \trsp{ \left(  A_{t} \exp (\frac{1}{2} M ) \right)} \right),$$
with $\delta \in \R^{m}$ and $M \in \mathrm{Sym} (\R^{m})$.

The parameters maintained by the algorithm are $(\mu , A )$, and the $x_{i}$ are sampled from $\mathcal{N} (\mu , A\trsp{A}).$
\end{prop}

\begin{prop}[pure rank-$\mu$ CMA-ES as IGO]
The pure rank-$\mu$ CMA-ES algorithm with sample size $N$, weights $w_{i}$, and learning rates $\eta_{\mu}$ and $\eta_{\Sigma}$ coincides with the $( \frac{\eta_{\mu}}{\dt}, \frac{\eta_{\Sigma}}{\dt} )$-twisted IGO algorithm with sample size $N$, weights $w_{i}$, step size $\dt$, and the parametrization $(\mu , \Sigma)$.
\end{prop}


\begin{theo}
Let $\eta_{\mu}, \eta_\Sigma\in \R$, $\mu_0 \in \Rd$, 
$A_0 \in GL_{d} (\R)$, and $(\dot{\mu}_{0}, \dot{\Sigma}_{0})\in T_{\mathcal{N}(\mu_{0},A_{0}\trsp{A_{0}})}\Gd$. Let \begin{equation} \deff{h}{\Gd}{\Gd}{\mathcal{N} (\mu,\Sigma)}{\mathcal{N} (\sqrt{\frac{\eta_\mu}{\eta_\Sigma}}\mu,\Sigma)}.\end{equation}
We denote by $\phi$ (resp. $\psi$) the Riemannian exponential of $\Gd$ (resp. $\Gd (\eta_\mu,\eta_\Sigma )$) at $\mathcal{N} (\sqrt{\frac{\eta_\mu}{\eta_\Sigma}}\mu_{0}, A_0 \trsp{A_0})$ (resp. $\mathcal{N} (\mu_{0}, A_0 \trsp{A_0})$). We have:
\begin{equation}
\psi ( \dot{\mu}_0, \dot{\Sigma}_0 )=h  \circ \phi  (\sqrt{\frac{\eta_\Sigma}{\eta_\mu}}\dot{\mu}_0, \dot{\Sigma}_0 )
\end{equation}
\end{theo}

\subsection{Trajectories of different IGO steps}
\label{SecTj}
As we have seen, two different IGO algorithms (or an IGO algorithm and the GIGO algorithm) coincide at first order in $\dt$ when $\dt\rightarrow 0$. In this section, we study the differences between pure rank-$\mu$ CMA-ES, xNES, and GIGO by looking at the second order in $\dt$, and in particular, we show that xNES and GIGO do not coincide in the general case.

We view the updates done by one step of the algorithms as paths on the manifold $\Gd$, from $(\mu (t), \Sigma (t))$ to $\left( \mu(t+\dt), \Sigma (t+\dt) \right)$, where $\dt$ is the time step of our algorithms, seen as IGO algorithms. More formally:

\begin{defi}
\label{tjmaj}
\begin{enumerate}

\item We call GIGO update trajectory the application
$$T_{\mathrm{GIGO}} \colon (\mu , \Sigma , v_{\mu}, v_{\Sigma} ) \mapsto \left( \dt \mapsto 
\exp_{\mathcal{N}(\mu, A\trsp{ A })}
 (\dt \eta_{\mu} v_{\mu} , \dt \eta_{\Sigma} v_{\Sigma}   ) \right).$$
($\exp$ is the exponential of the Riemannian manifold $\Gd (\eta_{\mu} ,\eta_{\Sigma} )$)
\item We call xNES update trajectory the application
$$T_{\mathrm{xNES}} \colon (\mu , \Sigma , v_{\mu}, v_{\Sigma} ) \mapsto \left( 
\dt \mapsto 
\mathcal{N}(\mu + \dt \eta_{\mu}  v_{\mu}, A \exp [\eta_{\Sigma}  \dt A^{-1}  v_{\Sigma} \trsp{(A^{-1})} ] \trsp{ A })\right), $$
with $A\trsp{A}=\Sigma$. The application above does not depend on the choice of a square root $A$.

\item We call CMA update trajectory the application 
$$T_{\mathrm{CMA}} \colon (\mu , \Sigma , v_{\mu}, v_{\Sigma} ) \mapsto \left( \dt \mapsto 
\mathcal{N}(\mu +\dt \eta_{\mu}  v_{\mu} ,A\trsp{ A} +  \dt \eta_{\Sigma}  v_{\Sigma}) \right).  $$
\end{enumerate}

These appliations map the set of tangent vectors to $\Gd$ ($T\Gd $) to the curves in $\Gd (\eta_{\mu}, \eta_{\Sigma})$.

We will also use the following notation:
$\mu_{\mathrm{GIGO}}:= \phi_{\mu} \circ T_{\text{GIGO}}$,
 $\mu_{\mathrm{xNES}}:= \phi_{\mu} \circ T_{\text{xNES}}$,
$ \mu_{\mathrm{CMA}}:= \phi_{\mu} \circ T_{\text{CMA}}$,
 $\Sigma_{\mathrm{GIGO}}:= \phi_{\Sigma} \circ T_{\text{GIGO}}$,
$ \Sigma_{\mathrm{xNES}}:= \phi_{\Sigma} \circ T_{\text{xNES}}$,  
 and $\Sigma_{\mathrm{CMA}}:= \phi_{\Sigma} \circ T_{\text{CMA}}$, 
 where $\phi_{\mu}$ (resp. $\phi_{\Sigma}$) extracts the $\mu$-component (resp. the $\Sigma$-component) of a curve.


\end{defi}

For instance, $T_{\mathrm{GIGO}} (\mu , \Sigma, v_{\mu} , v_{\Sigma} ) (\dt)$ gives the position (mean and covariance matrix) of the GIGO algorithm after a step of size $\dt$, while $\mu_{\text{GIGO}}$ and $\Sigma_{\text{GIGO}}$ give respectively the mean component and the covariance component of this position.

This formulation ensures that the trajectories we are comparing had the same initial position and the same initial speed, which is the case provided the sampled values\footnote{These ``sampled values" are values \emph{directly} sampled from $\mathcal{N}(\mu, \Sigma )$, \emph{not} from $\mathcal{N}(0, I)$ and transformed.} are the same.

Different IGO algorithms coincide at first order in $\dt$. The following proposition gives the second order expansion of the trajectories of the algorithms.

\begin{prop}[second derivatives of the trajectories]
\label{TjDeri}We have:
$$ \mu_{\mathrm{GIGO}}(\mu , \Sigma , v_{\mu}, v_{\Sigma} )''(0) = \eta_{\mu} \eta_{\Sigma} v_{\Sigma }\Sigma_{0}^{-1} v_{\mu}, $$

$$ \mu_{\mathrm{xNES}}(\mu , \Sigma , v_{\mu}, v_{\Sigma} )''(0) = \mu_{\mathrm{CMA}}(\mu , \Sigma , v_{\mu}, v_{\Sigma} )''(0)=0,$$

$$ \Sigma_{\mathrm{GIGO}}(\mu , \Sigma , v_{\mu}, v_{\Sigma} )''(0)=  \eta_{\Sigma}^{2}v_{\Sigma} \Sigma^{-1} v_{\Sigma} - \eta_{\mu} \eta_{\Sigma} v_{\mu} \trsp{v_{\mu}}, $$

$$ \Sigma_{\mathrm{xNES}}(\mu , \Sigma , v_{\mu}, v_{\Sigma} )''(0)
= \eta_{\Sigma}^{2}v_{\Sigma} \Sigma^{-1} v_{\Sigma}, $$

$$ \Sigma_{\mathrm{CMA}}(\mu , \Sigma , v_{\mu}, v_{\Sigma} )''(0)=0. $$

\end{prop}

\begin{proof}
We can immediately see that the second derivatives of $\mu_{\mathrm{xNES}}$, $\mu_{\mathrm{CMA}}$, and $\Sigma_{\mathrm{CMA}}$ are $0$. Next, we have

\begin{align*}
\Sigma_{\mathrm{xNES}}(\mu , \Sigma , v_{\mu}, v_{\Sigma} )(t) 
&=A \exp [t A^{-1} \eta_{\Sigma} v_{\Sigma} \trsp{(A^{-1})} ] \trsp{ A } \\
&=A\trsp{A}+ t \eta_{\Sigma} v_{\Sigma} + \frac{t^{2}}{2}\eta_{\Sigma}^{2}v_{\Sigma}\trsp{(A^{-1})}A^{-1}v_{\Sigma} +o(t^{2})\\
&=\Sigma + t \eta_{\Sigma} v_{\Sigma} +\frac{t^{2}}{2}\eta_{\Sigma}^{2}v_{\Sigma}\Sigma^{-1}v_{\Sigma} +o(t^{2}) .\end{align*}

The expression of $\Sigma_{\mathrm{xNES}}(\mu , \Sigma , v_{\mu}, v_{\Sigma} )''(0)$ follows.

Now, for GIGO, let us consider the geodesic starting at $(\mu_{0} , \Sigma_{0} )$ with initial speed $\left( \eta_{\mu}v_{\mu}, \eta_{\Sigma} v_{\Sigma} \right) $. By writing $J_{\mu} (0) = J_{\mu} (t)$, we find $\dot{\mu} (t) = \Sigma (t) \Sigma^{-1}_{0} \dot{\mu}_{0}$. We then easily have $\ddot{\mu} (0)= \dot{\Sigma}_{0} \Sigma^{-1}_{0} \dot{\mu}_{0}.$ In other words 
$$\mu_{\mathrm{GIGO}}(\mu , \Sigma , v_{\mu}, v_{\Sigma} )''(0)= \eta_{\mu} \eta_{\Sigma} v_{\Sigma }\Sigma_{0}^{-1} v_{\mu}.$$

Finally, by using Theorem \ref{Noe1}, and differentiating, we find
$$\ddot{\Sigma}
 = \eta_{\Sigma} \dot{\Sigma}(J_{\Sigma} - J_{\mu} \trsp{\mu} ) - \eta_{\Sigma} \Sigma J_{\mu} \trsp{\dot{\mu}},  $$
$$\ddot{\Sigma}_{0} 
= \eta_{\Sigma} \dot{\Sigma}_{0}\frac{1}{\eta_{\Sigma}} \Sigma^{-1}_{0}\dot{\Sigma}_{0} - \frac{\eta_{\Sigma}}{\eta_{\mu}} \dot{\mu}_{0} \trsp{\dot{\mu}_{0}}
=\eta_{\Sigma}^{2} v_{\Sigma} \Sigma_{0}^{-1} v_{\Sigma}-\eta_{\Sigma}\eta_{\mu} v_{\mu} \trsp{v_{\mu}}. $$
\end{proof}

In order to interpret these results, we will look at what happens in dimension $1$:\footnote{In higher dimensions, we can suppose the algorithms exhibit a similar behavior, but an exact interpretation is more difficult for GIGO in $\Gd$.}
\begin{itemize}
\item In \cite{xNES}, it has been noted that xNES converges to quadratic minimums more slowly than CMA, and that it is less subject to premature convergence. That fact can be explained by observing that the mean update is exactly the same for CMA and xNES whereas xNES tends to have a higher variance (Proposition \ref{TjDeri} shows this at order $2$, and it is easy to see that in dimension $1$, for any $\mu$, $\Sigma $, $v_{\mu}$, $v_{\Sigma}$, we have $\Sigma_{\mathrm{xNES}}(\mu , \Sigma, v_{\mu}, v_{\Sigma} ) > \Sigma_{\mathrm{CMA}}(\mu , \Sigma, v_{\mu}, v_{\Sigma} )$). 
\item At order $2$, GIGO moves the mean faster than xNES and CMA if the standard deviation is increasing, and more slowly if it is decreasing. This seems to be a reasonable behavior (if the covariance is decreasing, then the algorithm is presumably close to a minimum, and it should not leave the area too quickly). This remark holds only for isolated steps, because we do not take into account the evolution of the variance.
\item The geodesics of $\mathbb{G}_{1}$ are half-circles (see Figure \ref{FigGeo} below --- we recall $\mathbb{G}_{1}$ is the Poincaré half-plane). Consequently, if the mean is supposed to move (which always happens), then $\sigma \rightarrow 0$ when $\dt \rightarrow \infty$. For example, a step whose initial speed has no component on the standard deviation will always decrease it. See also Proposition \ref{deltacrit}, about the optimization of a linear function.
\item For the same reason, for a given initial speed, the update of $\mu$ \textit{always} stays bounded as a function of $\dt$: it is not possible to make one step of the GIGO algorithm go further than a fixed point by increasing $\dt$. Still, the geodesic followed by GIGO changes at each step, so the mean of the overall algorithm is \textit{not} bounded. 
\end{itemize}
\begin{figure}[h]
\begin{center}
\input{GrapheGeotex.tex}
\end{center}
\caption{One step of GIGO update}
\label{FigGeo}
\end{figure}
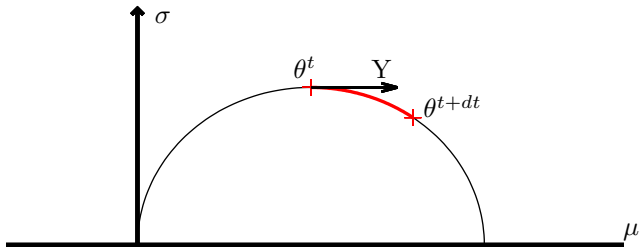


We now show that xNES follows the geodesics of $\Gd$ if the mean is fixed but that xNES and GIGO do not coincide otherwise.

\begin{prop}[xNES is not GIGO in the general case]
\label{xNESnotIGO}
Let $\mu, v_{\mu} \in \R^{d}, A\in \mathrm{GL}_{d}, v_{\Sigma} \in \mathrm{M}_{d}$.

Then the GIGO and xNES updates starting at $\mathcal{N} (\mu ,\Sigma)$ with initial speeds $v_{\mu}$ and $v_{\Sigma}$ follow the same trajectory if and only if the mean remains constant. In other words:

$T_{\mathrm{GIGO}} (\mu , \Sigma , v_{\mu}, v_{\Sigma} )=T_{\mathrm{xNES}} (\mu , \Sigma , v_{\mu}, v_{\Sigma})$ if and only if  $ v_{\mu} = 0$.
\end{prop} 

\begin{proof}
If $v_{\mu}=0$, then we can compute the GIGO update by using Theorem \ref{Noe1}: since $J_{\mu} =0 $, $\dot{\mu }=0$, and $\mu $ remains constant. Now, we have $J_{\Sigma} =  \Sigma^{-1} \dot{\Sigma }$: this is enough information to compute the update. Since this quantity is also preserved by the xNES algorithm (see for example the proof of Proposition \ref{xNESBGIGO}), the two updates coincide.

If $v_{\mu}\neq 0$, then $ \Sigma_{\mathrm{xNES}}(\mu , \Sigma , v_{\mu}, v_{\Sigma} )''(0) -\Sigma_{\mathrm{GIGO}}(\mu , \Sigma , v_{\mu}, v_{\Sigma} )''(0) =\eta_{\mu} \eta_{\Sigma} v_{\mu} \trsp{ v_{\mu }} \neq 0 , $ and in particular $T_{\mathrm{GIGO}} (\mu , \Sigma , v_{\mu}, v_{\Sigma} )  \neq T_{\mathrm{xNES}} (\mu , \Sigma , v_{\mu}, v_{\Sigma} )$.
\end{proof}

%
\subsection{Blockwise GIGO}
\label{ssBGIGO}
Although xNES is not GIGO, it is possible to define a family of algorithms extending GIGO and including xNES, by decomposing our family of probability distributions as a product, and by following the restricted geodesics simultaneously.

\begin{defi}[Splitting]
Let $\Theta$ be a Riemannian manifold. A splitting of $\Theta$ is $n$ manifolds $\Theta_{1},...,\Theta_{n}$ and a diffeomorphism  $\Theta \cong \Theta_{1} \times ... \times \Theta_{n}$.
If for all $x\in \Theta$, for all $1 \leq i < j \leq n$, we also have $T_{i,x}M \perp T_{j,x}M$ as subspaces of $T_{x}M$ (see Notation \ref{NotaDecou}), then the splitting is said to be compatible with the Riemannian strucutre\footnote{If the Riemannian manifold is not ambiguous, we will simply write a \textit{compatible splitting}.}.
\end{defi}

We now give some notation, and we define the blockwise GIGO update:

\begin{nota}
\label{NotaDecou}
Let $\Theta $ be a Riemannian manifold, $\Theta_{1},...,\Theta_{n}$ a splitting of $\Theta$, $\theta = (\theta_{1},...,\theta_{n} )\in \Theta$, $Y \in T_{\theta}\Theta$, and $1 \leq i \leq n$. 
\begin{itemize}

\item We denote by $\Theta_{\theta , i}$ the Riemannian manifold
 $$ \{ \theta_{1} \} \times ... \times \{ \theta_{i-1} \} \times \Theta_{i} \times \{ \theta_{i+1} \} \times ... \times \{ \theta_{n} \}, $$
with the metric induced from $\Theta$. There is a canonical isomorphism of vector spaces $T_{\theta} \Theta = \oplus_{i=1}^{n} T \Theta_{\theta , i}$. Moreover, if the splitting is compatible, it is an isomorphism of Euclidean spaces.
\item We denote by $\Phi_{\theta , i}$ the exponential at $\theta$ of the manifold $\Theta_{\theta , i}$.

 \end{itemize}
\end{nota}

\begin{defi}[Blockwise GIGO update]
Let $\Theta_{1},...,\Theta_{n}$ be a compatible splitting.\footnote{If the splitting is not compatible, it is possible to define exactly the same algorithm, but it does not seem relevant.}
The blockwise GIGO algorithm in $\Theta$ with splitting $\Theta_{1},...,\Theta_{n}$ associated with sample size $N$, step sizes $\dt_{1}, ..., \dt_{n}$ and selection scheme $w$ is given by the following update rule:

\begin{equation} \theta \leftarrow(\theta_{1}^{t+\dt_{1}},...,\theta_{n}^{t+\dt_{n}}) \end{equation}
where
\begin{equation} Y= I^{-1} (\theta^{t} ) \sum_{i=1}^{N} \hat{w}_{i} \dep{\ln P_{\theta } (x_{i} )}{\theta},\end{equation}
\begin{equation} \theta_{i}^{t+\dt_{i}} = \Phi_{\theta^{t} , i} (\dt_{i}Y_{i}),\end{equation}
with $Y_{i}$ the $T\Theta_{\theta ,i}$-component of $Y$. This update only depends on the splitting (and not on the parametrization inside each $\Theta_{i}$).
\end{defi}

Since blockwise GIGO only depends on the splitting (since we make no other choice\footnote{Except the tunable parameters: sample size, step sizes, and selection scheme.}), it can be thought as almost parametrization-invariant.

Notice that blockwise GIGO updates and twisted GIGO updates are two different things: firstly, blockwise GIGO can be defined on any manifold with a compatible splitting, whereas twisted GIGO (and twisted IGO) are only defined for Gaussians\footnote{Maybe a more general form of twisted IGO could be defined.}. But even in $\Gd (\eta_{\mu},\eta_{\Sigma})$, with the splitting $(\mu,\Sigma)$, these two algorithms are different: for instance, if $\eta_{\mu}=\eta_{\Sigma}$, and $\dt=1$ then the twisted GIGO is the regular GIGO algorithm, whereas blockwise GIGO is not (actually, we will prove that it is the xNES algorithm). The only thing blockwise GIGO and twisted GIGO have in common is that they are compatible\footnote{As defined at the end of Section \ref{SubSecGIGO}: A parameter $\theta^{t}$ following these updates with $\dt \rightarrow 0$ and $N \rightarrow \infty$ is a solution of Equation \ref{TwIGOFlow}. } with the $(\eta_{\mu},\eta_{\Sigma})$-twisted IGO flow (\ref{TwIGOFlow}).

We now have a new description of the xNES algorithm:

\begin{prop}[xNES is a Blockwise GIGO algorithm]
\label{xNESBGIGO}
The Blockwise GIGO algorithm in $\Gd$ with splitting $\Phi: \mathcal{N} (\mu, \Sigma) \mapsto (\mu, \Sigma)$, sample size $N$, step sizes $\dt_{\mu}, \dt_{\Sigma}$ and selection scheme $w$ coincides with the xNES algorithm with sample size $N$, weights $w_{i}$, and learning rates $\eta_{\mu}=\dt_{\mu}, \eta_{\sigma}=\eta_{B}=\dt_{\Sigma}$.
\end{prop}

\begin{proof}
Firstly, notice that the splitting $(\mu , \Sigma )$ is compatible, by Proposition \ref{FishG}.

Now, let us compute the Blockwise GIGO update: we have $\Gd \cong \Rd \times P_{d}$ where $P_{d}$ is the space of real positive definite matrices of dimension $d$. We have $\Theta_{\theta^{t},1}=(\Rd \times \{ \Sigma^{t} \}) \hookrightarrow \Gd$, $\Theta_{\theta^{t},2}=(\{ \mu^{t} \} \times P_{d})\hookrightarrow \Gd$. The induced metric on $\Theta_{\theta^{t},1}$ is the Euclidian metric, so we have
$$\mu \leftarrow \mu^{t} + \dt_{1} Y_{\mu}.$$

Since we have already shown (using the notation in Definition \ref{xN}) that $Y_{\mu} =AG_{\mu}$ (in the proof of Proposition \ref{xNESIGO}), 
 we find 
$$\mu \leftarrow \mu^{t} + \dt_{1} AG_{\mu}.$$
On $\Theta_{\theta^{t},2}$, we have the following Lagrangian for the geodesics:
$$\mathcal{L}(\Sigma ,\dot{\Sigma} )= \frac{1}{2} \tr (\dot{\Sigma}\Sigma^{-1}\dot{\Sigma}\Sigma^{-1} ). $$

By applying Noether's theorem, we find that $$J_{\Sigma} = \Sigma^{-1}\dot{\Sigma}$$ is invariant along the geodesics of $\Theta_{\theta^{t},2}$, so they are defined by the equation $\dot{\Sigma}=\Sigma J_{\Sigma}=\Sigma \Sigma_{0}^{-1} \dot{\Sigma_{0}}$ (and therefore, any update preserving the invariant $J_{\Sigma}$ will satisfy this first-order differential equation and follow the geodesics of $\Theta_{\theta^{t}, 2}$). The xNES update for the covariance matrix is given by $A(t)=A_{0}\exp (t G_{M}/2) $. Therefore, we have $\Sigma (t) = A_{0} \exp (tG_{M} ) \trsp{A_{0}}$,
$\Sigma^{-1} (t)= \trsp{(A_{0}^{-1})} \exp (-tG_{M})  A_{0}^{-1}$,
$ \dot{\Sigma} (t) = A_{0} \exp (tG_{M} ) G_{M}\trsp{A_{0}}$,
and finally
$\Sigma^{-1} (t)\dot{\Sigma} (t)=\trsp{(A_{0}^{-1})} G_{M}\trsp{A_{0}}=\Sigma^{-1}_{0} \dot{\Sigma}_{0} $. So xNES preserves $J_{\Sigma}$, and therefore, xNES follows the geodesics of $\Theta_{\theta^{t}, 2}$.\footnote{Notice that we had already proven this in Proposition \ref{xNESnotIGO}, since we are looking at the geodesics of $\Gd$ with fixed mean.}

\end{proof}

Although blockwise GIGO is somewhat ``less natural" than GIGO, it can be easier to compute for some splittings (as we have just seen), and in the case of the Gaussian distributions, the mean-covariance splitting seems natural.

Another advantage of blockwise GIGO is that it allows us to recover the learning rates without having to twist the metric, which felt like an ad-hoc solution. In particular, blockwise GIGO can be defined for families of probability distributions that are not Gaussian. 

%
%
%

\section{Numerical Experiments}
\label{SecNum}
We conclude this article with some numerical experiments to compare the behavior of GIGO, xNES and CMA-ES (we give the pseudocodes for these articles in the Appendix). We made two series of tests. The first one is a performance test, using classical benchmark functions, and the settings from \cite{xNES}. The goal of the second series of tests is to illustrate the computations in Section \ref{SecTj} by plotting the trajectories (standard deviation versus mean) of these three algorithms in dimension $1$. 

The source code is available at \url{https://www.lri.fr/~bensadon}. 

\subsection{Benchmarking}

For the first series of experiments, presented in Figure $\ref{FigureBenchmark}$, we used the following parameters, taken from \cite{xNES}:
\begin{itemize}
\item Varying dimension.

\item Sample size: $\lfloor 4+3 \log (d) \rfloor$

\item Weights:  $w_{i} = \frac{\max (0, \log (\frac{n}{2} +1)-\log (i)}{\sum_{j=1}^{N} \max (0, \log (\frac{n}{2} +1)-\log (j) } - \frac{1}{N}$

\item IGO step size and learning rates: $\dt=1, \eta_{\mu}=1, \eta_{\Sigma}=\frac{3}{5}\frac{ 3 +\log (d)}{d\sqrt{d}}.$

\item Initial position: $\theta^{0}=\mathcal{N} (x_{0} , I)$, where $x_{0}$ is a random point of the circle with center $0$, and radius $10$.

\item Euler method for GIGO: Number of steps: $100$. We used the GIGO-A variant of the algorithm.\footnote{No significant difference was noticed with GIGO-$\Sigma$, or with the exact GIGO algorithm. The only advantage of having an explicit solution of the geodesic equations is that the update is quicker to compute.}

\item We chose not to use the exact expression of the geodesics for this benchmarking to show that having to use the Euler method is fine. However, we did run the tests, and the results are basically the same as GIGO-$A$.
\end{itemize}

We plot the median number of runs to achieve target fitness ($10^{-8}$). Each algorithm has been tested in dimension $2$, $4$, $8$, $16$, $32$ and $64$: a missing point means that all runs converged prematurely. \clearpage

\textbf{Failed runs} 

In Figure \ref{FigureBenchmark}, a point is plotted even if only one run was successful. Below is the list of the settings for which some runs converged prematurely.

\begin{itemize}
\item Only one run reached the optimum for the cigar-tablet function with CMA in dimension $8$
\item Seven runs (out of $24$) reached the optimum for the Rosenbrock function with CMA in dimension $16$
\item About half the runs reached the optimum for the sphere function with CMA in dimension $4$.
\end{itemize}
For the following settings, \emph{all} runs converged prematurely.
\begin{itemize}
\item GIGO did not find the optimum of the Rosenbrock function in any dimension.
\item CMA did not find the optimum of the Rosenbrock function in dimension $2$, $4$, $32$ and $64$.
\item All the runs converged prematurely for the Cigar-tablet function in dimension $2$ with CMA, for the Sphere function in dimension $2$ for all algorithms, and for the Rosenbrock function in dimension $2$ and $4$ for all algorithms.
\end{itemize}

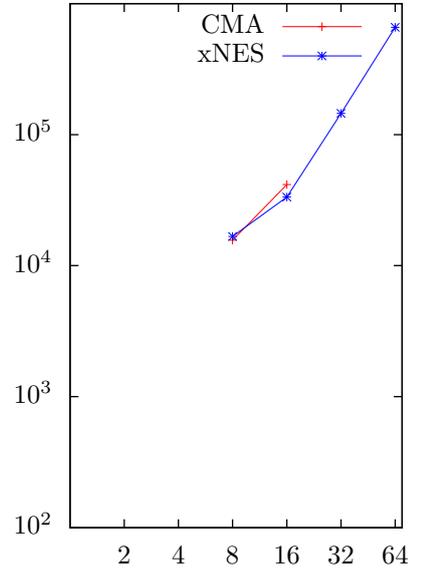
\begin{figure}
\hspace{-2.7cm}
\begin{tabular}{l l}

\begin{tabular}[b]{ | r | c  |l |}

\hline
Dimension & $d$ &From $2$ to $64$\\

Sample size  & $N$ & $4+3\log (d)$ \\

Weights  & $(w_{i})_{i\in [1,N]}$ & $\frac{\max (0, \log (\frac{n}{2} +1)-\log (i)}{\sum_{j=1}^{N} \max (0, \log (\frac{n}{2} +1)-\log (j) } - \frac{1}{N}$ \\

IGO Step size & $ \dt$ & $1$ \\

Mean  & $\eta_{\mu}$ & $1$ \\
learning rate & & \\

Covariance   & $\eta_{\Sigma}$ & $\frac{3}{5}\frac{ 3 +\log (d)}{d\sqrt{d}}$ \\
learning rate & & \\

\hline

Euler step-size   & $h$ &$0.01$($100$ steps) \\
(for GIGO only) & & \\

GIGO implementation   & & GIGO-$A$  \\

\hline
Sphere function & &$x \mapsto \sum_{i=1}^{d} x_{i}^{2}$\\

Cigar-tablet &  & $x \mapsto x_{1}^{2} +\sum_{i=2}^{d-1} 10^{4}x_{i}^{2} + 10^{8} x_{d}^{2}$\\

Rosenbrock & & $x \mapsto \sum_{i=1}^{d-1} (100(x_{i}^{2}-x_{i+1})^{2}+(x_{i}-1)^{2})$\\
\hline

$x$-axis &  & Dimension\\

$y$-axis & & Number of function calls\\
 & & to reach fitness $10^{-8}$.\\
\hline

\end{tabular}

&

\input{GrapheCTtex.tex}

\\

\input{GrapheSphtex.tex}

 &

\input{GrapheRostex.tex}

 \\
\end{tabular}

\caption{Median number of function calls to reach $10^{-8}$ fitness on $24$ runs for: Sphere function, Cigar-tablet function and Rosenbrock function. Initial position $\theta^{0}=\mathcal{N} (x_{0},I)$, with $x_{0}$ uniformly distributed on the circle of center $0$ and radius $10$. We recall that the ``CMA-ES" algorithm here is using the so-called \textit{pure rank-$\mu$ CMA update}.}
\label{FigureBenchmark}
\end{figure}

As the last item of the list above shows, all the algorithms converge prematurely in low dimension, probably because the covariance learning rate has been set too high (or because the sample size is too small). This is different from the results in \cite{xNES}.

This remark aside, as noted in \cite{xNES}, the xNES algorithm shows more robustness than CMA and GIGO: it is the only algorithm able to find the minimum of the Rosenbrock function in high dimensions. However, its convergence is consistently slower.

In terms of performance, when both of them work, CMA and GIGO are extremely close (GIGO is usually a bit better)\footnote{We recall that we use a simplified version of CMA (pure rank-$\mu$), that can be described in the IGO framework. It would be more accurate to call it ``IGO in the parametrization $(\mu,\Sigma)$".}. An advantage of GIGO is that it is theoretically defined for any $\dt$, $\eta_{\Sigma}$, whereas the covariance matrix maintained by CMA (not only pure rank-$\mu$ CMA) can stop being positive definite if $\eta_{\Sigma}\dt>1$. However, in that case, the GIGO algorithm is prone to premature convergence (remember Figure \ref{FigGeo}, and see Proposition \ref{deltacrit} below), and learning rates that high are rarely used in practice. 

\clearpage
\subsection{Plotting trajectories in $\mathbb{G}_{1}$}
\label{SubSecTj}
We want the second series of experiments to illustrate the remarks about the trajectories of the algorithms in Section \ref{SecTj}, so we decided to take a large sample size to limit randomness, and we chose a fixed starting point for the same reason. We use the weights below because of the property of quantile improvement proved in \cite{QuantImp}: the $1/4$-quantile will improve at each step. The parameters we used were the following:
\begin{itemize}

\item Sample size: $\lambda=5000$

\item Dimension $1$ only.

\item Weights: $w= 4.\mathbf{1}_{q\leq 1/4}$ ($w_{i}=4.\mathbf{1}_{i\leq 1250})$

\item IGO step size and learning rates: $\eta_{\mu}=1, \eta_{\Sigma}=\frac{3}{5}\frac{3 +   \log (d)}{d\sqrt{d}}=1.8,$ varying $\dt$.

\item Initial position: $\theta^{0}=\mathcal{N}(10, 1)$

\item Dots are placed at $t=0$, $1$, $2\ldots$ (except for the graph $\dt=1.5$, for which there is a dot for each step).
\end{itemize}

Figures \ref{Graphe11} to \ref{Graphe1n} show the optimization of $x\mapsto x^{2}$, and Figures \ref{Graphen1} to \ref{Graphenn} show the optimization of $x\mapsto -x$.

Figures \ref{Graphe1n-1}, \ref{Graphe1n} and \ref{Graphenn} show that when $\dt\geq 1$, GIGO reduces the covariance even at the first step. 
More generally, when using the GIGO algorithm in $\tGd$ for the optimization of a linear function, there exists a critical step size $\dt_{\mathrm{cr}}$ (depending on the learning rates $\eta_{\mu}, \eta_{\sigma}$, and on the weights $w_{i}$), above which GIGO will \textit{converge}, and we can compute its value when the weights are of the form $\mathbf{1}_{q\leq q_{0}} $\footnote{Notice that for $q_{0}\geq 0.5$, the discussion is not relevant because in that case, even the IGO \emph{flow} converges prematurely. Also compare with the critical $\dt$ of smoothed cross entropy method, and IGO ML, in \cite{IGO}.}.
\begin{prop}
Let $d\in \N$, $k$, $\eta_{\mu}$, $\eta_{\sigma} \in \R_{+}^{*}$ let $w= k.\mathbf{1}_{q \leq q_{0}}$, and let $$\deff{g}{\Rd}{\R}{x}{-x_{1}}.$$

Let $\mu_{n}$ be the first coordinate of the mean, and let $\sigma_{n}^{2}$ be the variance (at step $n$) maintained by the $(\eta_{\mu},\eta_{\sigma})$-twisted geodesic IGO algorithm in $\tGd$ associated with selection scheme $w$, sample size $\infty$\footnote{It has been proved in \cite{IGO} that IGO algorithms are \textit{consistent} with the IGO flow, i.e. they get closer and closer to the IGO flow as sample size tends to infinity. In other words $$\lim_{N \rightarrow \infty} \sum_{i=1}^{N} \hat{w}_{i} \dep{\ln P_{\theta} (x_{i})}{\theta}= \tilde{\nabla}_{\theta} \int_{X} W^{f}_{\theta^{t} (x)} P_{\theta} (\d x ).$$ Sample size $\infty$ means we replace the IGO speed in the GIGO update by its limit for large $N$. In particular, it is not random.}, and step size $\dt$, when optimizing $g$.

There exists $\dt_{cr}$ such that:
\begin{itemize}
\item if $\dt > \dt_{cr}$, $(\sigma_{n})$ converges to $0$ with exponential speed, and $(\mu_{n})$ converges.
\item if $\dt = \dt_{cr}$, $(\sigma_{n})$ remains constant, and $(\mu_{n})$ tends to $\infty$ with linear speed.
\item if $0 < \dt < \dt_{cr}$, both $(\sigma_{n})$ and $\mu_{n}$ tend to $\infty$ with exponential speed.
\end{itemize}

\end{prop}

The proof and the expression of $\dt_{\mathrm{cr}}$ can be found in the Appendix.

In the case corresponding to $k=4$, $n=1$, $q_{0}=1/4$, $\eta_{\mu}=1$, $\eta_{\sigma}=1.8$, we find:
\begin{equation}\dt_{\mathrm{cr}}    \approx  0.84  .\end{equation}

%

\clearpage


\begin{figure}
\vspace{-2cm}
\caption{Trajectories of GIGO, CMA and xNES optimizing $x\mapsto x^{2}$ in dimension $1$ with $\dt =0.01$, sample size $5000$, weights $w_{i}=4.\mathbf{1}_{i\leq 1250}$, and learning rates $\eta_{\mu}=1$, $\eta_{\Sigma}=1.8$. One dot every $100$ steps. All algorithms exhibit a similar behavior}
\begin{tabular}{c}
\input{Graphe001tex.tex}

\label{Graphe11}
\\
\input{Graphe01tex.tex}
\end{tabular}

\caption{Trajectories of GIGO, CMA and xNES optimizing $x\mapsto x^{2}$  in dimension $1$ with $\dt =0.1$, sample size $5000$, weights $w_{i}=4.\mathbf{1}_{i\leq 1250}$, and learning rates $\eta_{\mu}=1$, $\eta_{\Sigma}=1.8$. One dot every $10$ steps. All algorithms exhibit a similar behavior, differences start to appear. It cannot be seen on the graph, but the algorithm closest to zero after $400$ steps is CMA ($\sim 1.10^{-16}$, followed by xNES ($\sim 6.10^{-16}$) and GIGO ($\sim 2.10^{-15}$). }
\end{figure}

\begin{figure}
\vspace{-2.5cm}
\caption{Trajectories of GIGO, CMA and xNES optimizing $x\mapsto x^{2}$  in dimension $1$ with $\dt =0.5$, sample size $5000$, weights $w_{i}=4.\mathbf{1}_{i\leq 1250}$, and learning rates $\eta_{\mu}=1$, $\eta_{\Sigma}=1.8$. One dot every $2$ steps. Stronger differences. Notice that after one step, the lowest mean is still GIGO ($\sim 8.5$, whereas xNES is around $8.75$), but from the second step, GIGO has the highest mean because of the lower variance.}
\begin{tabular}{c}
\input{Graphe05tex.tex}

\\
\input{Graphe1tex.tex}
\end{tabular}
\caption{Trajectories of GIGO, CMA and xNES optimizing $x\mapsto x^{2}$  in dimension $1$ with $\dt =1$, sample size $5000$, weights $w_{i}=4.\mathbf{1}_{i\leq 1250}$, and learning rates $\eta_{\mu}=1$, $\eta_{\Sigma}=1.8$. One dot per step. The CMA-ES algorithm fails here, because at the fourth step, the covariance matrix is not positive definite anymore (It is easy to see that the CMA-ES update is always defined if $\dt \eta_{\Sigma} <1$, but this is not the case here). Also notice (see also Proposition \ref{deltacrit}) that at the first step, GIGO \emph{decreases} the variance, whereas the $\sigma$-component of the IGO speed is positive.}
\label{Graphe1n-1}
\end{figure}

\begin{figure}
\vspace{-1cm}
\caption{Trajectories of GIGO, CMA and xNES optimizing $x\mapsto x^{2}$  in dimension $1$ with $\dt =1.5$, sample size $5000$, weights $w_{i}=4.\mathbf{1}_{i\leq 1250}$, and learning rates $\eta_{\mu}=1$, $\eta_{\Sigma}=1.8$. One dot per step. Same as $\dt=1$ for CMA. GIGO converges prematurely.}
\begin{tabular}{c}
\input{Graphe15tex.tex}

\label{Graphe1n}
\\
\input{Lin001tex.tex}
\end{tabular}
\caption{Trajectories of GIGO, CMA and xNES optimizing $x\mapsto -x$  in dimension $1$ with $\dt =0.01$, sample size $5000$, weights $w_{i}=4.\mathbf{1}_{i\leq 1250}$, and learning rates $\eta_{\mu}=1$, $\eta_{\Sigma}=1.8$. One dot every $100$ steps. Almost the same for all algorithms.}
\label{Graphen1}
\end{figure} 

\begin{figure}
\vspace{-1.3cm}
\caption{Trajectories of GIGO, CMA and xNES optimizing $x\mapsto -x$  in dimension $1$ with $\dt =0.1$, sample size $5000$, weights $w_{i}=4.\mathbf{1}_{i\leq 1250}$, and learning rates $\eta_{\mu}=1$, $\eta_{\Sigma}=1.8$. One dot every $10$ steps. It is not obvious on the graph, but xNES is faster than CMA, which is faster than GIGO.}
\begin{tabular}{c}
\input{Lin01tex.tex}

\\
\input{Lin1tex.tex}
\end{tabular}
\caption{Trajectories of GIGO, CMA and xNES optimizing $x\mapsto -x$  in dimension $1$ with $\dt =1$, sample size $5000$, weights $w_{i}=4.\mathbf{1}_{i\leq 1250}$, and learning rates $\eta_{\mu}=1$, $\eta_{\Sigma}=1.8$. One dot per step. GIGO \emph{converges}, for the reasons discussed earlier.}
\label{Graphenn}
\end{figure} 

\clearpage
\section*{Conclusion}
\addcontentsline{toc}{section}{Conculsion}
We introduced the geodesic IGO algorithm and, we showed that in the case of Gaussian distributions, Noether's theorem directly gives a first order equation satisfied by the geodesics.
In terms of performance, the GIGO algorithm is similar to pure rank-$\mu$ CMA-ES, which is rather encouraging
: it would be interesting to test GIGO on real problems. Moreover, GIGO is a reasonable and totally parametrization invariant\footnote{Provided we can compute the solution of the equations of the geodesics.} algorithm, and as such, it should be studied for other families of probability distributions\footnote{Like Bernoulli distributions. However, in that case, the length of the geodescis is finite, and other problems arise.}. Noether's theorem could be a crucial tool for this.

We also showed that xNES and GIGO are not the same algorithm, and we defined Blockwise GIGO, a simple extension of the GIGO algorithm showing that xNES has a special status as it admits a definition which is ``almost" parametrization-invariant.

\newpage

\section*{Appendix}
\addcontentsline{toc}{section}{Appendix}

\begin{proof}[Proof of Proposition \ref{deltacrit}]

Let us first consider the case $k=1$.

When optimizing a linear function, the non-twisted IGO flow in $\tGd$ with the selection function $w:q\mapsto \mathbf{1}_{q\leq q_{0}}$ is known \cite{IGO}, and in particular, we have:

\begin{equation}
\label{muntIGOf}
\mu_{t}=\mu_{0}+ \frac{\beta (q_{0})}{\alpha (q_{0})} \sigma_{t},
\end{equation}
\begin{equation}
\label{sigmantIGOf}
\sigma_{t}=\sigma_{0}\exp ( \alpha(q_{0}) t),
\end{equation}
where, if we denote by $\mathcal{N}$ a random vector following a standard normal distribution, and $\mathcal{F}$ the cumulative distribution of a standard normal distribution,
\begin{equation}
\alpha (q_{0},d)=\frac{1}{2d} \left( \int_{0}^{q_{0}} \mathcal{F}^{-1} (u)^{2} du - q_{0} \right),
\end{equation}
and 
\begin{equation}
\beta (q_{0} )= \mathbb{E} ( \mathcal{N}  \mathbf{1}_{\mathcal{N} \leq \mathcal{F}^{-1}(q_{0})}).
\end{equation}
In particular, $\alpha:=\alpha (\frac{1}{4}, 1 )\approx 0.107 $ and $\beta:=\beta (\frac{1}{4})\approx -0.319$.

With a minor modification of the proof in \cite{IGO}, we find that the $(\eta_{\mu},\eta_{\sigma})$-twisted IGO flow is given by:
\begin{equation}
\label{muIGOf}
\mu_{t}=\mu_{0}+ \frac{\beta (q_{0})}{\alpha (q_{0})} \sigma_{0} \exp (\eta_{\mu} \alpha (q_{0})t),
\end{equation}
\begin{equation}
\label{sigmaIGOf}
\sigma_{t}=\sigma_{0}\exp (\eta_{\sigma}\alpha(q_{0}) t),
\end{equation}

Notice that Equation \ref{muIGOf} shows that the assertions about the convergence of $(\sigma_{n})$ immediately imply 
the assertions about the convergence of $(\mu_{n})$.\newline

Let us now consider a step of the GIGO algorithm:
The twisted IGO speed is $Y=( \eta_{\mu} \beta \sigma_{0},\eta_{\sigma} \alpha \sigma_{0} )$, with $\alpha \sigma_{0}>0$ (i.e. the variance should be increased\footnote{This is where we need $q_{0} < 0.5$.}).

Proposition \ref{calcultgdtwist} shows that the covariance at the end of the step is (using the same notation):
\begin{equation}\sigma ( \dt )=\sigma (0) \Im (\frac{die^{v\dt}-c}{cie^{v\dt}+d})=\sigma (0) \frac{e^{v\dt}(d^{2}+c^{2})}{c^{2}e^{2v\dt}+d^{2}} =:\sigma (0 ) f(\dt), \end{equation}
and it is easy to see that $f$ only depends on $\dt $ (and on $q_{0}$). In other words, $f (\dt )$ will be the same at each step of the algorithm. The existence of $\dt_{\mathrm{cr}}$ easily follows\footnote{Also recall Figure \ref{FigurePKR} in Section \ref{PrelimHspace}.}, and $\dt_{\mathrm{cr}}$ is the positive solution of $f (x)=1$.  

After a quick computation, we find:

\begin{equation} \exp (v\dt_{\mathrm{cr}} )=
\frac{\sqrt{1+u^{2}}+1}{\sqrt{1+u^{2}}-1}. \end{equation}
where 
\begin{equation} u:=\sqrt{\frac{\eta_{\mu}}{2n\eta_{\sigma}}}\frac{\beta}{\alpha} , \end{equation}
and
\begin{equation} v:=\sqrt{\eta_{\sigma}^{2}\alpha^{2}+
\frac{\eta_{\mu}\eta_{\sigma}}{2n}\beta^{2}}. \end{equation}
\label{deltacrit}

Finally, for $w=k.\mathbf{1}_{q\leq q_{0}}$, Proposition \ref{obvrem} shows that
\begin{equation}\dt_{\mathrm{cr}}    =\frac{1}{k}   \frac{1}{v}\ln \left( \frac{\sqrt{1+u^{2}}+1}{\sqrt{1+u^{2}}-1} \right). \end{equation}

\end{proof}

\subsection*{Generalization of the twisted Fisher metric}
\addcontentsline{toc}{subsection}{Generalization of the twisted Fisher metric}

The following definition is a more general way to introduce the twisted Fisher metric.

\begin{defi}
\label{TwistGen}
Let $(\Theta,g) $ be a Riemannian manifold, $(\Theta_{1},g\vert_{\Theta_1}),...,(\Theta_{n},g\vert_{\Theta_n})$ a splitting (as defined in Section \ref{ssBGIGO}) of $\Theta$ compatible with the metric $g$.

We call \emph{$(\eta_1,...,\eta_n)$-twisted metric on $(\Theta , g)$ for the splitting $\Theta_1,..., \Theta_n$} the metric $g'$ on $\Theta$ defined by $g'\vert_{\Theta_i}=\frac{1}{\eta_i}g\vert_{\Theta_i}$ for $1\leq i \leq n$.
\end{defi}

\begin{prop}
The $(\eta_\mu , \eta_\Sigma )$-twisted metric on $\Gd$ with the Fisher metric for the splitting $\mathcal{N}(\mu , \Sigma) \mapsto (\mu, \Sigma )$ coincides with the $(\eta_{\mu},\eta_\Sigma)$-twisted Fisher metric from Definition \ref{TwistMet}.
\end{prop}

\begin{proof}
It is easy to see that the $(\eta_{\mu},\eta_\Sigma)$-twisted Fisher metric satisfies the condition in Definition \ref{TwistGen}.
\end{proof}

\subsection*{Pseudocodes}
\addcontentsline{toc}{subsection}{Pseudocodes}

\subsubsection*{For all algorithms}

All studied algorithms have a common part, given here:

Variables: $\mu, \Sigma$ (or $A$ such that $\Sigma=A\trsp{A}$).

List of parameters: $f \colon \Rd \to \R$, step size $\dt$, learning rates $\eta_{\mu}, \eta_{\Sigma}$, sample size $\lambda$, weights $(w_{i})_{i\in [1,\lambda]}$, $N$ number of steps for the Euler method, $r$ Euler step size reduction factor (for GIGO-$\Sigma$ only).

\begin{algorithm}
\caption{For all algorithms}

\begin{algorithmic}
\STATE $\mu \leftarrow \mu_{0}$
\IF{The algorithm updates $\Sigma$ directly}

\STATE $\Sigma \leftarrow \Sigma_{0}$

\STATE Find some $A$ such that $\Sigma=A\trsp{A}$

\ELSE[The algorithm updates a square root $A$ of $\Sigma$]

\STATE $A\leftarrow A_{0}$

\STATE $\Sigma=A\trsp{A}$
\ENDIF
\WHILE{NOT (Termination criterion)}

\FOR{$i=1$ to $\lambda$} 

\STATE $z_{i} \sim \mathcal{N} (0,I )$

\STATE $x_{i}=Az_{i}+\mu$
\ENDFOR

\STATE Compute the IGO initial speed (included in the algorithms below)

\STATE Update the mean and the covariance (the updates are Algorithms \ref{algo1} to \ref{algoN}).

\ENDWHILE

\end{algorithmic}
\label{algobase}
\end{algorithm}

Notice that we always need a square root $A$ of $\Sigma$ to sample the $x_{i}$, but the decomposition $\Sigma=A\trsp{A}$ is not unique. Two different decompositions will give two algorithms such that one is a \textit{modification} of the other as a stochastic process: same law (the $x_{i}$ are abstractly sampled from $\mathcal{N} (\mu ,\Sigma)$), but different trajectories (for given $z_{i}$, different choices for the square root will give different $x_{i}$, see also footnote \ref{ftModif}, in Section \ref{Discussion}). For GIGO-$\Sigma$,\footnote{We recall that the other implementation directly maintains a square root of $\Sigma$.} since we have to invert the covariance matrix, we used the Cholesky decomposition ($A$ lower triangular). Usually, in CMA-ES, the square root of $\Sigma$ ($\Sigma=A\trsp{A}$, $A$ symmetric) is used.

%
%
%
%
%
\clearpage
\subsubsection*{Updates}

When describing the different updates,  $\mu$, $\Sigma$, $A$, the $x_{i}$ and the $z_{i}$ are those defined in Algorithm \ref{algobase}. 

For Algorithm \ref{algo1} (GIGO-$\Sigma$), when the covariance matrix after one step is not positive definite, we compute the update again, with a step size divided by $r$ for the Euler method.\footnote{We have no reason to recommend any particular value of $r$, the only constraint is $r>1$.}

\begin{algorithm}

\caption{GIGO Update, one step, updating the covariance matrix}
\begin{algorithmic}
\label{algo1}

\STATE $1$. Compute the IGO speed:

\STATE $\displaystyle{ v_{\mu}=A\sum_{i=1}^{\lambda} w_{i} z_{i}},$

\STATE $\displaystyle{ v_{\Sigma}
=A \sum_{i=1}^{\lambda} w_{i} \left( z_{i} \trsp{ z_{i}} - I \right) \trsp{A} }.\newline$

\STATE $2$. Compute the Noether invariants:\\

\STATE $\displaystyle{J_{\mu} \leftarrow \Sigma^{-1} v_{\mu}},$

\STATE $\displaystyle{J_{\Sigma} \leftarrow\Sigma^{-1} ( v_{\mu}\trt \mu +v_{\Sigma}  )}.\newline$

\STATE $3$. Solve numerically the equations of the geodesics:\\

\STATE Unhappy $\leftarrow$ true

\STATE $\displaystyle{\mu_{0} \leftarrow \mu}$

\STATE $\displaystyle{\Sigma_{0} \leftarrow  \Sigma }$

\STATE $k=0$

\WHILE{Unhappy}

\STATE $\displaystyle{\mu \leftarrow \mu_{0}}$

\STATE $\displaystyle{\Sigma \leftarrow  \Sigma_{0} }$

\STATE $h \leftarrow \dt/(N r^{k})$

\FOR{$i=1$ to $Nr^{k}$}

\STATE $\displaystyle{\mu \leftarrow \mu  + h \eta_{\mu} \Sigma J_{\mu}}$ 
 
\STATE $\displaystyle{\Sigma \leftarrow \Sigma + h \eta_{\Sigma} \Sigma (J_{\Sigma} -J_{\mu} \trsp{\mu})}$

\ENDFOR

\IF{$\Sigma $ positive definite} 

\STATE Unhappy $\leftarrow$ false 

\ENDIF 

\STATE $k \leftarrow k+1$

\ENDWHILE

\RETURN $\mu$, $\Sigma$
\end{algorithmic}
\end{algorithm}
\clearpage

\begin{algorithm}
\caption{GIGO Update, one step, updating a square root of the covariance matrix}
\begin{algorithmic}
\label{algo2}
\STATE $1$. Compute the IGO speed:

\STATE $\displaystyle{ v_{\mu}=A\sum_{i=1}^{\lambda} w_{i} z_{i}},$

\STATE $\displaystyle{ v_{\Sigma}
=A \sum_{i=1}^{\lambda} w_{i} \left( z_{i} \trsp{ z_{i}} - I \right) \trsp{A} }.\newline$

\STATE $2$. Compute the Noether invariants:\\

\STATE $\displaystyle{J_{\mu} \leftarrow \Sigma^{-1} v_{\mu}},$

\STATE $\displaystyle{J_{\Sigma} \leftarrow\Sigma^{-1} ( v_{\mu}\trt \mu +v_{\Sigma}  )}.\newline$

\STATE $3$. Solve numerically the equations of the geodesics:\\

\STATE  $h\leftarrow \dt/N$

\FOR{$i=1$ to $N$}

\STATE $\displaystyle{\mu \leftarrow \mu  + h \eta_{\mu} A\trsp{A} J_{\mu}}$ 
 
\STATE $\displaystyle{A \leftarrow A + \frac{h}{2} \eta_{\Sigma}  \trsp{(J_{\Sigma} -J_{\mu} \trsp{\mu})} A}$

\ENDFOR

\RETURN $\mu$, $A$
\end{algorithmic}
\end{algorithm}

\begin{algorithm}
\caption{Exact GIGO, one step. Not exactly our implementation, see the discussion after Corollary \ref{VraiesGeos}.}
\begin{algorithmic}
\label{algo2}
\STATE $1$. Compute the IGO speed:

\STATE $\displaystyle{ v_{\mu}=A\sum_{i=1}^{\lambda} w_{i} z_{i}},$

\STATE $\displaystyle{ v_{\Sigma}
=A \sum_{i=1}^{\lambda} w_{i} \left( z_{i} \trsp{ z_{i}} - I \right) \trsp{A} }.\newline$

\STATE $2$. Learning rates

\STATE  $\displaystyle{\lambda \leftarrow \sqrt{\frac{\eta_{\Sigma}}{\eta_\mu}} }$\\

\STATE $\displaystyle{\mu \leftarrow \lambda \mu }$\\

\STATE $\displaystyle{v_{\mu} \leftarrow \eta_\mu \lambda v_{\mu} }$\\

\STATE $\displaystyle{v_{\Sigma} \leftarrow \eta_\Sigma v_{\Sigma} }$\\

\STATE $3$. Intermediate computations.

\STATE $\displaystyle{G^{2} \leftarrow A^{-1}\left( v_{\Sigma} \trsp{(A^{-1})}A^{-1}v_\Sigma + 2 v_{\mu} \trsp{v_{\mu}}\right)\trsp{(A^{-1})}}$\\ 

\STATE $\displaystyle{C_{1} \leftarrow \mathrm{ch} (\frac{G}{2})}$\\

\STATE $\displaystyle{C_2 \leftarrow \mathrm{sh} (\frac{G}{2})G^{-1}}$\\

\STATE $\displaystyle{R \leftarrow \trsp{\left( (C_1 - A^{-1}v_{\Sigma}\trsp{(A^{-1})} C_2 )^{-1} \right)}}$\\

\STATE $4$. Actual update\\

\STATE $\displaystyle{\mu \leftarrow \mu + 2ARC_2 A^{-1} v_{\mu} }$ \\

\STATE $\displaystyle{A \leftarrow AR}$\\

\STATE $5$. Return to the ``real" $\mu$\\

\STATE $\displaystyle{\mu \leftarrow \frac{\mu}{\lambda}}$\\

\RETURN $\mu$, $A$
\end{algorithmic}
\end{algorithm}

\begin{algorithm}
\caption{xNES update, one step}
\begin{algorithmic}
\STATE $1$. Compute $G_{\mu}$ and $G_{M}$ (equivalent to the computation of the IGO speed):

\STATE $\displaystyle{ G_{\mu}=\sum_{i=1}^{\lambda} w_{i} z_{i} }$ 

\STATE $\displaystyle{ G_{M}=\sum_{i=1}^{\lambda} w_{i} \left( z_{i} \trsp{ z_{i}} - I \right) }\newline$

\STATE $2$. Actual update:

\STATE $\displaystyle{\mu \leftarrow \mu + \eta_{\mu}AG_{\mu} }$

\STATE $\displaystyle{A \leftarrow A + A\exp (\eta_{\Sigma}G_{M}/2) }\newline$

\RETURN $\mu$, $A$
\end{algorithmic}
\end{algorithm}
\clearpage
\begin{algorithm}

\caption{pure rank-$\mu$ CMA-ES update, one step}
\begin{algorithmic}
\label{algoN}

\STATE $1$. Computation of the IGO speed:

\STATE $\displaystyle{v_{\mu}=\sum_{i=1}^{\lambda} w_{i} (x_{i} - \mu )}$

\STATE $\displaystyle{v_{\Sigma}=\sum_{i=1}^{\lambda} w_{i} \left( (x_{i} - \mu )\trsp{ (x_{i} - \mu )} - \Sigma \right)}$

\STATE $2$. Actual update:

\STATE $\displaystyle{\mu \leftarrow \mu + \eta_{\mu} v_{\mu} }$

\STATE $\displaystyle{\Sigma \leftarrow \Sigma + \eta_{\Sigma} v_{\Sigma}}$

\RETURN $\mu$, $\Sigma$
\end{algorithmic}
\end{algorithm}
\clearpage

\begin{algorithm}[h]
\caption{GIGO in $\tGd$, one step}
\begin{algorithmic}
\STATE $1$. Compute the IGO speed:

\STATE $\displaystyle{Y_{\mu}=\sum_{i=1}^{\lambda} w_{i} (x_{i} - \mu )}$

\STATE $\displaystyle{Y_{\sigma}=\sum_{i=1}^{\lambda} w_{i} \left( \frac{\trsp{(x_{i}-\mu )}(x_{i}-\mu )}{2d\sigma} - \frac{\sigma}{2}    \right)}$

\STATE $2$. Better parametrization:

\STATE $\displaystyle{\lambda:=\sqrt{\frac{2d\eta_{\mu}}{ \eta_{\sigma}}}}$
\STATE $\displaystyle{v_{r}:=\frac{\eta_{\mu}}{\lambda} \Vert Y_{\mu} \Vert}$

\STATE $\displaystyle{v_{\sigma}:=\eta_{\sigma} Y_{\sigma}}$

\STATE $3$. Find $a,b,c,d,v$ corresponding to $\mu, \sigma, \dot{\mu},\dot{\sigma}$:
\STATE $\displaystyle{v=\sqrt{\frac{v_{r}^{2}+v_{\sigma}^{2}}{\sigma^{2}}}}$
\STATE $\displaystyle{ S_{0}:= \frac{v_{\sigma} }{v \sigma^{2}}}$
\STATE $\displaystyle{ M_{0}:= \frac{v_{r} }{v \sigma^{2}}}$

\STATE $\displaystyle{C:=\frac{ \sqrt{S_{0}^{2}+M_{0}^{2}} - S_{0}}{2}}$

\STATE $\displaystyle{D:=\frac{ \sqrt{S_{0}^{2}+M_{0}^{2}} + S_{0}}{2}}$
\STATE $\displaystyle{c:=\sqrt{C}}$

\STATE $\displaystyle{d:=  \sqrt{D}}$

\STATE $4$. Actual Update:

\STATE $\displaystyle{z:=\sigma \frac{die^{v \dt}-c}{cie^{v\dt}+d}}$

\STATE $\displaystyle{\mu := \mu +  \lambda \Re(z) \frac{Y_{\mu}}{\Vert Y_{\mu} \Vert}}$

\STATE $\displaystyle{\sigma := \Im(z)}$

\RETURN $\mu$, $\sigma$
\end{algorithmic}
\label{AlgoTGD}
\end{algorithm}
%

\nocite{IGO}
\nocite{DC}
\nocite{Ama}
\nocite{AM}
\nocite{Whi}
\nocite{HMK}
\nocite{SolGeo}
\nocite{Eriksen}
\nocite{GeoRemarks}
\newpage
\bibliographystyle{plain}
\bibliography{BibliGIGO}

\end{document}

%% file: GrGtex.tex
\begingroup
  \makeatletter
  \providecommand\color[2][]{%
    \GenericError{(gnuplot) \space\space\space\@spaces}{%
      Package color not loaded in conjunction with
      terminal option `colourtext'%
    }{See the gnuplot documentation for explanation.%
    }{Either use 'blacktext' in gnuplot or load the package
      color.sty in LaTeX.}%
    \renewcommand\color[2][]{}%
  }%
  \providecommand\includegraphics[2][]{%
    \GenericError{(gnuplot) \space\space\space\@spaces}{%
      Package graphicx or graphics not loaded%
    }{See the gnuplot documentation for explanation.%
    }{The gnuplot epslatex terminal needs graphicx.sty or graphics.sty.}%
    \renewcommand\includegraphics[2][]{}%
  }%
  \providecommand\rotatebox[2]{#2}%
  \@ifundefined{ifGPcolor}{%
    \newif\ifGPcolor
    \GPcolortrue
  }{}%
  \@ifundefined{ifGPblacktext}{%
    \newif\ifGPblacktext
    \GPblacktexttrue
  }{}%
  \let\gplgaddtomacro\g@addto@macro
  \gdef\gplbacktext{}%
  \gdef\gplfronttext{}%
  \makeatother
  \ifGPblacktext
    \def\colorrgb#1{}%
    \def\colorgray#1{}%
  \else
    \ifGPcolor
      \def\colorrgb#1{\color[rgb]{#1}}%
      \def\colorgray#1{\color[gray]{#1}}%
      \expandafter\def\csname LTw\endcsname{\color{white}}%
      \expandafter\def\csname LTb\endcsname{\color{black}}%
      \expandafter\def\csname LTa\endcsname{\color{black}}%
      \expandafter\def\csname LT0\endcsname{\color[rgb]{1,0,0}}%
      \expandafter\def\csname LT1\endcsname{\color[rgb]{0,1,0}}%
      \expandafter\def\csname LT2\endcsname{\color[rgb]{0,0,1}}%
      \expandafter\def\csname LT3\endcsname{\color[rgb]{1,0,1}}%
      \expandafter\def\csname LT4\endcsname{\color[rgb]{0,1,1}}%
      \expandafter\def\csname LT5\endcsname{\color[rgb]{1,1,0}}%
      \expandafter\def\csname LT6\endcsname{\color[rgb]{0,0,0}}%
      \expandafter\def\csname LT7\endcsname{\color[rgb]{1,0.3,0}}%
      \expandafter\def\csname LT8\endcsname{\color[rgb]{0.5,0.5,0.5}}%
    \else
      \def\colorrgb#1{\color{black}}%
      \def\colorgray#1{\color[gray]{#1}}%
      \expandafter\def\csname LTw\endcsname{\color{white}}%
      \expandafter\def\csname LTb\endcsname{\color{black}}%
      \expandafter\def\csname LTa\endcsname{\color{black}}%
      \expandafter\def\csname LT0\endcsname{\color{black}}%
      \expandafter\def\csname LT1\endcsname{\color{black}}%
      \expandafter\def\csname LT2\endcsname{\color{black}}%
      \expandafter\def\csname LT3\endcsname{\color{black}}%
      \expandafter\def\csname LT4\endcsname{\color{black}}%
      \expandafter\def\csname LT5\endcsname{\color{black}}%
      \expandafter\def\csname LT6\endcsname{\color{black}}%
      \expandafter\def\csname LT7\endcsname{\color{black}}%
      \expandafter\def\csname LT8\endcsname{\color{black}}%
    \fi
  \fi
  \setlength{\unitlength}{0.0500bp}%
  \begin{picture}(4896.00,2304.00)%
    \gplgaddtomacro\gplbacktext{%
      \csname LTb\endcsname%
      \put(4916,382){\makebox(0,0)[l]{\strut{}$\mu$}}%
      \put(1164,1980){\makebox(0,0)[l]{\strut{}$\sigma$}}%
      \put(469,1684){\makebox(0,0)[l]{\strut{}$\gamma_{1}$}}%
      \put(3000,1447){\makebox(0,0)[l]{\strut{}$\gamma_{2}$}}%
    }%
    \gplgaddtomacro\gplfronttext{%
    }%
    \gplbacktext
    \put(0,0){\includegraphics{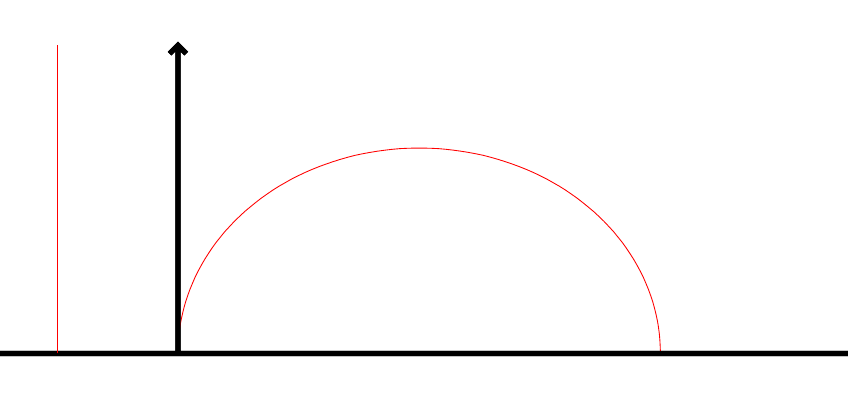}}%
    \gplfronttext
  \end{picture}%
\endgroup

%% file: GrapheGeotex.tex
\begingroup
  \makeatletter
  \providecommand\color[2][]{%
    \GenericError{(gnuplot) \space\space\space\@spaces}{%
      Package color not loaded in conjunction with
      terminal option `colourtext'%
    }{See the gnuplot documentation for explanation.%
    }{Either use 'blacktext' in gnuplot or load the package
      color.sty in LaTeX.}%
    \renewcommand\color[2][]{}%
  }%
  \providecommand\includegraphics[2][]{%
    \GenericError{(gnuplot) \space\space\space\@spaces}{%
      Package graphicx or graphics not loaded%
    }{See the gnuplot documentation for explanation.%
    }{The gnuplot epslatex terminal needs graphicx.sty or graphics.sty.}%
    \renewcommand\includegraphics[2][]{}%
  }%
  \providecommand\rotatebox[2]{#2}%
  \@ifundefined{ifGPcolor}{%
    \newif\ifGPcolor
    \GPcolortrue
  }{}%
  \@ifundefined{ifGPblacktext}{%
    \newif\ifGPblacktext
    \GPblacktexttrue
  }{}%
  \let\gplgaddtomacro\g@addto@macro
  \gdef\gplbacktext{}%
  \gdef\gplfronttext{}%
  \makeatother
  \ifGPblacktext
    \def\colorrgb#1{}%
    \def\colorgray#1{}%
  \else
    \ifGPcolor
      \def\colorrgb#1{\color[rgb]{#1}}%
      \def\colorgray#1{\color[gray]{#1}}%
      \expandafter\def\csname LTw\endcsname{\color{white}}%
      \expandafter\def\csname LTb\endcsname{\color{black}}%
      \expandafter\def\csname LTa\endcsname{\color{black}}%
      \expandafter\def\csname LT0\endcsname{\color[rgb]{1,0,0}}%
      \expandafter\def\csname LT1\endcsname{\color[rgb]{0,1,0}}%
      \expandafter\def\csname LT2\endcsname{\color[rgb]{0,0,1}}%
      \expandafter\def\csname LT3\endcsname{\color[rgb]{1,0,1}}%
      \expandafter\def\csname LT4\endcsname{\color[rgb]{0,1,1}}%
      \expandafter\def\csname LT5\endcsname{\color[rgb]{1,1,0}}%
      \expandafter\def\csname LT6\endcsname{\color[rgb]{0,0,0}}%
      \expandafter\def\csname LT7\endcsname{\color[rgb]{1,0.3,0}}%
      \expandafter\def\csname LT8\endcsname{\color[rgb]{0.5,0.5,0.5}}%
    \else
      \def\colorrgb#1{\color{black}}%
      \def\colorgray#1{\color[gray]{#1}}%
      \expandafter\def\csname LTw\endcsname{\color{white}}%
      \expandafter\def\csname LTb\endcsname{\color{black}}%
      \expandafter\def\csname LTa\endcsname{\color{black}}%
      \expandafter\def\csname LT0\endcsname{\color{black}}%
      \expandafter\def\csname LT1\endcsname{\color{black}}%
      \expandafter\def\csname LT2\endcsname{\color{black}}%
      \expandafter\def\csname LT3\endcsname{\color{black}}%
      \expandafter\def\csname LT4\endcsname{\color{black}}%
      \expandafter\def\csname LT5\endcsname{\color{black}}%
      \expandafter\def\csname LT6\endcsname{\color{black}}%
      \expandafter\def\csname LT7\endcsname{\color{black}}%
      \expandafter\def\csname LT8\endcsname{\color{black}}%
    \fi
  \fi
  \setlength{\unitlength}{0.0500bp}%
  \begin{picture}(4608.00,2304.00)%
    \gplgaddtomacro\gplbacktext{%
      \csname LTb\endcsname%
      \put(3111,1270){\makebox(0,0)[l]{\strut{}$\theta^{t+dt}$}}%
      \put(2141,1566){\makebox(0,0)[l]{\strut{}$\theta^{t}$}}%
      \put(4599,382){\makebox(0,0)[l]{\strut{}$\mu$}}%
      \put(1106,1980){\makebox(0,0)[l]{\strut{}$\sigma$}}%
      \put(2723,1566){\makebox(0,0)[l]{\strut{}Y}}%
    }%
    \gplgaddtomacro\gplfronttext{%
    }%
    \gplgaddtomacro\gplfronttext{%
    }%
    \gplbacktext
    \put(0,0){\includegraphics{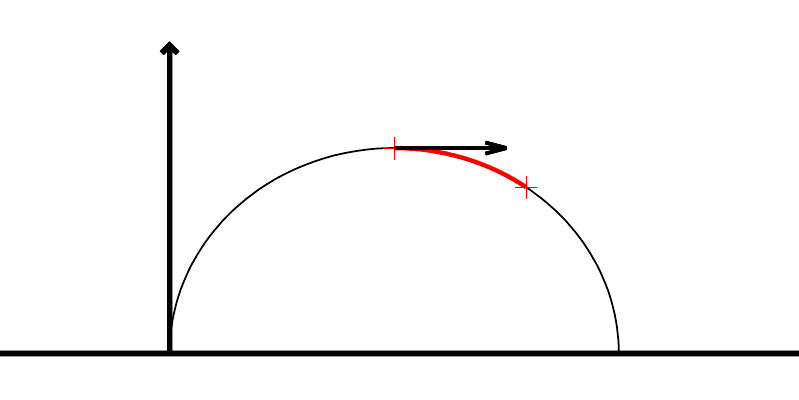}}%
    \gplfronttext
  \end{picture}%
\endgroup

%% file: GrapheCTtex.tex
\begingroup
  \makeatletter
  \providecommand\color[2][]{%
    \GenericError{(gnuplot) \space\space\space\@spaces}{%
      Package color not loaded in conjunction with
      terminal option `colourtext'%
    }{See the gnuplot documentation for explanation.%
    }{Either use 'blacktext' in gnuplot or load the package
      color.sty in LaTeX.}%
    \renewcommand\color[2][]{}%
  }%
  \providecommand\includegraphics[2][]{%
    \GenericError{(gnuplot) \space\space\space\@spaces}{%
      Package graphicx or graphics not loaded%
    }{See the gnuplot documentation for explanation.%
    }{The gnuplot epslatex terminal needs graphicx.sty or graphics.sty.}%
    \renewcommand\includegraphics[2][]{}%
  }%
  \providecommand\rotatebox[2]{#2}%
  \@ifundefined{ifGPcolor}{%
    \newif\ifGPcolor
    \GPcolortrue
  }{}%
  \@ifundefined{ifGPblacktext}{%
    \newif\ifGPblacktext
    \GPblacktexttrue
  }{}%
  \let\gplgaddtomacro\g@addto@macro
  \gdef\gplbacktext{}%
  \gdef\gplfronttext{}%
  \makeatother
  \ifGPblacktext
    \def\colorrgb#1{}%
    \def\colorgray#1{}%
  \else
    \ifGPcolor
      \def\colorrgb#1{\color[rgb]{#1}}%
      \def\colorgray#1{\color[gray]{#1}}%
      \expandafter\def\csname LTw\endcsname{\color{white}}%
      \expandafter\def\csname LTb\endcsname{\color{black}}%
      \expandafter\def\csname LTa\endcsname{\color{black}}%
      \expandafter\def\csname LT0\endcsname{\color[rgb]{1,0,0}}%
      \expandafter\def\csname LT1\endcsname{\color[rgb]{0,1,0}}%
      \expandafter\def\csname LT2\endcsname{\color[rgb]{0,0,1}}%
      \expandafter\def\csname LT3\endcsname{\color[rgb]{1,0,1}}%
      \expandafter\def\csname LT4\endcsname{\color[rgb]{0,1,1}}%
      \expandafter\def\csname LT5\endcsname{\color[rgb]{1,1,0}}%
      \expandafter\def\csname LT6\endcsname{\color[rgb]{0,0,0}}%
      \expandafter\def\csname LT7\endcsname{\color[rgb]{1,0.3,0}}%
      \expandafter\def\csname LT8\endcsname{\color[rgb]{0.5,0.5,0.5}}%
    \else
      \def\colorrgb#1{\color{black}}%
      \def\colorgray#1{\color[gray]{#1}}%
      \expandafter\def\csname LTw\endcsname{\color{white}}%
      \expandafter\def\csname LTb\endcsname{\color{black}}%
      \expandafter\def\csname LTa\endcsname{\color{black}}%
      \expandafter\def\csname LT0\endcsname{\color{black}}%
      \expandafter\def\csname LT1\endcsname{\color{black}}%
      \expandafter\def\csname LT2\endcsname{\color{black}}%
      \expandafter\def\csname LT3\endcsname{\color{black}}%
      \expandafter\def\csname LT4\endcsname{\color{black}}%
      \expandafter\def\csname LT5\endcsname{\color{black}}%
      \expandafter\def\csname LT6\endcsname{\color{black}}%
      \expandafter\def\csname LT7\endcsname{\color{black}}%
      \expandafter\def\csname LT8\endcsname{\color{black}}%
    \fi
  \fi
  \setlength{\unitlength}{0.0500bp}%
  \begin{picture}(3600.00,5040.00)%
    \gplgaddtomacro\gplbacktext{%
      \csname LTb\endcsname%
      \put(594,440){\makebox(0,0)[r]{\strut{}$10^2$}}%
      \put(594,1425){\makebox(0,0)[r]{\strut{}$10^3$}}%
      \put(594,2410){\makebox(0,0)[r]{\strut{}$10^4$}}%
      \put(594,3394){\makebox(0,0)[r]{\strut{}$10^5$}}%
      \put(1130,220){\makebox(0,0){\strut{}2}}%
      \put(1534,220){\makebox(0,0){\strut{}4}}%
      \put(1938,220){\makebox(0,0){\strut{}8}}%
      \put(2343,220){\makebox(0,0){\strut{}16}}%
      \put(2747,220){\makebox(0,0){\strut{}32}}%
      \put(3151,220){\makebox(0,0){\strut{}64}}%
      \put(1964,4709){\makebox(0,0){\strut{}Cigar-tablet}}%
    }%
    \gplgaddtomacro\gplfronttext{%
      \csname LTb\endcsname%
      \put(2046,4206){\makebox(0,0)[r]{\strut{}CMA}}%
      \csname LTb\endcsname%
      \put(2046,3986){\makebox(0,0)[r]{\strut{}GIGO}}%
      \csname LTb\endcsname%
      \put(2046,3766){\makebox(0,0)[r]{\strut{}xNES}}%
    }%
    \gplbacktext
    \put(0,0){\includegraphics{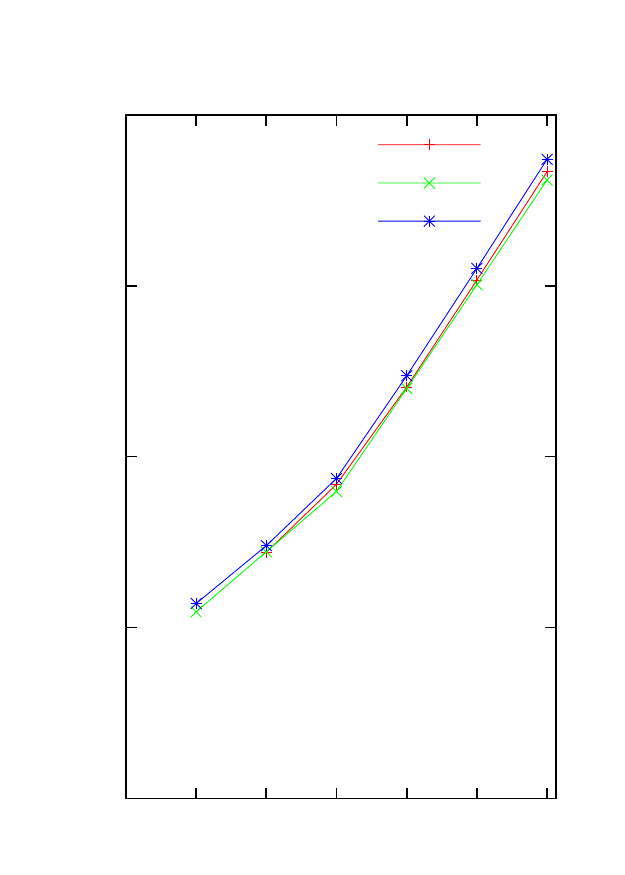}}%
    \gplfronttext
  \end{picture}%
\endgroup

%% file: GrapheSphtex.tex
\begingroup
  \makeatletter
  \providecommand\color[2][]{%
    \GenericError{(gnuplot) \space\space\space\@spaces}{%
      Package color not loaded in conjunction with
      terminal option `colourtext'%
    }{See the gnuplot documentation for explanation.%
    }{Either use 'blacktext' in gnuplot or load the package
      color.sty in LaTeX.}%
    \renewcommand\color[2][]{}%
  }%
  \providecommand\includegraphics[2][]{%
    \GenericError{(gnuplot) \space\space\space\@spaces}{%
      Package graphicx or graphics not loaded%
    }{See the gnuplot documentation for explanation.%
    }{The gnuplot epslatex terminal needs graphicx.sty or graphics.sty.}%
    \renewcommand\includegraphics[2][]{}%
  }%
  \providecommand\rotatebox[2]{#2}%
  \@ifundefined{ifGPcolor}{%
    \newif\ifGPcolor
    \GPcolortrue
  }{}%
  \@ifundefined{ifGPblacktext}{%
    \newif\ifGPblacktext
    \GPblacktexttrue
  }{}%
  \let\gplgaddtomacro\g@addto@macro
  \gdef\gplbacktext{}%
  \gdef\gplfronttext{}%
  \makeatother
  \ifGPblacktext
    \def\colorrgb#1{}%
    \def\colorgray#1{}%
  \else
    \ifGPcolor
      \def\colorrgb#1{\color[rgb]{#1}}%
      \def\colorgray#1{\color[gray]{#1}}%
      \expandafter\def\csname LTw\endcsname{\color{white}}%
      \expandafter\def\csname LTb\endcsname{\color{black}}%
      \expandafter\def\csname LTa\endcsname{\color{black}}%
      \expandafter\def\csname LT0\endcsname{\color[rgb]{1,0,0}}%
      \expandafter\def\csname LT1\endcsname{\color[rgb]{0,1,0}}%
      \expandafter\def\csname LT2\endcsname{\color[rgb]{0,0,1}}%
      \expandafter\def\csname LT3\endcsname{\color[rgb]{1,0,1}}%
      \expandafter\def\csname LT4\endcsname{\color[rgb]{0,1,1}}%
      \expandafter\def\csname LT5\endcsname{\color[rgb]{1,1,0}}%
      \expandafter\def\csname LT6\endcsname{\color[rgb]{0,0,0}}%
      \expandafter\def\csname LT7\endcsname{\color[rgb]{1,0.3,0}}%
      \expandafter\def\csname LT8\endcsname{\color[rgb]{0.5,0.5,0.5}}%
    \else
      \def\colorrgb#1{\color{black}}%
      \def\colorgray#1{\color[gray]{#1}}%
      \expandafter\def\csname LTw\endcsname{\color{white}}%
      \expandafter\def\csname LTb\endcsname{\color{black}}%
      \expandafter\def\csname LTa\endcsname{\color{black}}%
      \expandafter\def\csname LT0\endcsname{\color{black}}%
      \expandafter\def\csname LT1\endcsname{\color{black}}%
      \expandafter\def\csname LT2\endcsname{\color{black}}%
      \expandafter\def\csname LT3\endcsname{\color{black}}%
      \expandafter\def\csname LT4\endcsname{\color{black}}%
      \expandafter\def\csname LT5\endcsname{\color{black}}%
      \expandafter\def\csname LT6\endcsname{\color{black}}%
      \expandafter\def\csname LT7\endcsname{\color{black}}%
      \expandafter\def\csname LT8\endcsname{\color{black}}%
    \fi
  \fi
  \setlength{\unitlength}{0.0500bp}%
  \begin{picture}(3600.00,5040.00)%
    \gplgaddtomacro\gplbacktext{%
      \csname LTb\endcsname%
      \put(594,440){\makebox(0,0)[r]{\strut{}$10^2$}}%
      \put(594,1425){\makebox(0,0)[r]{\strut{}$10^3$}}%
      \put(594,2410){\makebox(0,0)[r]{\strut{}$10^4$}}%
      \put(594,3394){\makebox(0,0)[r]{\strut{}$10^5$}}%
      \put(1130,220){\makebox(0,0){\strut{}2}}%
      \put(1534,220){\makebox(0,0){\strut{}4}}%
      \put(1938,220){\makebox(0,0){\strut{}8}}%
      \put(2343,220){\makebox(0,0){\strut{}16}}%
      \put(2747,220){\makebox(0,0){\strut{}32}}%
      \put(3151,220){\makebox(0,0){\strut{}64}}%
      \put(1964,4709){\makebox(0,0){\strut{}Sphere}}%
    }%
    \gplgaddtomacro\gplfronttext{%
      \csname LTb\endcsname%
      \put(2178,4206){\makebox(0,0)[r]{\strut{}CMA}}%
      \csname LTb\endcsname%
      \put(2178,3986){\makebox(0,0)[r]{\strut{}GIGO}}%
      \csname LTb\endcsname%
      \put(2178,3766){\makebox(0,0)[r]{\strut{}xNES}}%
    }%
    \gplbacktext
    \put(0,0){\includegraphics{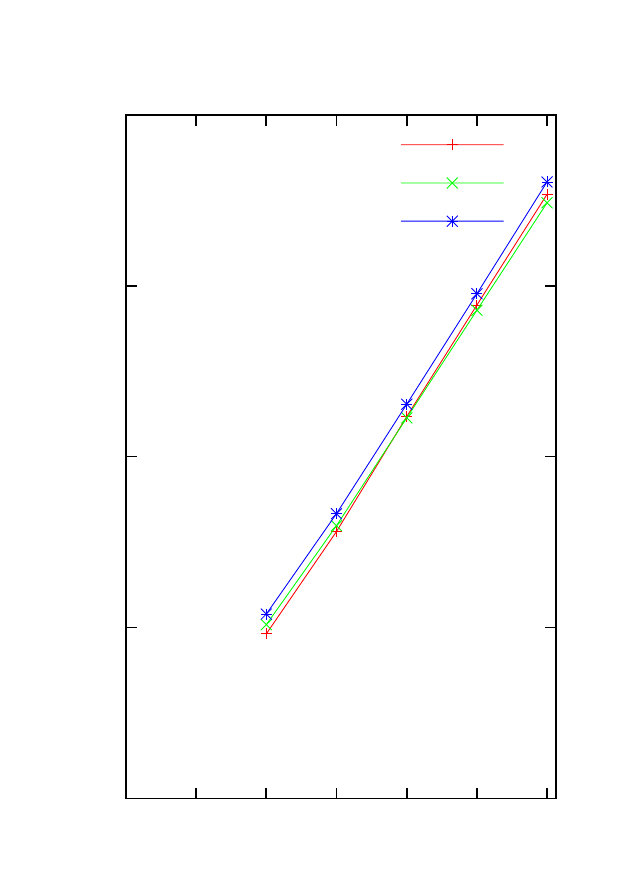}}%
    \gplfronttext
  \end{picture}%
\endgroup

%% file: GrapheRostex.tex
\begingroup
  \makeatletter
  \providecommand\color[2][]{%
    \GenericError{(gnuplot) \space\space\space\@spaces}{%
      Package color not loaded in conjunction with
      terminal option `colourtext'%
    }{See the gnuplot documentation for explanation.%
    }{Either use 'blacktext' in gnuplot or load the package
      color.sty in LaTeX.}%
    \renewcommand\color[2][]{}%
  }%
  \providecommand\includegraphics[2][]{%
    \GenericError{(gnuplot) \space\space\space\@spaces}{%
      Package graphicx or graphics not loaded%
    }{See the gnuplot documentation for explanation.%
    }{The gnuplot epslatex terminal needs graphicx.sty or graphics.sty.}%
    \renewcommand\includegraphics[2][]{}%
  }%
  \providecommand\rotatebox[2]{#2}%
  \@ifundefined{ifGPcolor}{%
    \newif\ifGPcolor
    \GPcolortrue
  }{}%
  \@ifundefined{ifGPblacktext}{%
    \newif\ifGPblacktext
    \GPblacktexttrue
  }{}%
  \let\gplgaddtomacro\g@addto@macro
  \gdef\gplbacktext{}%
  \gdef\gplfronttext{}%
  \makeatother
  \ifGPblacktext
    \def\colorrgb#1{}%
    \def\colorgray#1{}%
  \else
    \ifGPcolor
      \def\colorrgb#1{\color[rgb]{#1}}%
      \def\colorgray#1{\color[gray]{#1}}%
      \expandafter\def\csname LTw\endcsname{\color{white}}%
      \expandafter\def\csname LTb\endcsname{\color{black}}%
      \expandafter\def\csname LTa\endcsname{\color{black}}%
      \expandafter\def\csname LT0\endcsname{\color[rgb]{1,0,0}}%
      \expandafter\def\csname LT1\endcsname{\color[rgb]{0,1,0}}%
      \expandafter\def\csname LT2\endcsname{\color[rgb]{0,0,1}}%
      \expandafter\def\csname LT3\endcsname{\color[rgb]{1,0,1}}%
      \expandafter\def\csname LT4\endcsname{\color[rgb]{0,1,1}}%
      \expandafter\def\csname LT5\endcsname{\color[rgb]{1,1,0}}%
      \expandafter\def\csname LT6\endcsname{\color[rgb]{0,0,0}}%
      \expandafter\def\csname LT7\endcsname{\color[rgb]{1,0.3,0}}%
      \expandafter\def\csname LT8\endcsname{\color[rgb]{0.5,0.5,0.5}}%
    \else
      \def\colorrgb#1{\color{black}}%
      \def\colorgray#1{\color[gray]{#1}}%
      \expandafter\def\csname LTw\endcsname{\color{white}}%
      \expandafter\def\csname LTb\endcsname{\color{black}}%
      \expandafter\def\csname LTa\endcsname{\color{black}}%
      \expandafter\def\csname LT0\endcsname{\color{black}}%
      \expandafter\def\csname LT1\endcsname{\color{black}}%
      \expandafter\def\csname LT2\endcsname{\color{black}}%
      \expandafter\def\csname LT3\endcsname{\color{black}}%
      \expandafter\def\csname LT4\endcsname{\color{black}}%
      \expandafter\def\csname LT5\endcsname{\color{black}}%
      \expandafter\def\csname LT6\endcsname{\color{black}}%
      \expandafter\def\csname LT7\endcsname{\color{black}}%
      \expandafter\def\csname LT8\endcsname{\color{black}}%
    \fi
  \fi
  \setlength{\unitlength}{0.0500bp}%
  \begin{picture}(3600.00,5040.00)%
    \gplgaddtomacro\gplbacktext{%
      \csname LTb\endcsname%
      \put(594,440){\makebox(0,0)[r]{\strut{}$10^2$}}%
      \put(594,1425){\makebox(0,0)[r]{\strut{}$10^3$}}%
      \put(594,2410){\makebox(0,0)[r]{\strut{}$10^4$}}%
      \put(594,3394){\makebox(0,0)[r]{\strut{}$10^5$}}%
      \put(1130,220){\makebox(0,0){\strut{}2}}%
      \put(1534,220){\makebox(0,0){\strut{}4}}%
      \put(1938,220){\makebox(0,0){\strut{}8}}%
      \put(2343,220){\makebox(0,0){\strut{}16}}%
      \put(2747,220){\makebox(0,0){\strut{}32}}%
      \put(3151,220){\makebox(0,0){\strut{}64}}%
      \put(1964,4709){\makebox(0,0){\strut{}Rosenbrock}}%
    }%
    \gplgaddtomacro\gplfronttext{%
      \csname LTb\endcsname%
      \put(2178,4206){\makebox(0,0)[r]{\strut{}CMA}}%
      \csname LTb\endcsname%
      \put(2178,3986){\makebox(0,0)[r]{\strut{}xNES}}%
    }%
    \gplbacktext
    \put(0,0){\includegraphics{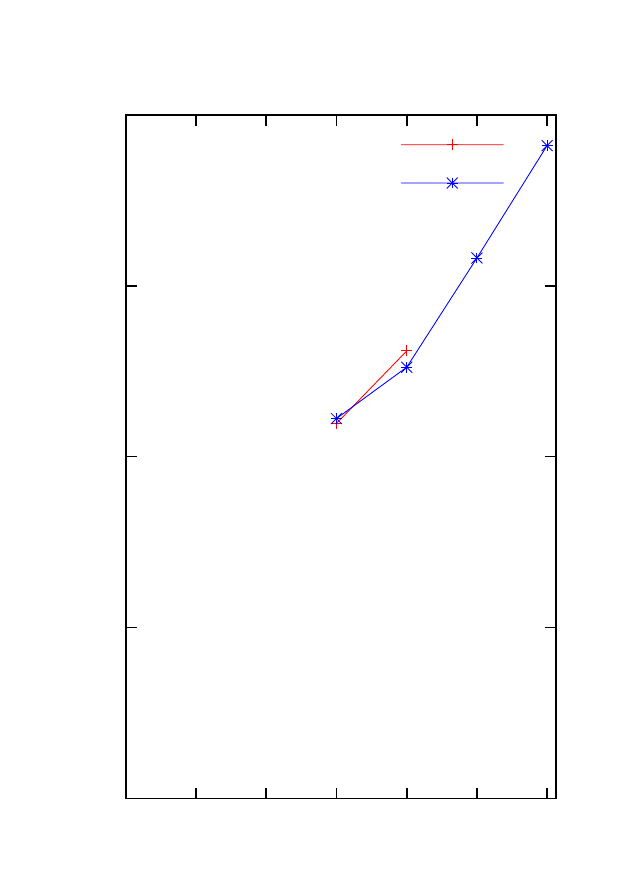}}%
    \gplfronttext
  \end{picture}%
\endgroup

%% file: Graphe001tex.tex
\begingroup
  \makeatletter
  \providecommand\color[2][]{%
    \GenericError{(gnuplot) \space\space\space\@spaces}{%
      Package color not loaded in conjunction with
      terminal option `colourtext'%
    }{See the gnuplot documentation for explanation.%
    }{Either use 'blacktext' in gnuplot or load the package
      color.sty in LaTeX.}%
    \renewcommand\color[2][]{}%
  }%
  \providecommand\includegraphics[2][]{%
    \GenericError{(gnuplot) \space\space\space\@spaces}{%
      Package graphicx or graphics not loaded%
    }{See the gnuplot documentation for explanation.%
    }{The gnuplot epslatex terminal needs graphicx.sty or graphics.sty.}%
    \renewcommand\includegraphics[2][]{}%
  }%
  \providecommand\rotatebox[2]{#2}%
  \@ifundefined{ifGPcolor}{%
    \newif\ifGPcolor
    \GPcolortrue
  }{}%
  \@ifundefined{ifGPblacktext}{%
    \newif\ifGPblacktext
    \GPblacktexttrue
  }{}%
  \let\gplgaddtomacro\g@addto@macro
  \gdef\gplbacktext{}%
  \gdef\gplfronttext{}%
  \makeatother
  \ifGPblacktext
    \def\colorrgb#1{}%
    \def\colorgray#1{}%
  \else
    \ifGPcolor
      \def\colorrgb#1{\color[rgb]{#1}}%
      \def\colorgray#1{\color[gray]{#1}}%
      \expandafter\def\csname LTw\endcsname{\color{white}}%
      \expandafter\def\csname LTb\endcsname{\color{black}}%
      \expandafter\def\csname LTa\endcsname{\color{black}}%
      \expandafter\def\csname LT0\endcsname{\color[rgb]{1,0,0}}%
      \expandafter\def\csname LT1\endcsname{\color[rgb]{0,1,0}}%
      \expandafter\def\csname LT2\endcsname{\color[rgb]{0,0,1}}%
      \expandafter\def\csname LT3\endcsname{\color[rgb]{1,0,1}}%
      \expandafter\def\csname LT4\endcsname{\color[rgb]{0,1,1}}%
      \expandafter\def\csname LT5\endcsname{\color[rgb]{1,1,0}}%
      \expandafter\def\csname LT6\endcsname{\color[rgb]{0,0,0}}%
      \expandafter\def\csname LT7\endcsname{\color[rgb]{1,0.3,0}}%
      \expandafter\def\csname LT8\endcsname{\color[rgb]{0.5,0.5,0.5}}%
    \else
      \def\colorrgb#1{\color{black}}%
      \def\colorgray#1{\color[gray]{#1}}%
      \expandafter\def\csname LTw\endcsname{\color{white}}%
      \expandafter\def\csname LTb\endcsname{\color{black}}%
      \expandafter\def\csname LTa\endcsname{\color{black}}%
      \expandafter\def\csname LT0\endcsname{\color{black}}%
      \expandafter\def\csname LT1\endcsname{\color{black}}%
      \expandafter\def\csname LT2\endcsname{\color{black}}%
      \expandafter\def\csname LT3\endcsname{\color{black}}%
      \expandafter\def\csname LT4\endcsname{\color{black}}%
      \expandafter\def\csname LT5\endcsname{\color{black}}%
      \expandafter\def\csname LT6\endcsname{\color{black}}%
      \expandafter\def\csname LT7\endcsname{\color{black}}%
      \expandafter\def\csname LT8\endcsname{\color{black}}%
    \fi
  \fi
  \setlength{\unitlength}{0.0500bp}%
  \begin{picture}(7200.00,5040.00)%
    \gplgaddtomacro\gplbacktext{%
      \csname LTb\endcsname%
      \put(946,704){\makebox(0,0)[r]{\strut{} 0}}%
      \put(946,1213){\makebox(0,0)[r]{\strut{} 0.5}}%
      \put(946,1722){\makebox(0,0)[r]{\strut{} 1}}%
      \put(946,2231){\makebox(0,0)[r]{\strut{} 1.5}}%
      \put(946,2740){\makebox(0,0)[r]{\strut{} 2}}%
      \put(946,3248){\makebox(0,0)[r]{\strut{} 2.5}}%
      \put(946,3757){\makebox(0,0)[r]{\strut{} 3}}%
      \put(946,4266){\makebox(0,0)[r]{\strut{} 3.5}}%
      \put(946,4775){\makebox(0,0)[r]{\strut{} 4}}%
      \put(1078,484){\makebox(0,0){\strut{} 0}}%
      \put(2223,484){\makebox(0,0){\strut{} 2}}%
      \put(3368,484){\makebox(0,0){\strut{} 4}}%
      \put(4513,484){\makebox(0,0){\strut{} 6}}%
      \put(5658,484){\makebox(0,0){\strut{} 8}}%
      \put(6803,484){\makebox(0,0){\strut{} 10}}%
      \put(176,2739){{\makebox(0,0){\strut{}$\sigma$}}}%
      \put(3940,154){\makebox(0,0){\strut{}$\mu$}}%
    }%
    \gplgaddtomacro\gplfronttext{%
      \csname LTb\endcsname%
      \put(5816,4602){\makebox(0,0)[r]{\strut{}GIGO}}%
      \csname LTb\endcsname%
      \put(5816,4382){\makebox(0,0)[r]{\strut{}CMA}}%
      \csname LTb\endcsname%
      \put(5816,4162){\makebox(0,0)[r]{\strut{}xNES}}%
    }%
    \gplbacktext
    \put(0,0){\includegraphics{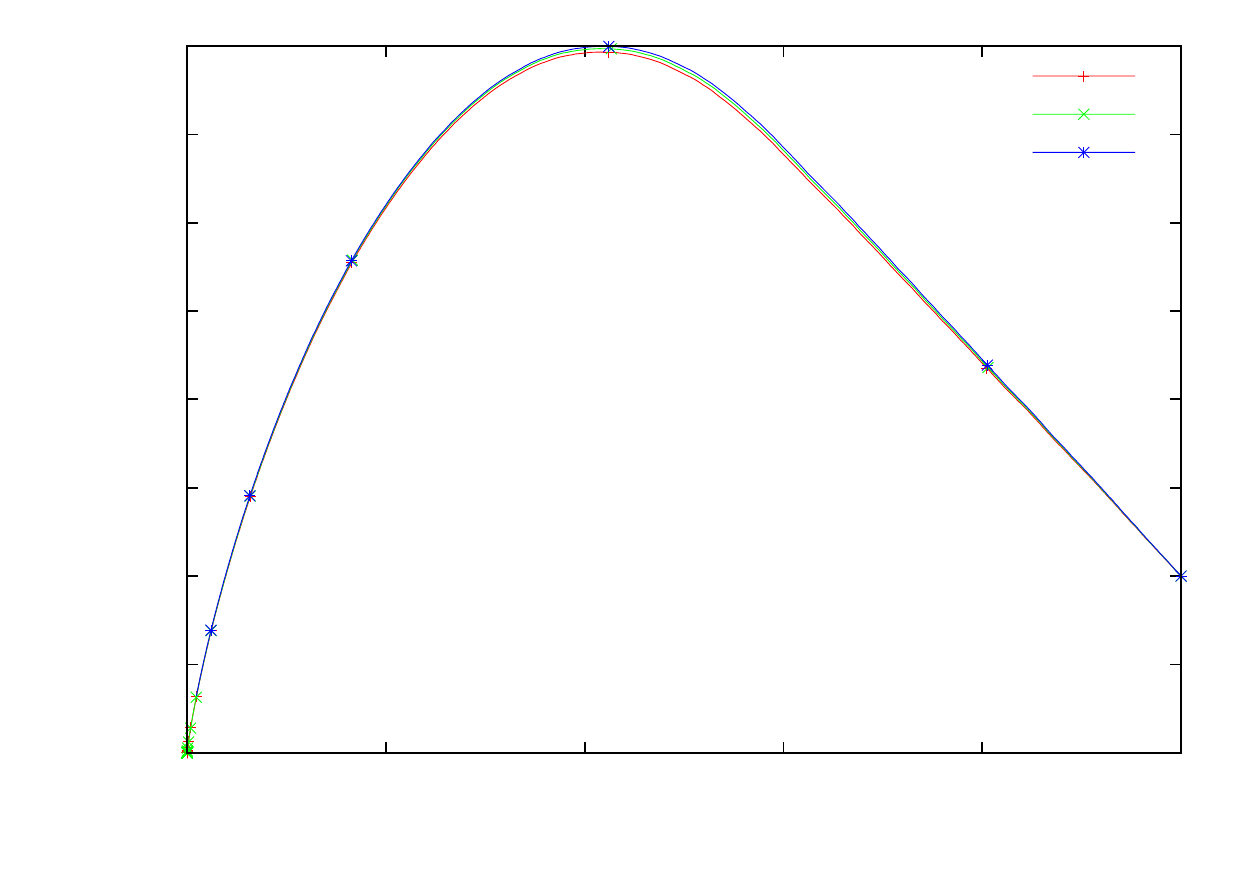}}%
    \gplfronttext
  \end{picture}%
\endgroup

%% file: Graphe01tex.tex
\begingroup
  \makeatletter
  \providecommand\color[2][]{%
    \GenericError{(gnuplot) \space\space\space\@spaces}{%
      Package color not loaded in conjunction with
      terminal option `colourtext'%
    }{See the gnuplot documentation for explanation.%
    }{Either use 'blacktext' in gnuplot or load the package
      color.sty in LaTeX.}%
    \renewcommand\color[2][]{}%
  }%
  \providecommand\includegraphics[2][]{%
    \GenericError{(gnuplot) \space\space\space\@spaces}{%
      Package graphicx or graphics not loaded%
    }{See the gnuplot documentation for explanation.%
    }{The gnuplot epslatex terminal needs graphicx.sty or graphics.sty.}%
    \renewcommand\includegraphics[2][]{}%
  }%
  \providecommand\rotatebox[2]{#2}%
  \@ifundefined{ifGPcolor}{%
    \newif\ifGPcolor
    \GPcolortrue
  }{}%
  \@ifundefined{ifGPblacktext}{%
    \newif\ifGPblacktext
    \GPblacktexttrue
  }{}%
  \let\gplgaddtomacro\g@addto@macro
  \gdef\gplbacktext{}%
  \gdef\gplfronttext{}%
  \makeatother
  \ifGPblacktext
    \def\colorrgb#1{}%
    \def\colorgray#1{}%
  \else
    \ifGPcolor
      \def\colorrgb#1{\color[rgb]{#1}}%
      \def\colorgray#1{\color[gray]{#1}}%
      \expandafter\def\csname LTw\endcsname{\color{white}}%
      \expandafter\def\csname LTb\endcsname{\color{black}}%
      \expandafter\def\csname LTa\endcsname{\color{black}}%
      \expandafter\def\csname LT0\endcsname{\color[rgb]{1,0,0}}%
      \expandafter\def\csname LT1\endcsname{\color[rgb]{0,1,0}}%
      \expandafter\def\csname LT2\endcsname{\color[rgb]{0,0,1}}%
      \expandafter\def\csname LT3\endcsname{\color[rgb]{1,0,1}}%
      \expandafter\def\csname LT4\endcsname{\color[rgb]{0,1,1}}%
      \expandafter\def\csname LT5\endcsname{\color[rgb]{1,1,0}}%
      \expandafter\def\csname LT6\endcsname{\color[rgb]{0,0,0}}%
      \expandafter\def\csname LT7\endcsname{\color[rgb]{1,0.3,0}}%
      \expandafter\def\csname LT8\endcsname{\color[rgb]{0.5,0.5,0.5}}%
    \else
      \def\colorrgb#1{\color{black}}%
      \def\colorgray#1{\color[gray]{#1}}%
      \expandafter\def\csname LTw\endcsname{\color{white}}%
      \expandafter\def\csname LTb\endcsname{\color{black}}%
      \expandafter\def\csname LTa\endcsname{\color{black}}%
      \expandafter\def\csname LT0\endcsname{\color{black}}%
      \expandafter\def\csname LT1\endcsname{\color{black}}%
      \expandafter\def\csname LT2\endcsname{\color{black}}%
      \expandafter\def\csname LT3\endcsname{\color{black}}%
      \expandafter\def\csname LT4\endcsname{\color{black}}%
      \expandafter\def\csname LT5\endcsname{\color{black}}%
      \expandafter\def\csname LT6\endcsname{\color{black}}%
      \expandafter\def\csname LT7\endcsname{\color{black}}%
      \expandafter\def\csname LT8\endcsname{\color{black}}%
    \fi
  \fi
  \setlength{\unitlength}{0.0500bp}%
  \begin{picture}(7200.00,5040.00)%
    \gplgaddtomacro\gplbacktext{%
      \csname LTb\endcsname%
      \put(946,704){\makebox(0,0)[r]{\strut{} 0}}%
      \put(946,1156){\makebox(0,0)[r]{\strut{} 0.5}}%
      \put(946,1609){\makebox(0,0)[r]{\strut{} 1}}%
      \put(946,2061){\makebox(0,0)[r]{\strut{} 1.5}}%
      \put(946,2513){\makebox(0,0)[r]{\strut{} 2}}%
      \put(946,2966){\makebox(0,0)[r]{\strut{} 2.5}}%
      \put(946,3418){\makebox(0,0)[r]{\strut{} 3}}%
      \put(946,3870){\makebox(0,0)[r]{\strut{} 3.5}}%
      \put(946,4323){\makebox(0,0)[r]{\strut{} 4}}%
      \put(946,4775){\makebox(0,0)[r]{\strut{} 4.5}}%
      \put(1078,484){\makebox(0,0){\strut{}-2}}%
      \put(2032,484){\makebox(0,0){\strut{} 0}}%
      \put(2986,484){\makebox(0,0){\strut{} 2}}%
      \put(3940,484){\makebox(0,0){\strut{} 4}}%
      \put(4895,484){\makebox(0,0){\strut{} 6}}%
      \put(5849,484){\makebox(0,0){\strut{} 8}}%
      \put(6803,484){\makebox(0,0){\strut{} 10}}%
      \put(176,2739){{\makebox(0,0){\strut{}$\sigma$}}}%
      \put(3940,154){\makebox(0,0){\strut{}$\mu$}}%
    }%
    \gplgaddtomacro\gplfronttext{%
      \csname LTb\endcsname%
      \put(5816,4602){\makebox(0,0)[r]{\strut{}GIGO}}%
      \csname LTb\endcsname%
      \put(5816,4382){\makebox(0,0)[r]{\strut{}CMA}}%
      \csname LTb\endcsname%
      \put(5816,4162){\makebox(0,0)[r]{\strut{}xNES}}%
    }%
    \gplbacktext
    \put(0,0){\includegraphics{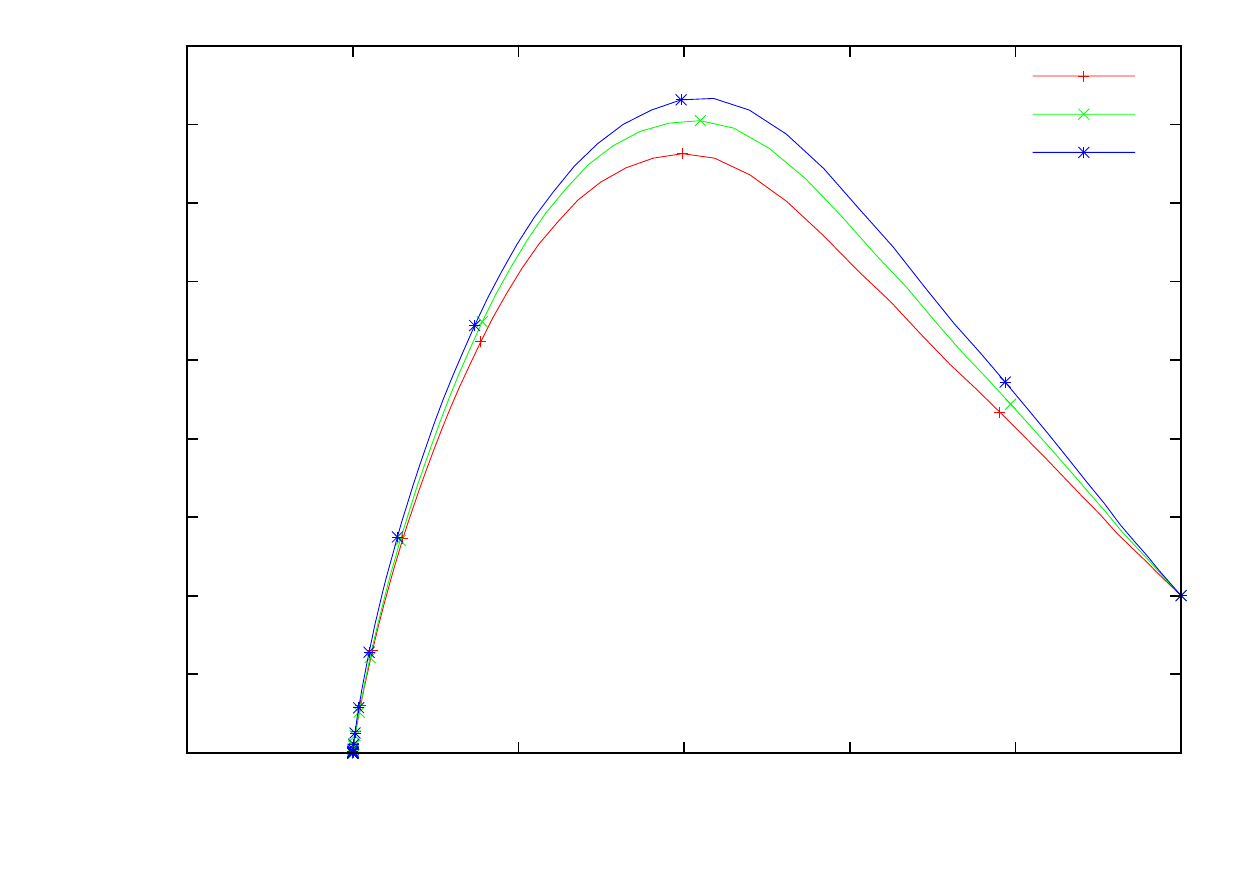}}%
    \gplfronttext
  \end{picture}%
\endgroup

%% file: Graphe05tex.tex
\begingroup
  \makeatletter
  \providecommand\color[2][]{%
    \GenericError{(gnuplot) \space\space\space\@spaces}{%
      Package color not loaded in conjunction with
      terminal option `colourtext'%
    }{See the gnuplot documentation for explanation.%
    }{Either use 'blacktext' in gnuplot or load the package
      color.sty in LaTeX.}%
    \renewcommand\color[2][]{}%
  }%
  \providecommand\includegraphics[2][]{%
    \GenericError{(gnuplot) \space\space\space\@spaces}{%
      Package graphicx or graphics not loaded%
    }{See the gnuplot documentation for explanation.%
    }{The gnuplot epslatex terminal needs graphicx.sty or graphics.sty.}%
    \renewcommand\includegraphics[2][]{}%
  }%
  \providecommand\rotatebox[2]{#2}%
  \@ifundefined{ifGPcolor}{%
    \newif\ifGPcolor
    \GPcolortrue
  }{}%
  \@ifundefined{ifGPblacktext}{%
    \newif\ifGPblacktext
    \GPblacktexttrue
  }{}%
  \let\gplgaddtomacro\g@addto@macro
  \gdef\gplbacktext{}%
  \gdef\gplfronttext{}%
  \makeatother
  \ifGPblacktext
    \def\colorrgb#1{}%
    \def\colorgray#1{}%
  \else
    \ifGPcolor
      \def\colorrgb#1{\color[rgb]{#1}}%
      \def\colorgray#1{\color[gray]{#1}}%
      \expandafter\def\csname LTw\endcsname{\color{white}}%
      \expandafter\def\csname LTb\endcsname{\color{black}}%
      \expandafter\def\csname LTa\endcsname{\color{black}}%
      \expandafter\def\csname LT0\endcsname{\color[rgb]{1,0,0}}%
      \expandafter\def\csname LT1\endcsname{\color[rgb]{0,1,0}}%
      \expandafter\def\csname LT2\endcsname{\color[rgb]{0,0,1}}%
      \expandafter\def\csname LT3\endcsname{\color[rgb]{1,0,1}}%
      \expandafter\def\csname LT4\endcsname{\color[rgb]{0,1,1}}%
      \expandafter\def\csname LT5\endcsname{\color[rgb]{1,1,0}}%
      \expandafter\def\csname LT6\endcsname{\color[rgb]{0,0,0}}%
      \expandafter\def\csname LT7\endcsname{\color[rgb]{1,0.3,0}}%
      \expandafter\def\csname LT8\endcsname{\color[rgb]{0.5,0.5,0.5}}%
    \else
      \def\colorrgb#1{\color{black}}%
      \def\colorgray#1{\color[gray]{#1}}%
      \expandafter\def\csname LTw\endcsname{\color{white}}%
      \expandafter\def\csname LTb\endcsname{\color{black}}%
      \expandafter\def\csname LTa\endcsname{\color{black}}%
      \expandafter\def\csname LT0\endcsname{\color{black}}%
      \expandafter\def\csname LT1\endcsname{\color{black}}%
      \expandafter\def\csname LT2\endcsname{\color{black}}%
      \expandafter\def\csname LT3\endcsname{\color{black}}%
      \expandafter\def\csname LT4\endcsname{\color{black}}%
      \expandafter\def\csname LT5\endcsname{\color{black}}%
      \expandafter\def\csname LT6\endcsname{\color{black}}%
      \expandafter\def\csname LT7\endcsname{\color{black}}%
      \expandafter\def\csname LT8\endcsname{\color{black}}%
    \fi
  \fi
  \setlength{\unitlength}{0.0500bp}%
  \begin{picture}(7200.00,5040.00)%
    \gplgaddtomacro\gplbacktext{%
      \csname LTb\endcsname%
      \put(946,704){\makebox(0,0)[r]{\strut{} 0}}%
      \put(946,1111){\makebox(0,0)[r]{\strut{} 0.5}}%
      \put(946,1518){\makebox(0,0)[r]{\strut{} 1}}%
      \put(946,1925){\makebox(0,0)[r]{\strut{} 1.5}}%
      \put(946,2332){\makebox(0,0)[r]{\strut{} 2}}%
      \put(946,2740){\makebox(0,0)[r]{\strut{} 2.5}}%
      \put(946,3147){\makebox(0,0)[r]{\strut{} 3}}%
      \put(946,3554){\makebox(0,0)[r]{\strut{} 3.5}}%
      \put(946,3961){\makebox(0,0)[r]{\strut{} 4}}%
      \put(946,4368){\makebox(0,0)[r]{\strut{} 4.5}}%
      \put(946,4775){\makebox(0,0)[r]{\strut{} 5}}%
      \put(1078,484){\makebox(0,0){\strut{}-2}}%
      \put(2032,484){\makebox(0,0){\strut{} 0}}%
      \put(2986,484){\makebox(0,0){\strut{} 2}}%
      \put(3940,484){\makebox(0,0){\strut{} 4}}%
      \put(4895,484){\makebox(0,0){\strut{} 6}}%
      \put(5849,484){\makebox(0,0){\strut{} 8}}%
      \put(6803,484){\makebox(0,0){\strut{} 10}}%
      \put(176,2739){{\makebox(0,0){\strut{}$\sigma$}}}%
      \put(3940,154){\makebox(0,0){\strut{}$\mu$}}%
    }%
    \gplgaddtomacro\gplfronttext{%
      \csname LTb\endcsname%
      \put(5816,4602){\makebox(0,0)[r]{\strut{}GIGO}}%
      \csname LTb\endcsname%
      \put(5816,4382){\makebox(0,0)[r]{\strut{}CMA}}%
      \csname LTb\endcsname%
      \put(5816,4162){\makebox(0,0)[r]{\strut{}xNES}}%
    }%
    \gplbacktext
    \put(0,0){\includegraphics{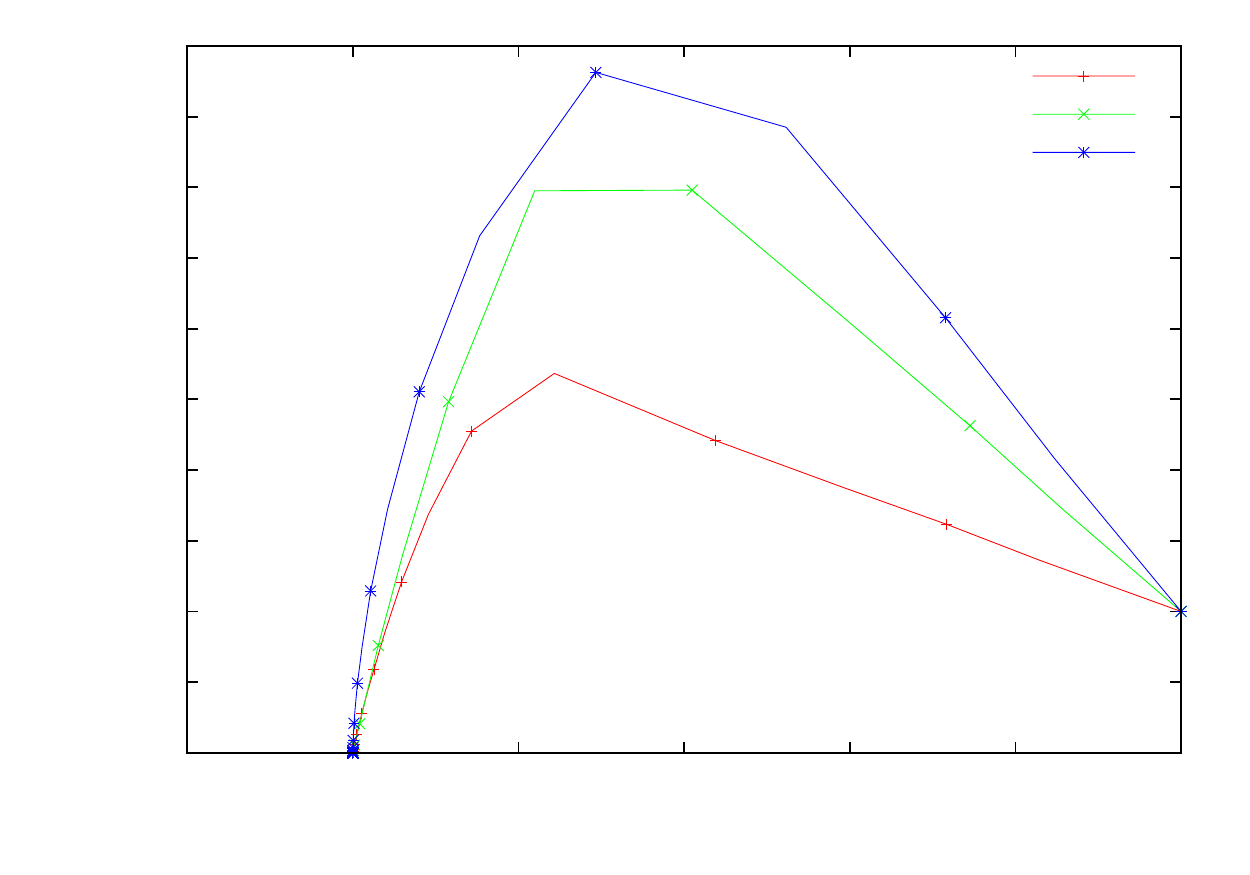}}%
    \gplfronttext
  \end{picture}%
\endgroup

%% file: Graphe1tex.tex
\begingroup
  \makeatletter
  \providecommand\color[2][]{%
    \GenericError{(gnuplot) \space\space\space\@spaces}{%
      Package color not loaded in conjunction with
      terminal option `colourtext'%
    }{See the gnuplot documentation for explanation.%
    }{Either use 'blacktext' in gnuplot or load the package
      color.sty in LaTeX.}%
    \renewcommand\color[2][]{}%
  }%
  \providecommand\includegraphics[2][]{%
    \GenericError{(gnuplot) \space\space\space\@spaces}{%
      Package graphicx or graphics not loaded%
    }{See the gnuplot documentation for explanation.%
    }{The gnuplot epslatex terminal needs graphicx.sty or graphics.sty.}%
    \renewcommand\includegraphics[2][]{}%
  }%
  \providecommand\rotatebox[2]{#2}%
  \@ifundefined{ifGPcolor}{%
    \newif\ifGPcolor
    \GPcolortrue
  }{}%
  \@ifundefined{ifGPblacktext}{%
    \newif\ifGPblacktext
    \GPblacktexttrue
  }{}%
  \let\gplgaddtomacro\g@addto@macro
  \gdef\gplbacktext{}%
  \gdef\gplfronttext{}%
  \makeatother
  \ifGPblacktext
    \def\colorrgb#1{}%
    \def\colorgray#1{}%
  \else
    \ifGPcolor
      \def\colorrgb#1{\color[rgb]{#1}}%
      \def\colorgray#1{\color[gray]{#1}}%
      \expandafter\def\csname LTw\endcsname{\color{white}}%
      \expandafter\def\csname LTb\endcsname{\color{black}}%
      \expandafter\def\csname LTa\endcsname{\color{black}}%
      \expandafter\def\csname LT0\endcsname{\color[rgb]{1,0,0}}%
      \expandafter\def\csname LT1\endcsname{\color[rgb]{0,1,0}}%
      \expandafter\def\csname LT2\endcsname{\color[rgb]{0,0,1}}%
      \expandafter\def\csname LT3\endcsname{\color[rgb]{1,0,1}}%
      \expandafter\def\csname LT4\endcsname{\color[rgb]{0,1,1}}%
      \expandafter\def\csname LT5\endcsname{\color[rgb]{1,1,0}}%
      \expandafter\def\csname LT6\endcsname{\color[rgb]{0,0,0}}%
      \expandafter\def\csname LT7\endcsname{\color[rgb]{1,0.3,0}}%
      \expandafter\def\csname LT8\endcsname{\color[rgb]{0.5,0.5,0.5}}%
    \else
      \def\colorrgb#1{\color{black}}%
      \def\colorgray#1{\color[gray]{#1}}%
      \expandafter\def\csname LTw\endcsname{\color{white}}%
      \expandafter\def\csname LTb\endcsname{\color{black}}%
      \expandafter\def\csname LTa\endcsname{\color{black}}%
      \expandafter\def\csname LT0\endcsname{\color{black}}%
      \expandafter\def\csname LT1\endcsname{\color{black}}%
      \expandafter\def\csname LT2\endcsname{\color{black}}%
      \expandafter\def\csname LT3\endcsname{\color{black}}%
      \expandafter\def\csname LT4\endcsname{\color{black}}%
      \expandafter\def\csname LT5\endcsname{\color{black}}%
      \expandafter\def\csname LT6\endcsname{\color{black}}%
      \expandafter\def\csname LT7\endcsname{\color{black}}%
      \expandafter\def\csname LT8\endcsname{\color{black}}%
    \fi
  \fi
  \setlength{\unitlength}{0.0500bp}%
  \begin{picture}(7200.00,5040.00)%
    \gplgaddtomacro\gplbacktext{%
      \csname LTb\endcsname%
      \put(682,704){\makebox(0,0)[r]{\strut{} 0}}%
      \put(682,1213){\makebox(0,0)[r]{\strut{} 1}}%
      \put(682,1722){\makebox(0,0)[r]{\strut{} 2}}%
      \put(682,2231){\makebox(0,0)[r]{\strut{} 3}}%
      \put(682,2740){\makebox(0,0)[r]{\strut{} 4}}%
      \put(682,3248){\makebox(0,0)[r]{\strut{} 5}}%
      \put(682,3757){\makebox(0,0)[r]{\strut{} 6}}%
      \put(682,4266){\makebox(0,0)[r]{\strut{} 7}}%
      \put(682,4775){\makebox(0,0)[r]{\strut{} 8}}%
      \put(814,484){\makebox(0,0){\strut{}-2}}%
      \put(1812,484){\makebox(0,0){\strut{} 0}}%
      \put(2810,484){\makebox(0,0){\strut{} 2}}%
      \put(3808,484){\makebox(0,0){\strut{} 4}}%
      \put(4807,484){\makebox(0,0){\strut{} 6}}%
      \put(5805,484){\makebox(0,0){\strut{} 8}}%
      \put(6803,484){\makebox(0,0){\strut{} 10}}%
      \put(176,2739){{\makebox(0,0){\strut{}$\sigma$}}}%
      \put(3808,154){\makebox(0,0){\strut{}$\mu$}}%
    }%
    \gplgaddtomacro\gplfronttext{%
      \csname LTb\endcsname%
      \put(5816,4602){\makebox(0,0)[r]{\strut{}GIGO}}%
      \csname LTb\endcsname%
      \put(5816,4382){\makebox(0,0)[r]{\strut{}CMA}}%
      \csname LTb\endcsname%
      \put(5816,4162){\makebox(0,0)[r]{\strut{}xNES}}%
    }%
    \gplbacktext
    \put(0,0){\includegraphics{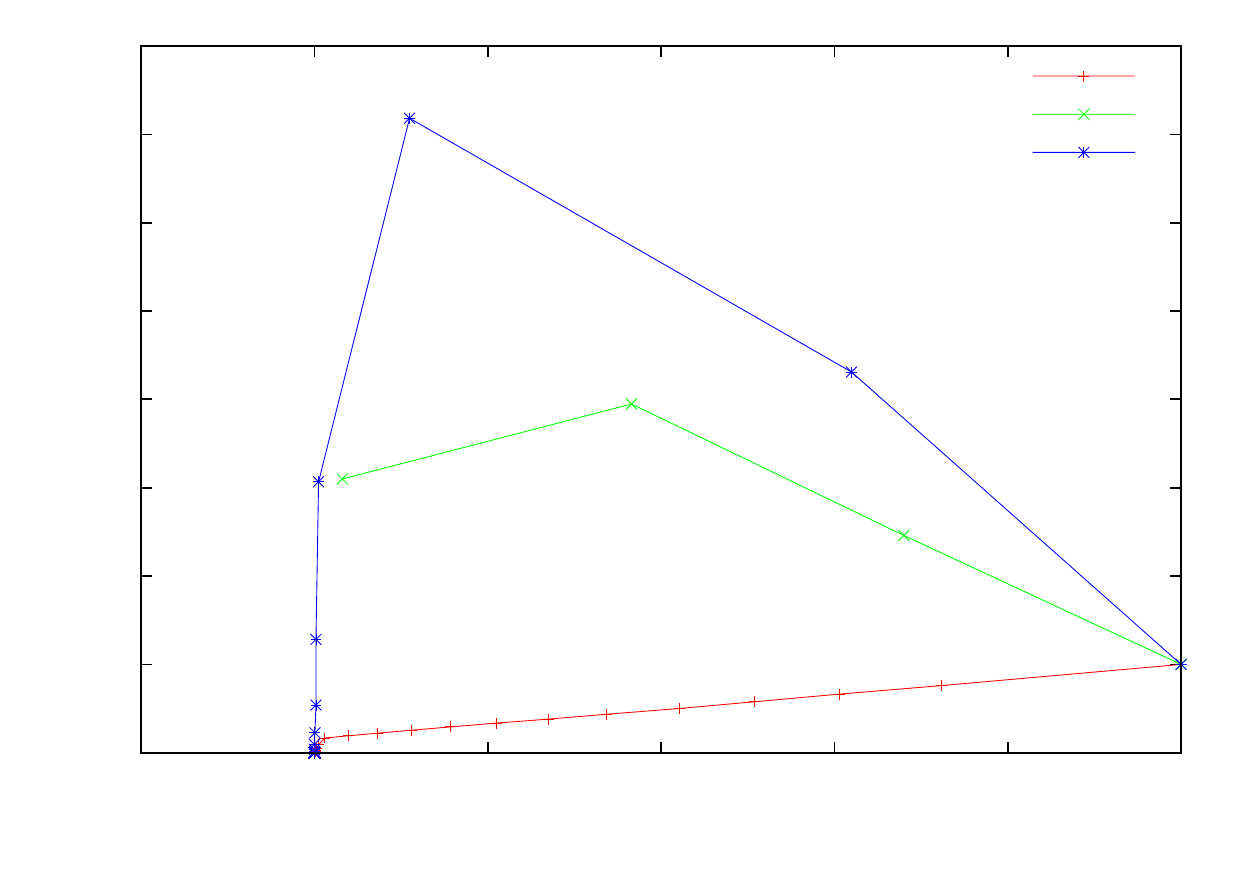}}%
    \gplfronttext
  \end{picture}%
\endgroup

%% file: Graphe15tex.tex
\begingroup
  \makeatletter
  \providecommand\color[2][]{%
    \GenericError{(gnuplot) \space\space\space\@spaces}{%
      Package color not loaded in conjunction with
      terminal option `colourtext'%
    }{See the gnuplot documentation for explanation.%
    }{Either use 'blacktext' in gnuplot or load the package
      color.sty in LaTeX.}%
    \renewcommand\color[2][]{}%
  }%
  \providecommand\includegraphics[2][]{%
    \GenericError{(gnuplot) \space\space\space\@spaces}{%
      Package graphicx or graphics not loaded%
    }{See the gnuplot documentation for explanation.%
    }{The gnuplot epslatex terminal needs graphicx.sty or graphics.sty.}%
    \renewcommand\includegraphics[2][]{}%
  }%
  \providecommand\rotatebox[2]{#2}%
  \@ifundefined{ifGPcolor}{%
    \newif\ifGPcolor
    \GPcolortrue
  }{}%
  \@ifundefined{ifGPblacktext}{%
    \newif\ifGPblacktext
    \GPblacktexttrue
  }{}%
  \let\gplgaddtomacro\g@addto@macro
  \gdef\gplbacktext{}%
  \gdef\gplfronttext{}%
  \makeatother
  \ifGPblacktext
    \def\colorrgb#1{}%
    \def\colorgray#1{}%
  \else
    \ifGPcolor
      \def\colorrgb#1{\color[rgb]{#1}}%
      \def\colorgray#1{\color[gray]{#1}}%
      \expandafter\def\csname LTw\endcsname{\color{white}}%
      \expandafter\def\csname LTb\endcsname{\color{black}}%
      \expandafter\def\csname LTa\endcsname{\color{black}}%
      \expandafter\def\csname LT0\endcsname{\color[rgb]{1,0,0}}%
      \expandafter\def\csname LT1\endcsname{\color[rgb]{0,1,0}}%
      \expandafter\def\csname LT2\endcsname{\color[rgb]{0,0,1}}%
      \expandafter\def\csname LT3\endcsname{\color[rgb]{1,0,1}}%
      \expandafter\def\csname LT4\endcsname{\color[rgb]{0,1,1}}%
      \expandafter\def\csname LT5\endcsname{\color[rgb]{1,1,0}}%
      \expandafter\def\csname LT6\endcsname{\color[rgb]{0,0,0}}%
      \expandafter\def\csname LT7\endcsname{\color[rgb]{1,0.3,0}}%
      \expandafter\def\csname LT8\endcsname{\color[rgb]{0.5,0.5,0.5}}%
    \else
      \def\colorrgb#1{\color{black}}%
      \def\colorgray#1{\color[gray]{#1}}%
      \expandafter\def\csname LTw\endcsname{\color{white}}%
      \expandafter\def\csname LTb\endcsname{\color{black}}%
      \expandafter\def\csname LTa\endcsname{\color{black}}%
      \expandafter\def\csname LT0\endcsname{\color{black}}%
      \expandafter\def\csname LT1\endcsname{\color{black}}%
      \expandafter\def\csname LT2\endcsname{\color{black}}%
      \expandafter\def\csname LT3\endcsname{\color{black}}%
      \expandafter\def\csname LT4\endcsname{\color{black}}%
      \expandafter\def\csname LT5\endcsname{\color{black}}%
      \expandafter\def\csname LT6\endcsname{\color{black}}%
      \expandafter\def\csname LT7\endcsname{\color{black}}%
      \expandafter\def\csname LT8\endcsname{\color{black}}%
    \fi
  \fi
  \setlength{\unitlength}{0.0500bp}%
  \begin{picture}(7200.00,5040.00)%
    \gplgaddtomacro\gplbacktext{%
      \csname LTb\endcsname%
      \put(682,704){\makebox(0,0)[r]{\strut{} 0}}%
      \put(682,1156){\makebox(0,0)[r]{\strut{} 1}}%
      \put(682,1609){\makebox(0,0)[r]{\strut{} 2}}%
      \put(682,2061){\makebox(0,0)[r]{\strut{} 3}}%
      \put(682,2513){\makebox(0,0)[r]{\strut{} 4}}%
      \put(682,2966){\makebox(0,0)[r]{\strut{} 5}}%
      \put(682,3418){\makebox(0,0)[r]{\strut{} 6}}%
      \put(682,3870){\makebox(0,0)[r]{\strut{} 7}}%
      \put(682,4323){\makebox(0,0)[r]{\strut{} 8}}%
      \put(682,4775){\makebox(0,0)[r]{\strut{} 9}}%
      \put(814,484){\makebox(0,0){\strut{}-2}}%
      \put(1812,484){\makebox(0,0){\strut{} 0}}%
      \put(2810,484){\makebox(0,0){\strut{} 2}}%
      \put(3808,484){\makebox(0,0){\strut{} 4}}%
      \put(4807,484){\makebox(0,0){\strut{} 6}}%
      \put(5805,484){\makebox(0,0){\strut{} 8}}%
      \put(6803,484){\makebox(0,0){\strut{} 10}}%
      \put(176,2739){{\makebox(0,0){\strut{}$\sigma$}}}%
      \put(3808,154){\makebox(0,0){\strut{}$\mu$}}%
    }%
    \gplgaddtomacro\gplfronttext{%
      \csname LTb\endcsname%
      \put(5816,4602){\makebox(0,0)[r]{\strut{}GIGO}}%
      \csname LTb\endcsname%
      \put(5816,4382){\makebox(0,0)[r]{\strut{}CMA}}%
      \csname LTb\endcsname%
      \put(5816,4162){\makebox(0,0)[r]{\strut{}xNES}}%
    }%
    \gplbacktext
    \put(0,0){\includegraphics{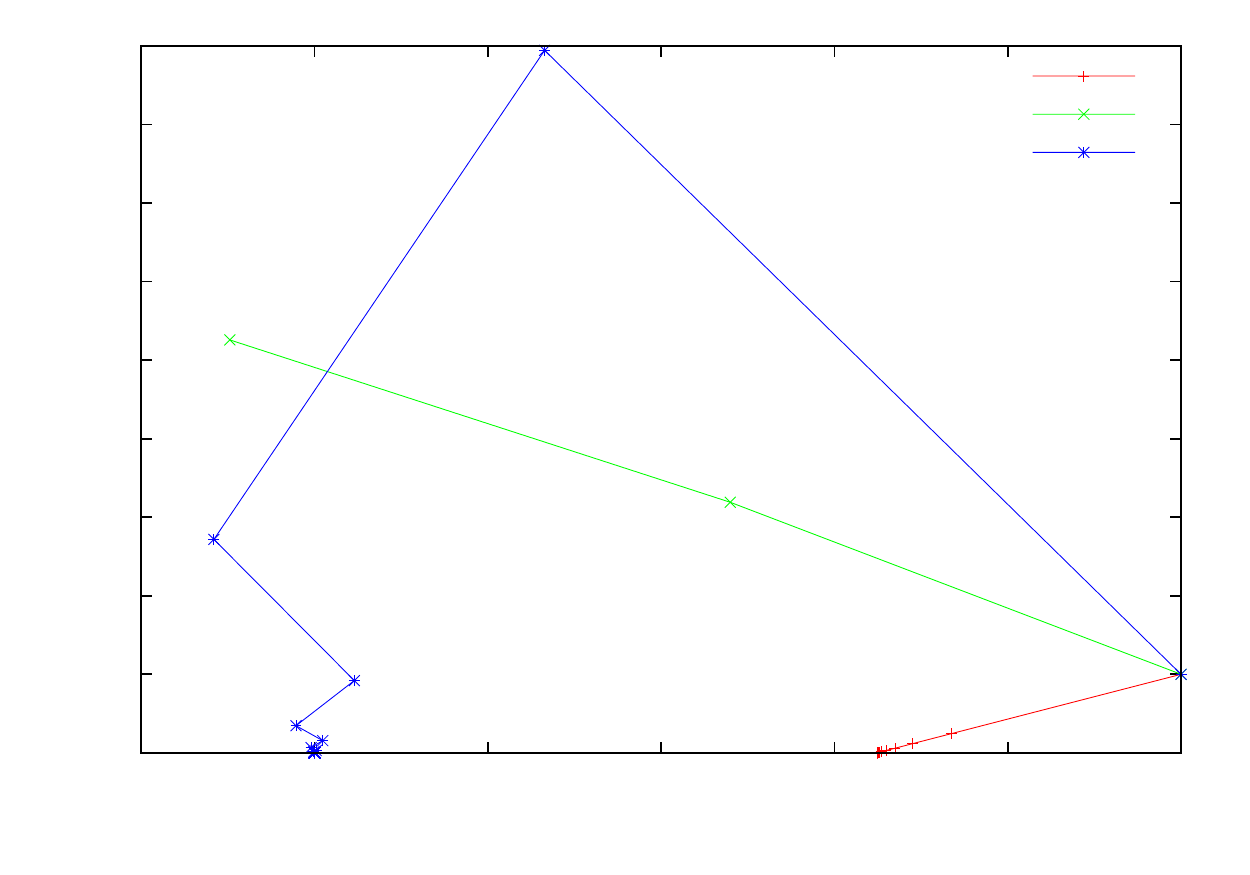}}%
    \gplfronttext
  \end{picture}%
\endgroup

%% file: Lin001tex.tex
\begingroup
  \makeatletter
  \providecommand\color[2][]{%
    \GenericError{(gnuplot) \space\space\space\@spaces}{%
      Package color not loaded in conjunction with
      terminal option `colourtext'%
    }{See the gnuplot documentation for explanation.%
    }{Either use 'blacktext' in gnuplot or load the package
      color.sty in LaTeX.}%
    \renewcommand\color[2][]{}%
  }%
  \providecommand\includegraphics[2][]{%
    \GenericError{(gnuplot) \space\space\space\@spaces}{%
      Package graphicx or graphics not loaded%
    }{See the gnuplot documentation for explanation.%
    }{The gnuplot epslatex terminal needs graphicx.sty or graphics.sty.}%
    \renewcommand\includegraphics[2][]{}%
  }%
  \providecommand\rotatebox[2]{#2}%
  \@ifundefined{ifGPcolor}{%
    \newif\ifGPcolor
    \GPcolortrue
  }{}%
  \@ifundefined{ifGPblacktext}{%
    \newif\ifGPblacktext
    \GPblacktexttrue
  }{}%
  \let\gplgaddtomacro\g@addto@macro
  \gdef\gplbacktext{}%
  \gdef\gplfronttext{}%
  \makeatother
  \ifGPblacktext
    \def\colorrgb#1{}%
    \def\colorgray#1{}%
  \else
    \ifGPcolor
      \def\colorrgb#1{\color[rgb]{#1}}%
      \def\colorgray#1{\color[gray]{#1}}%
      \expandafter\def\csname LTw\endcsname{\color{white}}%
      \expandafter\def\csname LTb\endcsname{\color{black}}%
      \expandafter\def\csname LTa\endcsname{\color{black}}%
      \expandafter\def\csname LT0\endcsname{\color[rgb]{1,0,0}}%
      \expandafter\def\csname LT1\endcsname{\color[rgb]{0,1,0}}%
      \expandafter\def\csname LT2\endcsname{\color[rgb]{0,0,1}}%
      \expandafter\def\csname LT3\endcsname{\color[rgb]{1,0,1}}%
      \expandafter\def\csname LT4\endcsname{\color[rgb]{0,1,1}}%
      \expandafter\def\csname LT5\endcsname{\color[rgb]{1,1,0}}%
      \expandafter\def\csname LT6\endcsname{\color[rgb]{0,0,0}}%
      \expandafter\def\csname LT7\endcsname{\color[rgb]{1,0.3,0}}%
      \expandafter\def\csname LT8\endcsname{\color[rgb]{0.5,0.5,0.5}}%
    \else
      \def\colorrgb#1{\color{black}}%
      \def\colorgray#1{\color[gray]{#1}}%
      \expandafter\def\csname LTw\endcsname{\color{white}}%
      \expandafter\def\csname LTb\endcsname{\color{black}}%
      \expandafter\def\csname LTa\endcsname{\color{black}}%
      \expandafter\def\csname LT0\endcsname{\color{black}}%
      \expandafter\def\csname LT1\endcsname{\color{black}}%
      \expandafter\def\csname LT2\endcsname{\color{black}}%
      \expandafter\def\csname LT3\endcsname{\color{black}}%
      \expandafter\def\csname LT4\endcsname{\color{black}}%
      \expandafter\def\csname LT5\endcsname{\color{black}}%
      \expandafter\def\csname LT6\endcsname{\color{black}}%
      \expandafter\def\csname LT7\endcsname{\color{black}}%
      \expandafter\def\csname LT8\endcsname{\color{black}}%
    \fi
  \fi
  \setlength{\unitlength}{0.0500bp}%
  \begin{picture}(7200.00,5040.00)%
    \gplgaddtomacro\gplbacktext{%
      \csname LTb\endcsname%
      \put(1342,704){\makebox(0,0)[r]{\strut{} 1}}%
      \put(1342,1286){\makebox(0,0)[r]{\strut{} 100}}%
      \put(1342,1867){\makebox(0,0)[r]{\strut{} 10000}}%
      \put(1342,2449){\makebox(0,0)[r]{\strut{} 1e+006}}%
      \put(1342,3030){\makebox(0,0)[r]{\strut{} 1e+008}}%
      \put(1342,3612){\makebox(0,0)[r]{\strut{} 1e+010}}%
      \put(1342,4193){\makebox(0,0)[r]{\strut{} 1e+012}}%
      \put(1342,4775){\makebox(0,0)[r]{\strut{} 1e+014}}%
      \put(1474,484){\makebox(0,0){\strut{} 1}}%
      \put(2235,484){\makebox(0,0){\strut{} 100}}%
      \put(2997,484){\makebox(0,0){\strut{} 10000}}%
      \put(3758,484){\makebox(0,0){\strut{} 1e+006}}%
      \put(4519,484){\makebox(0,0){\strut{} 1e+008}}%
      \put(5280,484){\makebox(0,0){\strut{} 1e+010}}%
      \put(6042,484){\makebox(0,0){\strut{} 1e+012}}%
      \put(6803,484){\makebox(0,0){\strut{} 1e+014}}%
      \put(176,2739){{\makebox(0,0){\strut{}$\sigma$}}}%
      \put(4138,154){\makebox(0,0){\strut{}$\mu$}}%
    }%
    \gplgaddtomacro\gplfronttext{%
      \csname LTb\endcsname%
      \put(3190,4602){\makebox(0,0)[r]{\strut{}GIGO}}%
      \csname LTb\endcsname%
      \put(3190,4382){\makebox(0,0)[r]{\strut{}CMA}}%
      \csname LTb\endcsname%
      \put(3190,4162){\makebox(0,0)[r]{\strut{}xNES}}%
    }%
    \gplbacktext
    \put(0,0){\includegraphics{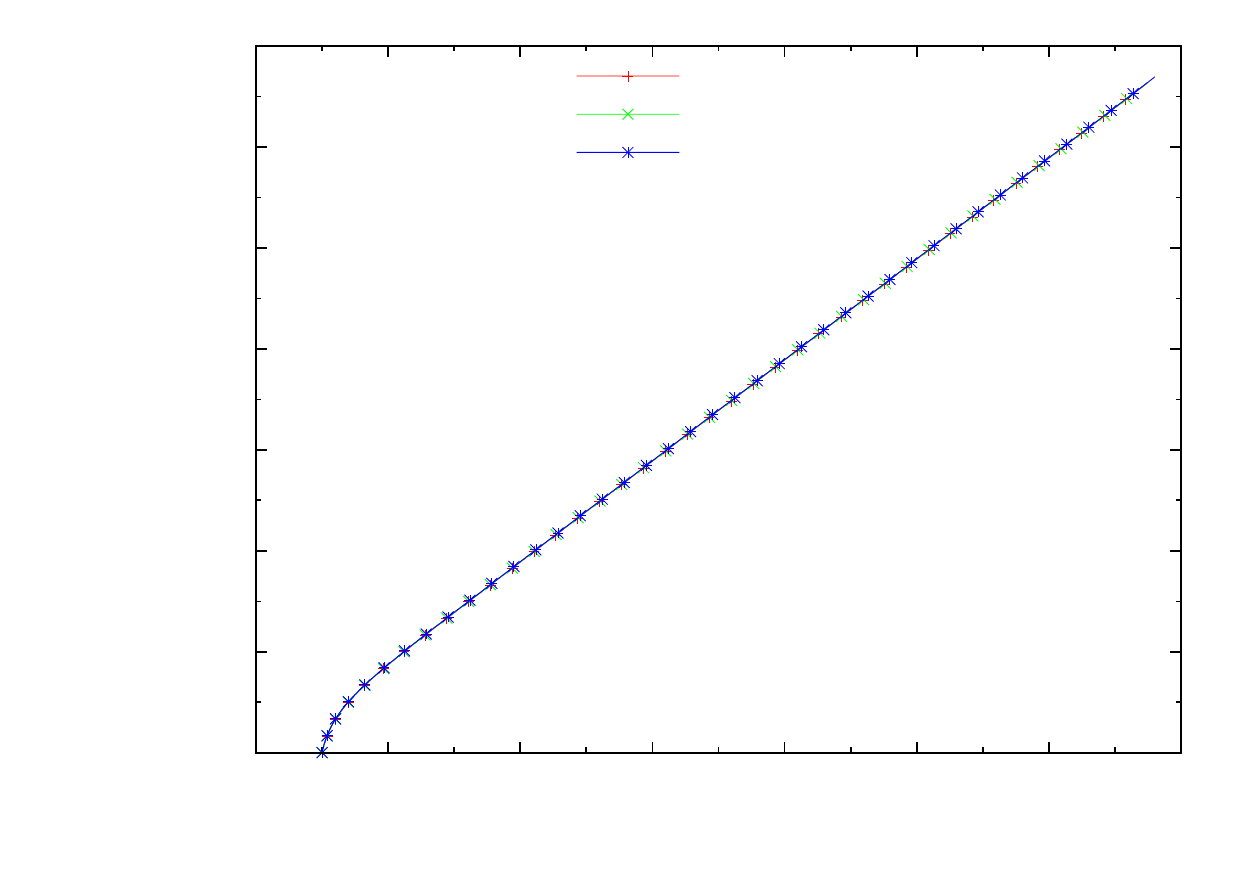}}%
    \gplfronttext
  \end{picture}%
\endgroup

%% file: Lin01tex.tex
\begingroup
  \makeatletter
  \providecommand\color[2][]{%
    \GenericError{(gnuplot) \space\space\space\@spaces}{%
      Package color not loaded in conjunction with
      terminal option `colourtext'%
    }{See the gnuplot documentation for explanation.%
    }{Either use 'blacktext' in gnuplot or load the package
      color.sty in LaTeX.}%
    \renewcommand\color[2][]{}%
  }%
  \providecommand\includegraphics[2][]{%
    \GenericError{(gnuplot) \space\space\space\@spaces}{%
      Package graphicx or graphics not loaded%
    }{See the gnuplot documentation for explanation.%
    }{The gnuplot epslatex terminal needs graphicx.sty or graphics.sty.}%
    \renewcommand\includegraphics[2][]{}%
  }%
  \providecommand\rotatebox[2]{#2}%
  \@ifundefined{ifGPcolor}{%
    \newif\ifGPcolor
    \GPcolortrue
  }{}%
  \@ifundefined{ifGPblacktext}{%
    \newif\ifGPblacktext
    \GPblacktexttrue
  }{}%
  \let\gplgaddtomacro\g@addto@macro
  \gdef\gplbacktext{}%
  \gdef\gplfronttext{}%
  \makeatother
  \ifGPblacktext
    \def\colorrgb#1{}%
    \def\colorgray#1{}%
  \else
    \ifGPcolor
      \def\colorrgb#1{\color[rgb]{#1}}%
      \def\colorgray#1{\color[gray]{#1}}%
      \expandafter\def\csname LTw\endcsname{\color{white}}%
      \expandafter\def\csname LTb\endcsname{\color{black}}%
      \expandafter\def\csname LTa\endcsname{\color{black}}%
      \expandafter\def\csname LT0\endcsname{\color[rgb]{1,0,0}}%
      \expandafter\def\csname LT1\endcsname{\color[rgb]{0,1,0}}%
      \expandafter\def\csname LT2\endcsname{\color[rgb]{0,0,1}}%
      \expandafter\def\csname LT3\endcsname{\color[rgb]{1,0,1}}%
      \expandafter\def\csname LT4\endcsname{\color[rgb]{0,1,1}}%
      \expandafter\def\csname LT5\endcsname{\color[rgb]{1,1,0}}%
      \expandafter\def\csname LT6\endcsname{\color[rgb]{0,0,0}}%
      \expandafter\def\csname LT7\endcsname{\color[rgb]{1,0.3,0}}%
      \expandafter\def\csname LT8\endcsname{\color[rgb]{0.5,0.5,0.5}}%
    \else
      \def\colorrgb#1{\color{black}}%
      \def\colorgray#1{\color[gray]{#1}}%
      \expandafter\def\csname LTw\endcsname{\color{white}}%
      \expandafter\def\csname LTb\endcsname{\color{black}}%
      \expandafter\def\csname LTa\endcsname{\color{black}}%
      \expandafter\def\csname LT0\endcsname{\color{black}}%
      \expandafter\def\csname LT1\endcsname{\color{black}}%
      \expandafter\def\csname LT2\endcsname{\color{black}}%
      \expandafter\def\csname LT3\endcsname{\color{black}}%
      \expandafter\def\csname LT4\endcsname{\color{black}}%
      \expandafter\def\csname LT5\endcsname{\color{black}}%
      \expandafter\def\csname LT6\endcsname{\color{black}}%
      \expandafter\def\csname LT7\endcsname{\color{black}}%
      \expandafter\def\csname LT8\endcsname{\color{black}}%
    \fi
  \fi
  \setlength{\unitlength}{0.0500bp}%
  \begin{picture}(7200.00,5040.00)%
    \gplgaddtomacro\gplbacktext{%
      \csname LTb\endcsname%
      \put(1342,704){\makebox(0,0)[r]{\strut{} 1}}%
      \put(1342,1286){\makebox(0,0)[r]{\strut{} 100}}%
      \put(1342,1867){\makebox(0,0)[r]{\strut{} 10000}}%
      \put(1342,2449){\makebox(0,0)[r]{\strut{} 1e+006}}%
      \put(1342,3030){\makebox(0,0)[r]{\strut{} 1e+008}}%
      \put(1342,3612){\makebox(0,0)[r]{\strut{} 1e+010}}%
      \put(1342,4193){\makebox(0,0)[r]{\strut{} 1e+012}}%
      \put(1342,4775){\makebox(0,0)[r]{\strut{} 1e+014}}%
      \put(1474,484){\makebox(0,0){\strut{} 1}}%
      \put(2235,484){\makebox(0,0){\strut{} 100}}%
      \put(2997,484){\makebox(0,0){\strut{} 10000}}%
      \put(3758,484){\makebox(0,0){\strut{} 1e+006}}%
      \put(4519,484){\makebox(0,0){\strut{} 1e+008}}%
      \put(5280,484){\makebox(0,0){\strut{} 1e+010}}%
      \put(6042,484){\makebox(0,0){\strut{} 1e+012}}%
      \put(6803,484){\makebox(0,0){\strut{} 1e+014}}%
      \put(176,2739){{\makebox(0,0){\strut{}$\sigma $}}}%
      \put(4138,154){\makebox(0,0){\strut{}$\mu $}}%
    }%
    \gplgaddtomacro\gplfronttext{%
      \csname LTb\endcsname%
      \put(3058,4602){\makebox(0,0)[r]{\strut{}GIGO}}%
      \csname LTb\endcsname%
      \put(3058,4382){\makebox(0,0)[r]{\strut{}CMA}}%
      \csname LTb\endcsname%
      \put(3058,4162){\makebox(0,0)[r]{\strut{}xNES}}%
    }%
    \gplbacktext
    \put(0,0){\includegraphics{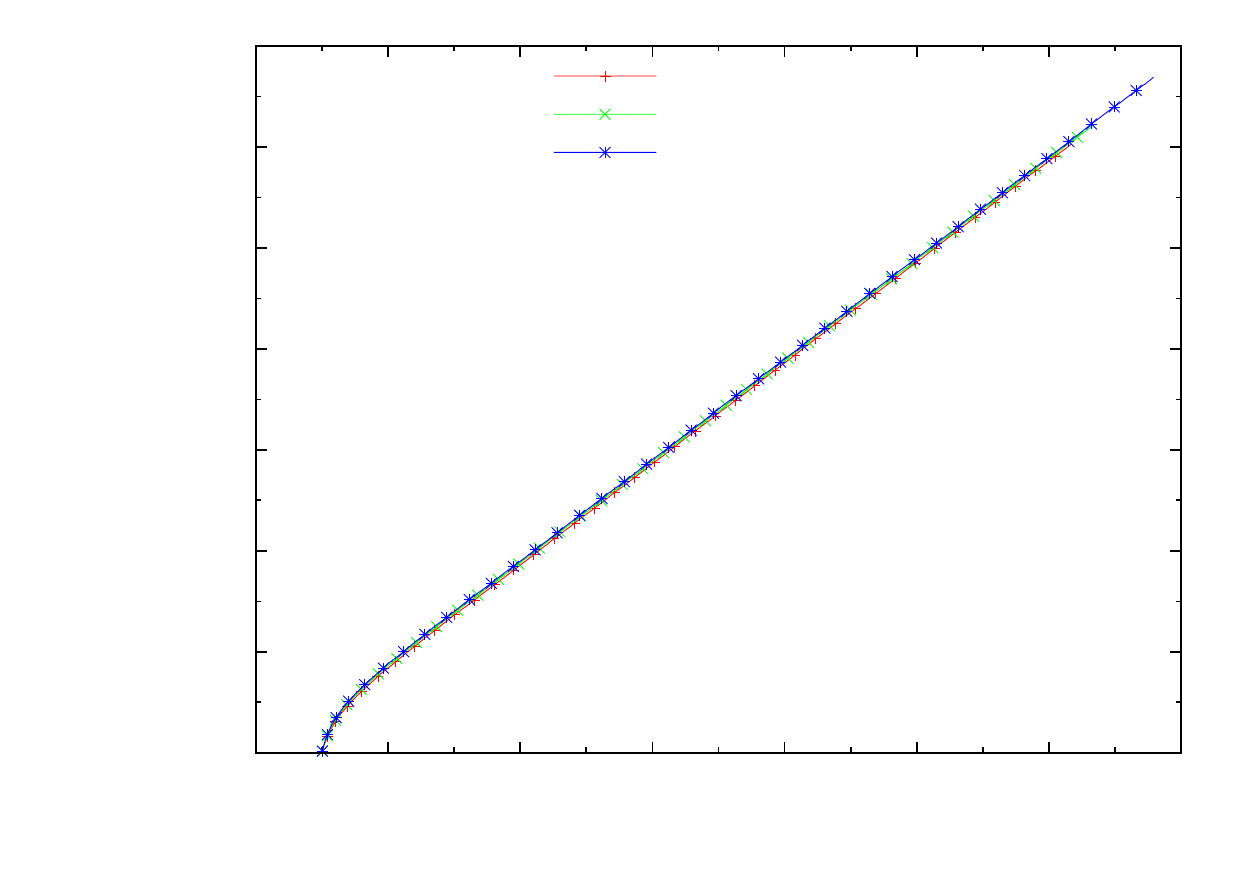}}%
    \gplfronttext
  \end{picture}%
\endgroup

%% file: Lin1tex.tex
\begingroup
  \makeatletter
  \providecommand\color[2][]{%
    \GenericError{(gnuplot) \space\space\space\@spaces}{%
      Package color not loaded in conjunction with
      terminal option `colourtext'%
    }{See the gnuplot documentation for explanation.%
    }{Either use 'blacktext' in gnuplot or load the package
      color.sty in LaTeX.}%
    \renewcommand\color[2][]{}%
  }%
  \providecommand\includegraphics[2][]{%
    \GenericError{(gnuplot) \space\space\space\@spaces}{%
      Package graphicx or graphics not loaded%
    }{See the gnuplot documentation for explanation.%
    }{The gnuplot epslatex terminal needs graphicx.sty or graphics.sty.}%
    \renewcommand\includegraphics[2][]{}%
  }%
  \providecommand\rotatebox[2]{#2}%
  \@ifundefined{ifGPcolor}{%
    \newif\ifGPcolor
    \GPcolortrue
  }{}%
  \@ifundefined{ifGPblacktext}{%
    \newif\ifGPblacktext
    \GPblacktexttrue
  }{}%
  \let\gplgaddtomacro\g@addto@macro
  \gdef\gplbacktext{}%
  \gdef\gplfronttext{}%
  \makeatother
  \ifGPblacktext
    \def\colorrgb#1{}%
    \def\colorgray#1{}%
  \else
    \ifGPcolor
      \def\colorrgb#1{\color[rgb]{#1}}%
      \def\colorgray#1{\color[gray]{#1}}%
      \expandafter\def\csname LTw\endcsname{\color{white}}%
      \expandafter\def\csname LTb\endcsname{\color{black}}%
      \expandafter\def\csname LTa\endcsname{\color{black}}%
      \expandafter\def\csname LT0\endcsname{\color[rgb]{1,0,0}}%
      \expandafter\def\csname LT1\endcsname{\color[rgb]{0,1,0}}%
      \expandafter\def\csname LT2\endcsname{\color[rgb]{0,0,1}}%
      \expandafter\def\csname LT3\endcsname{\color[rgb]{1,0,1}}%
      \expandafter\def\csname LT4\endcsname{\color[rgb]{0,1,1}}%
      \expandafter\def\csname LT5\endcsname{\color[rgb]{1,1,0}}%
      \expandafter\def\csname LT6\endcsname{\color[rgb]{0,0,0}}%
      \expandafter\def\csname LT7\endcsname{\color[rgb]{1,0.3,0}}%
      \expandafter\def\csname LT8\endcsname{\color[rgb]{0.5,0.5,0.5}}%
    \else
      \def\colorrgb#1{\color{black}}%
      \def\colorgray#1{\color[gray]{#1}}%
      \expandafter\def\csname LTw\endcsname{\color{white}}%
      \expandafter\def\csname LTb\endcsname{\color{black}}%
      \expandafter\def\csname LTa\endcsname{\color{black}}%
      \expandafter\def\csname LT0\endcsname{\color{black}}%
      \expandafter\def\csname LT1\endcsname{\color{black}}%
      \expandafter\def\csname LT2\endcsname{\color{black}}%
      \expandafter\def\csname LT3\endcsname{\color{black}}%
      \expandafter\def\csname LT4\endcsname{\color{black}}%
      \expandafter\def\csname LT5\endcsname{\color{black}}%
      \expandafter\def\csname LT6\endcsname{\color{black}}%
      \expandafter\def\csname LT7\endcsname{\color{black}}%
      \expandafter\def\csname LT8\endcsname{\color{black}}%
    \fi
  \fi
  \setlength{\unitlength}{0.0500bp}%
  \begin{picture}(7200.00,5040.00)%
    \gplgaddtomacro\gplbacktext{%
      \csname LTb\endcsname%
      \put(1342,704){\makebox(0,0)[r]{\strut{} 0.0001}}%
      \put(1342,1156){\makebox(0,0)[r]{\strut{} 0.01}}%
      \put(1342,1609){\makebox(0,0)[r]{\strut{} 1}}%
      \put(1342,2061){\makebox(0,0)[r]{\strut{} 100}}%
      \put(1342,2513){\makebox(0,0)[r]{\strut{} 10000}}%
      \put(1342,2966){\makebox(0,0)[r]{\strut{} 1e+006}}%
      \put(1342,3418){\makebox(0,0)[r]{\strut{} 1e+008}}%
      \put(1342,3870){\makebox(0,0)[r]{\strut{} 1e+010}}%
      \put(1342,4323){\makebox(0,0)[r]{\strut{} 1e+012}}%
      \put(1342,4775){\makebox(0,0)[r]{\strut{} 1e+014}}%
      \put(1474,484){\makebox(0,0){\strut{} 1}}%
      \put(2235,484){\makebox(0,0){\strut{} 100}}%
      \put(2997,484){\makebox(0,0){\strut{} 10000}}%
      \put(3758,484){\makebox(0,0){\strut{} 1e+006}}%
      \put(4519,484){\makebox(0,0){\strut{} 1e+008}}%
      \put(5280,484){\makebox(0,0){\strut{} 1e+010}}%
      \put(6042,484){\makebox(0,0){\strut{} 1e+012}}%
      \put(6803,484){\makebox(0,0){\strut{} 1e+014}}%
      \put(176,2739){{\makebox(0,0){\strut{}$\sigma $}}}%
      \put(4138,154){\makebox(0,0){\strut{}$\mu $}}%
    }%
    \gplgaddtomacro\gplfronttext{%
      \csname LTb\endcsname%
      \put(2926,4602){\makebox(0,0)[r]{\strut{}GIGO}}%
      \csname LTb\endcsname%
      \put(2926,4382){\makebox(0,0)[r]{\strut{}CMA}}%
      \csname LTb\endcsname%
      \put(2926,4162){\makebox(0,0)[r]{\strut{}xNES}}%
    }%
    \gplbacktext
    \put(0,0){\includegraphics{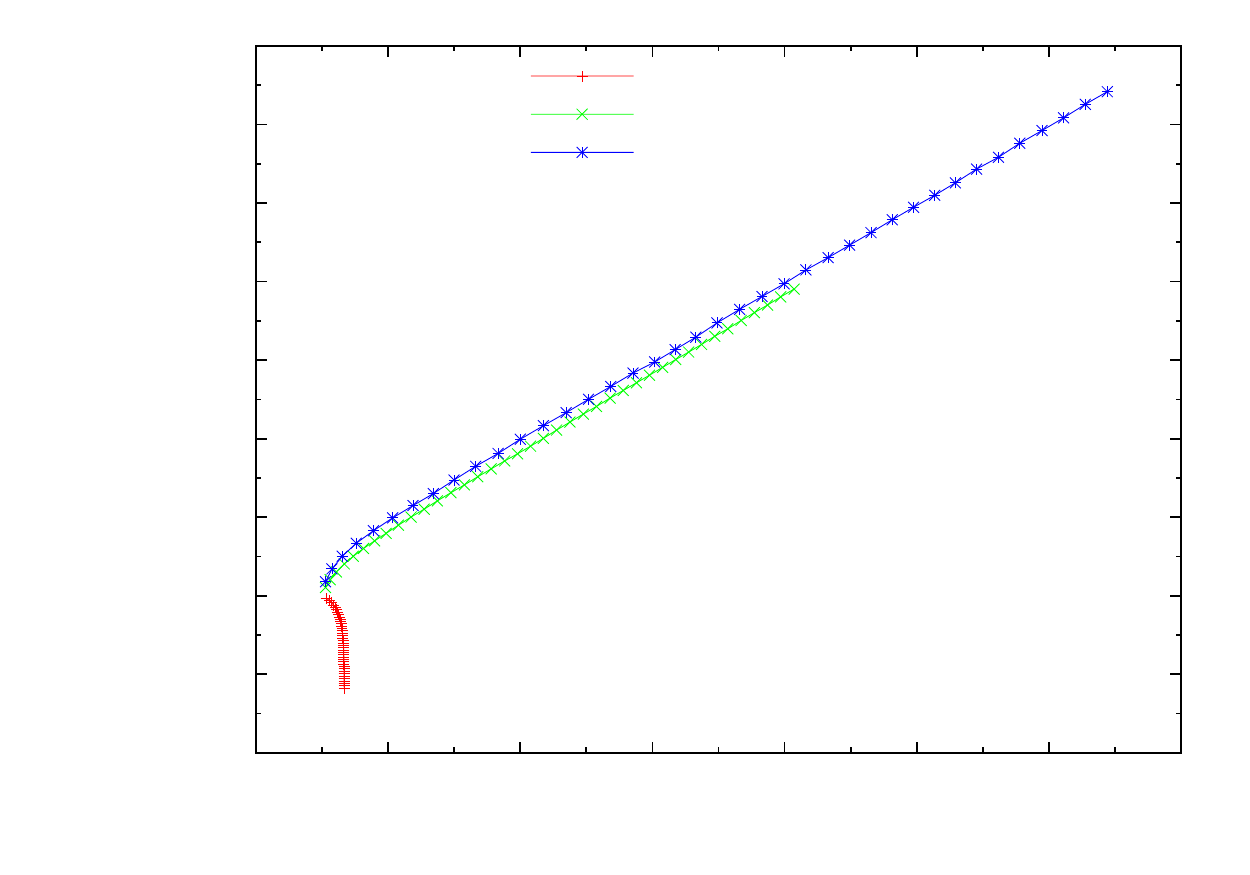}}%
    \gplfronttext
  \end{picture}%
\endgroup